\documentclass[reqno,12pt]{amsart}
\usepackage[T1]{fontenc}
\usepackage{enumerate}
\theoremstyle{plain}
\newtheorem{thm}{Theorem}
\newtheorem{expl}{Example}

\begin{document}
\title{ Some boundedness results for systems of two rational difference equations}
\author[G. Lugo]{G. Lugo}
\author[F.J. Palladino]{F.J. Palladino$^{1}$}
\date{September 9, 2009}

\begin{abstract}
\noindent We study $k^{th}$ order systems of two rational
difference equations
$$x_n=\frac{\alpha+\sum^{k}_{i=1}\beta_{i}x_{n-i} + \sum^{k}_{i=1}\gamma_{i}y_{n-i}}{A+\sum^{k}_{j=1}B_{j}x_{n-j} + \sum^{k}_{j=1}C_{j}y_{n-j}},\quad n\in\mathbb{N},$$
$$y_n=\frac{p+\sum^{k}_{i=1}\delta_{i}x_{n-i} + \sum^{k}_{i=1}\epsilon_{i}y_{n-i}}{q+\sum^{k}_{j=1}D_{j}x_{n-j} + \sum^{k}_{j=1}E_{j}y_{n-j}},\quad n\in\mathbb{N}.$$
In particular we assume non-negative parameters and non-negative initial conditions. We develop several approaches which allow us to extend well known boundedness results on the $k^{th}$ order rational difference equation to the setting of systems in certain cases.
\end{abstract}
\maketitle
\par\vspace{0.01 cm}
\hspace{0.5 cm}
Department of Mathematics, University of Rhode Island,
\newline\indent
\hspace{0.6 cm} Kingston, RI 02881-0816, USA;
\newline\indent
\par\vspace{0.05 cm}
\hspace{0.6 cm} $^1$Corresponding author. email: frank@math.uri.edu
\par\vspace{0.05 cm}
%\hspace{5 cm} Stony Brook University
%\par\vspace{0.05 cm}
%\hspace{5 cm} Mathematics Department
%\par\vspace{0.05 cm}
%\hspace{5 cm} Stony Brook, NY 11794-3651
%\newline\indent
\footnotetext
{\indent \hspace{.1 cm} {\bf AMS Subject
Classification: 39A10,39A11}}

\section{Introduction}

There has been recent interest in the study of systems of rational difference equations. The purpose of this article is to provide several analogues for successful techniques which were responsible for some well known boundedness results on the $k^{th}$ order rational difference equation.
The well known results we refer to are those presented in \cite{c1}. Several years following the appearance of \cite{c1} in the literature, \cite{fpinv} appeared in the literature. \cite{fpinv} primarily served to generalize some of the results presented in \cite{c1}.
This was done by showing that some analogous results held for certain recursive inequalities. It turns out that this generalization is very useful when studying the boundedness character of systems of rational difference equations.
Often times a complicated system will have a simpler difference inequality associated. We then use Theorem 1 of \cite{fpinv} to show that one of the sequences $\{x_{n}\}^{\infty}_{n=1}$
or $\{y_{n}\}^{\infty}_{n=1}$ is bounded. In this way Theorem 1 of \cite{fpinv} will provide the primary mechanism used in the proof of many of the results presented below.

\section{Using a Comparison}

Since we are presenting many results on systems of two rational difference equations which are analogues to the results of \cite{c1} and since we make heavy use of \cite{fpinv}, it should come as no surprise that we make use of similar notation. So we let $I_{\beta}=\{i\in\{1,\dots,k\}| \beta_{i} > 0\}$, $I_{\gamma}=\{i\in\{1,\dots,k\}| \gamma_{i} > 0\}$, $I_{\delta}=\{i\in\{1,\dots,k\}| \delta_{i} > 0\}$, $I_{\epsilon}=\{i\in\{1,\dots,k\}| \epsilon_{i} > 0\}$, $I_{B}=\{j\in\{1,\dots,k\}| B_{j} > 0\}$, $I_{C}=\{j\in\{1,\dots,k\}| C_{j} > 0\}$, $I_{D}=\{j\in\{1,\dots,k\}| D_{j} > 0\}$, and $I_{E}=\{j\in\{1,\dots,k\}| E_{j} > 0\}$. We also define $\min_{+}(\cdot,\cdot)$ to be the minimum of non-zero elements. This notation will allow us to shorten some arguments which we present later. In order for $\min_{+}(\cdot,\cdot)$ to be well defined at least one of the elements must be non-zero.
For the results in this section, we assume that there exists constants $M_{1},M_{2} > 0$ so that $M_{1}y_{n}\leq x_{n} \leq M_{2}y_{n}$ for all $n\in\mathbb{N}$. We use this comparison and Theorem 1 of \cite{fpinv} to prove that every solution is bounded for certain systems of rational difference equations.\newline

We begin with two theorems which generalize the standard iteration result to systems of two rational difference equations where there exists constants $M_{1},M_{2} > 0$ so that $M_{1}y_{n}\leq x_{n} \leq M_{2}y_{n}$ for all $n\in\mathbb{N}$.
Eventually these two theorems are subsumed by several slightly more general theorems at the end of the article, particularly Theorem $20$ and Theorem $21$.
In the following theorem, we use the results from \cite{fpinv} to iterate with respect to $x_{n}$.
\begin{thm}
Suppose that we have a $k^{th}$ order system of two rational
difference equations
$$x_n=\frac{\alpha+\sum^{k}_{i=1}\beta_{i}x_{n-i} + \sum^{k}_{i=1}\gamma_{i}y_{n-i}}{A+\sum^{k}_{j=1}B_{j}x_{n-j} + \sum^{k}_{j=1}C_{j}y_{n-j}},\quad n\in\mathbb{N},$$
$$y_n=\frac{p+\sum^{k}_{i=1}\delta_{i}x_{n-i} + \sum^{k}_{i=1}\epsilon_{i}y_{n-i}}{q+\sum^{k}_{j=1}D_{j}x_{n-j} + \sum^{k}_{j=1}E_{j}y_{n-j}},\quad n\in\mathbb{N},$$
with non-negative parameters and non-negative initial conditions.
Further assume that there exists constants $M_{1},M_{2} > 0$ so that $M_{1}y_{n}\leq x_{n} \leq M_{2}y_{n}$ for all $n\in\mathbb{N}$. Also suppose that $A>0$ and there exists a positive integer $\eta$, such that for every sequence $\{c_{m}\} ^{\infty}_{m=1} $ with $c_m\in I_{\beta}\cup I_{\gamma}$ for $m=1,2, \dots $ there exists positive integers,  $
N_{1}
$,$
N_{2} \leq \eta
$ $ $, such that $ \sum^{N_{2}}_{m=N_{1}} c_{m} \in I_{B}\cup I_{C}$.\newline
Then there exists $M > 0$ and $N\in\mathbb{N}$ so that given any non-negative initial conditions, we have $x_{n},y_{n}\leq M$ for all $n>N$.
\end{thm}

\begin{proof}
Since $M_{1}y_{n}\leq x_{n} \leq M_{2}y_{n}$ we have
$$x_{n}=\frac{\alpha+\sum^{k}_{i=1}\beta_{i}x_{n-i} + \sum^{k}_{i=1}\gamma_{i}y_{n-i}}{A+\sum^{k}_{j=1}B_{j}x_{n-j} + \sum^{k}_{j=1}C_{j}y_{n-j}} \leq \frac{\alpha+\sum^{k}_{i=1}\beta_{i}x_{n-i} + (\frac{1}{M_{1}})\sum^{k}_{i=1}\gamma_{i}x_{n-i}}{A+\sum^{k}_{j=1}B_{j}x_{n-j} + (\frac{1}{M_{2}})\sum^{k}_{j=1}C_{j}x_{n-j}}.$$
Thus we see that the sequence $\{x_{n}\}^{\infty}_{n=1}$ satisfies the above difference inequality. Moreover using Theorem 1 in \cite{fpinv} we see that there exists $M > 0$ and $N\in\mathbb{N}$ so that given any non-negative initial conditions, we have $x_{n}\leq M$ for all $n>N$. The fact that $M_{1}y_{n}\leq x_{n} \leq M_{2}y_{n}$ immediately yields the full result.\newline
\end{proof}

Now, in the following theorem, we use the results from \cite{fpinv} to iterate with respect to $y_{n}$.
\begin{thm}
 Suppose that we have a $k^{th}$ order system of two rational
difference equations
$$x_n=\frac{\alpha+\sum^{k}_{i=1}\beta_{i}x_{n-i} + \sum^{k}_{i=1}\gamma_{i}y_{n-i}}{A+\sum^{k}_{j=1}B_{j}x_{n-j} + \sum^{k}_{j=1}C_{j}y_{n-j}},\quad n\in\mathbb{N},$$
$$y_n=\frac{p+\sum^{k}_{i=1}\delta_{i}x_{n-i} + \sum^{k}_{i=1}\epsilon_{i}y_{n-i}}{q+\sum^{k}_{j=1}D_{j}x_{n-j} + \sum^{k}_{j=1}E_{j}y_{n-j}},\quad n\in\mathbb{N},$$
with non-negative parameters and non-negative initial conditions.
Further assume that there exists constants $M_{1},M_{2} > 0$ so that $M_{1}y_{n}\leq x_{n} \leq M_{2}y_{n}$ for all $n\in\mathbb{N}$.
Also suppose that $q>0$ and there exists a positive integer $\eta$, such that for every sequence $\{c_{m}\} ^{\infty}_{m=1} $ with $c_m\in I_{\delta}\cup I_{\epsilon}$ for $m=1,2, \dots $ there exists positive integers,  $
N_{1}
$,$
N_{2} \leq \eta
$ $ $, such that $ \sum^{N_{2}}_{m=N_{1}} c_{m} \in I_{D}\cup I_{E}$.\newline
Then there exists $M > 0$ and $N\in\mathbb{N}$ so that given any non-negative initial conditions, we have $x_{n},y_{n}\leq M$ for all $n>N$.
\end{thm}

\begin{proof}
Since $M_{1}y_{n}\leq x_{n} \leq M_{2}y_{n}$ we have
$$y_{n}=\frac{p+\sum^{k}_{i=1}\delta_{i}x_{n-i} + \sum^{k}_{i=1}\epsilon_{i}y_{n-i}}{q+\sum^{k}_{j=1}D_{j}x_{n-j} + \sum^{k}_{j=1}E_{j}y_{n-j}} \leq \frac{p+\sum^{k}_{i=1}\delta_{i}M_{2}y_{n-i} + \sum^{k}_{i=1}\epsilon_{i}y_{n-i}}{q+\sum^{k}_{j=1}D_{j}M_{1}y_{n-j} + \sum^{k}_{j=1}E_{j}y_{n-j}}.$$
Thus we see that the sequence $\{y_{n}\}^{\infty}_{n=1}$ satisfies the above difference inequality. Using Theorem 1 in \cite{fpinv} we see that there exists $M > 0$ and $N\in\mathbb{N}$ so that given any non-negative initial conditions, we have $y_{n}\leq M$ for all $n>N$. The fact that $M_{1}y_{n}\leq x_{n} \leq M_{2}y_{n}$ yields the full result.\newline
\end{proof}

\begin{thm}
 Suppose that we have a $k^{th}$ order system of two rational
difference equations
$$x_n=\frac{\alpha+\sum^{k}_{i=1}\beta_{i}x_{n-i} + \sum^{k}_{i=1}\gamma_{i}y_{n-i}}{A+\sum^{k}_{j=1}B_{j}x_{n-j} + \sum^{k}_{j=1}C_{j}y_{n-j}},\quad n\in\mathbb{N},$$
$$y_n=\frac{p+\sum^{k}_{i=1}\delta_{i}x_{n-i} + \sum^{k}_{i=1}\epsilon_{i}y_{n-i}}{q+\sum^{k}_{j=1}D_{j}x_{n-j} + \sum^{k}_{j=1}E_{j}y_{n-j}},\quad n\in\mathbb{N},$$
with non-negative parameters and non-negative initial conditions.
Further assume that there exists constants $M_{1},M_{2} > 0$ so that $M_{1}y_{n}\leq x_{n} \leq M_{2}y_{n}$ for all $n\in\mathbb{N}$.
Also suppose that $A=0$ and one of the following holds
\begin{enumerate}[(i)]
\item $I_{B}\cup I_{C}\subset I_{\beta}\cup I_{\gamma}$ and $I_{B}\neq\emptyset$
\item $q=0$, $I_{D}\cup I_{E}\subset I_{\delta}\cup I_{\epsilon}$, and $I_{C}\neq\emptyset$
\item $p,q>0$, $I_{D}\cup I_{E}\subset I_{\delta}\cup I_{\epsilon}$, and $I_{C}\neq\emptyset$
\end{enumerate}
and there exists a positive integer $\eta$, such that for every sequence $\{c_{m}\} ^{\infty}_{m=1} $ with $c_m\in I_{\beta}\cup I_{\gamma}$ for $m=1,2, \dots $ there exists positive integers,  $
N_{1}
$,$
N_{2} \leq \eta
$ $ $, such that $ \sum^{N_{2}}_{m=N_{1}} c_{m} \in I_{B}\cup I_{C}$. \newline
Then there exists $M > 0$ and $N\in\mathbb{N}$ so that given any non-negative initial conditions, we have $x_{n},y_{n}\leq M$ for all $n>N$.
\end{thm}

\begin{proof}
Suppose that $I_{B}\cup I_{C}\subset I_{\beta}\cup I_{\gamma}$ and $I_{B}\neq\emptyset$ then we have
$$x_{n}=\frac{\alpha+\sum^{k}_{i=1}\beta_{i}x_{n-i} + \sum^{k}_{i=1}\gamma_{i}y_{n-i}}{\sum^{k}_{j=1}B_{j}x_{n-j} + \sum^{k}_{j=1}C_{j}y_{n-j}}\geq \frac{\sum_{i\in I_{\beta}}\beta_{i}M_{1}y_{n-i} + \sum_{i\in I_{\gamma}}\gamma_{i}y_{n-i}}{\sum_{j\in I_{\beta}\cup I_{\gamma}}B_{j}M_{2}y_{n-j}+C_{j}y_{n-j}}$$
$$\geq \frac{\min(\min_{i\in I_{\beta}}(\beta_{i}M_{1}),\min_{i\in I_{\gamma}}(\gamma_{i}))}{\max_{j\in I_{\beta}\cup I_{\gamma}}(B_{j}M_{2}+C_{j})}.$$
So in this case $\{x_{n}\}$ is bounded below by a constant.
Now suppose that $q=0$, $I_{D}\cup I_{E}\subset I_{\delta}\cup I_{\epsilon}$, and $I_{C}\neq\emptyset$ then we have
$$y_{n}=\frac{p+\sum^{k}_{i=1}\delta_{i}x_{n-i} + \sum^{k}_{i=1}\epsilon_{i}y_{n-i}}{\sum^{k}_{j=1}D_{j}x_{n-j} + \sum^{k}_{j=1}E_{j}y_{n-j}}\geq \frac{\sum_{i\in I_{\delta}}\delta_{i}M_{1}y_{n-i} + \sum_{i\in I_{\epsilon}}\epsilon_{i}y_{n-i}}{\sum_{j\in I_{\delta}\cup I_{\epsilon}}D_{j}M_{2}y_{n-j}+ E_{j}y_{n-j}}$$
$$\geq \frac{\min(\min_{i\in I_{\delta}}(\delta_{i}M_{1}),\min_{i\in I_{\epsilon}}(\epsilon_{i}))}{\max_{j\in I_{\delta}\cup I_{\epsilon}}(D_{j}M_{2}+E_{j})}.$$
So in this case $\{y_{n}\}$ is bounded below by a constant.
Now suppose that $p, q>0$, $I_{D}\cup I_{E}\subset I_{\delta}\cup I_{\epsilon}$, and $I_{C}\neq\emptyset$ then we have
$$y_{n}=\frac{p+\sum^{k}_{i=1}\delta_{i}x_{n-i} + \sum^{k}_{i=1}\epsilon_{i}y_{n-i}}{q + \sum^{k}_{j=1}D_{j}x_{n-j} + \sum^{k}_{j=1}E_{j}y_{n-j}}\geq \frac{p + \sum_{i\in I_{\delta}}\delta_{i}M_{1}y_{n-i} + \sum_{i\in I_{\epsilon}}\epsilon_{i}y_{n-i}}{q + \sum_{j\in I_{\delta}\cup I_{\epsilon}}D_{j}M_{2}y_{n-j}+ E_{j}y_{n-j}}$$
$$\geq \frac{\min(p,\min_{i\in I_{\delta}}(\delta_{i}M_{1}),\min_{i\in I_{\epsilon}}(\epsilon_{i}))}{\max_{j\in I_{\delta}\cup I_{\epsilon}}(q+D_{j}M_{2}+E_{j})}.$$
So in this case $\{y_{n}\}$ is bounded below by a constant.
Now we have shown that there exists $A_{2}>0$, so that $\{x_{n}\}$ is bounded below by $A_{2}$ and $I_{B}\neq\emptyset$ or $\{y_{n}\}$ is bounded below by $A_{2}$ and $I_{C}\neq\emptyset$.
We use this fact to show the following.
$$x_{n}=\frac{\alpha+\sum^{k}_{i=1}\beta_{i}x_{n-i} + \sum^{k}_{i=1}\gamma_{i}y_{n-i}}{\sum^{k}_{j=1}B_{j}x_{n-j} + \sum^{k}_{j=1}C_{j}y_{n-j}}$$
$$\leq \frac{\alpha+\sum^{k}_{i=1}\beta_{i}x_{n-i} + (\frac{1}{M_{1}})\sum^{k}_{i=1}\gamma_{i}x_{n-i}}{\min_+(\sum^{k}_{j=1}B_{j},\sum^{k}_{j=1}C_{j})\frac{A_{2}}{2}+\frac{1}{2}\sum^{k}_{j=1}B_{j}x_{n-j} + \frac{1}{2}(\frac{1}{M_{2}})\sum^{k}_{j=1}C_{j}x_{n-j}}.$$
Thus we see that the sequence $\{x_{n}\}^{\infty}_{n=1}$ satisfies the above difference inequality. Moreover using Theorem 1 in \cite{fpinv} we see that there exists $M > 0$ and $N\in\mathbb{N}$ so that given any non-negative initial conditions, we have $x_{n}\leq M$ for all $n>N$. The fact that $M_{1}y_{n}\leq x_{n} \leq M_{2}y_{n}$ immediately yields the full result.\newline

\end{proof}
Now, in the following theorem, we use the results from \cite{fpinv} to iterate with respect to $y_{n}$.
\begin{thm}
 Suppose that we have a $k^{th}$ order system of two rational
difference equations
$$x_n=\frac{\alpha+\sum^{k}_{i=1}\beta_{i}x_{n-i} + \sum^{k}_{i=1}\gamma_{i}y_{n-i}}{A+\sum^{k}_{j=1}B_{j}x_{n-j} + \sum^{k}_{j=1}C_{j}y_{n-j}},\quad n\in\mathbb{N},$$
$$y_n=\frac{p+\sum^{k}_{i=1}\delta_{i}x_{n-i} + \sum^{k}_{i=1}\epsilon_{i}y_{n-i}}{q+\sum^{k}_{j=1}D_{j}x_{n-j} + \sum^{k}_{j=1}E_{j}y_{n-j}},\quad n\in\mathbb{N},$$
with non-negative parameters and non-negative initial conditions.
Further assume that there exists constants $M_{1},M_{2} > 0$ so that $M_{1}y_{n}\leq x_{n} \leq M_{2}y_{n}$ for all $n\in\mathbb{N}$.
Also suppose that $q=0$ and one of the following holds
\begin{enumerate}[(i)]
\item $A=0$, $I_{B}\cup I_{C}\subset I_{\beta}\cup I_{\gamma}$, and $I_{D}\neq\emptyset$
\item $\alpha, A>0$, $I_{B}\cup I_{C}\subset I_{\beta}\cup I_{\gamma}$, and $I_{D}\neq\emptyset$
\item $I_{D}\cup I_{E}\subset I_{\delta}\cup I_{\epsilon}$ and $I_{E}\neq\emptyset$
\end{enumerate}
and there exists a positive integer $\eta$, such that for every sequence $\{c_{m}\} ^{\infty}_{m=1} $ with $c_m\in I_{\delta}\cup I_{\epsilon}$ for $m=1,2, \dots $ there exists positive integers,  $
N_{1}
$,$
N_{2} \leq \eta
$ $ $, such that $ \sum^{N_{2}}_{m=N_{1}} c_{m} \in I_{D}\cup I_{E}$.\newline
Then there exists $M > 0$ and $N\in\mathbb{N}$ so that given any non-negative initial conditions, we have $x_{n},y_{n}\leq M$ for all $n>N$.
\end{thm}

\begin{proof}
Suppose that $A=0$, $I_{B}\cup I_{C}\subset I_{\beta}\cup I_{\gamma}$, and $I_{D}\neq\emptyset$ then we have
$$x_{n}=\frac{\alpha+\sum^{k}_{i=1}\beta_{i}x_{n-i} + \sum^{k}_{i=1}\gamma_{i}y_{n-i}}{\sum^{k}_{j=1}B_{j}x_{n-j} + \sum^{k}_{j=1}C_{j}y_{n-j}}\geq \frac{\sum_{i\in I_{\beta}}\beta_{i}M_{1}y_{n-i} + \sum_{i\in I_{\gamma}}\gamma_{i}y_{n-i}}{\sum_{j\in I_{\beta}\cup I_{\gamma}}B_{j}M_{2}y_{n-j}+C_{j}y_{n-j}}$$
$$\geq \frac{\min(\min_{i\in I_{\beta}}(\beta_{i}M_{1}),\min_{i\in I_{\gamma}}(\gamma_{i}))}{\max_{j\in I_{\beta}\cup I_{\gamma}}(B_{j}M_{2}+C_{j})}.$$
So in this case $\{x_{n}\}$ is bounded below by a constant.
 Now suppose that $A,\alpha >0$, $I_{B}\cup I_{C}\subset I_{\beta}\cup I_{\gamma}$, and $I_{D}\neq\emptyset$ then we have
$$x_{n}=\frac{\alpha+\sum^{k}_{i=1}\beta_{i}x_{n-i} + \sum^{k}_{i=1}\gamma_{i}y_{n-i}}{ A+ \sum^{k}_{j=1}B_{j}x_{n-j} + \sum^{k}_{j=1}C_{j}y_{n-j}}\geq \frac{\alpha+\sum_{i\in I_{\beta}}\beta_{i}M_{1}y_{n-i} + \sum_{i\in I_{\gamma}}\gamma_{i}y_{n-i}}{A+\sum_{j\in I_{\beta}\cup I_{\gamma}}B_{j}M_{2}y_{n-j}+C_{j}y_{n-j}}$$
$$\geq \frac{\min(\alpha , \min_{i\in I_{\beta}}(\beta_{i}M_{1}),\min_{i\in I_{\gamma}}(\gamma_{i}))}{\max_{j\in I_{\beta}\cup I_{\gamma}}(A+B_{j}M_{2}+C_{j})}.$$
So in this case $\{x_{n}\}$ is bounded below by a constant.
Now suppose that $I_{D}\cup I_{E}\subset I_{\delta}\cup I_{\epsilon}$, and $I_{E}\neq\emptyset$ then we have
$$y_{n}=\frac{p+\sum^{k}_{i=1}\delta_{i}x_{n-i} + \sum^{k}_{i=1}\epsilon_{i}y_{n-i}}{\sum^{k}_{j=1}D_{j}x_{n-j} + \sum^{k}_{j=1}E_{j}y_{n-j}}\geq \frac{\sum_{i\in I_{\delta}}\delta_{i}M_{1}y_{n-i} + \sum_{i\in I_{\epsilon}}\epsilon_{i}y_{n-i}}{\sum_{j\in I_{\delta}\cup I_{\epsilon}}D_{j}M_{2}y_{n-j}+ E_{j}y_{n-j}}$$
$$\geq \frac{\min(\min_{i\in I_{\delta}}(\delta_{i}M_{1}),\min_{i\in I_{\epsilon}}(\epsilon_{i}))}{\max_{j\in I_{\delta}\cup I_{\epsilon}}(D_{j}M_{2}+E_{j})}.$$
So in this case $\{y_{n}\}$ is bounded below by a constant.
Now we have shown that there exists $q_{2}>0$, so that $\{x_{n}\}$ is bounded below by $q_{2}$ and $I_{D}\neq\emptyset$ or $\{y_{n}\}$ is bounded below by $q_{2}$ and $I_{E}\neq\emptyset$.
We use this fact to show the following.
$$y_{n}=\frac{p+\sum^{k}_{i=1}\delta_{i}x_{n-i} + \sum^{k}_{i=1}\epsilon_{i}y_{n-i}}{\sum^{k}_{j=1}D_{j}x_{n-j} + \sum^{k}_{j=1}E_{j}y_{n-j}}$$
$$\leq \frac{p+\sum^{k}_{i=1}\delta_{i}M_{2}y_{n-i} + \sum^{k}_{i=1}\epsilon_{i}y_{n-i}}{\min_+(\sum^{k}_{j=1}D_{j},\sum^{k}_{j=1}E_{j})\frac{q_{2}}{2}+\frac{1}{2}\sum^{k}_{j=1}D_{j}M_{1}y_{n-j} + \frac{1}{2}\sum^{k}_{j=1}E_{j}y_{n-j}}.$$
Thus we see that the sequence $\{y_{n}\}^{\infty}_{n=1}$ satisfies the above difference inequality. Again using Theorem 1 in \cite{fpinv} we see that there exists $M > 0$ and $N\in\mathbb{N}$ so that given any non-negative initial conditions, we have $y_{n}\leq M$ for all $n>N$. The fact that $M_{1}y_{n}\leq x_{n} \leq M_{2}y_{n}$ yields the full result.\newline
\end{proof}

For the following two theorems, no iteration is necessary. We only need to use some algebraic techniques similar to those used in \cite{c1}, coupled with the condition that there exists constants $M_{1},M_{2} > 0$ so that $M_{1}y_{n}\leq x_{n} \leq M_{2}y_{n}$ for all $n\in\mathbb{N}$.

\begin{thm}
 Suppose that we have a $k^{th}$ order system of two rational
difference equations
$$x_n=\frac{\alpha+\sum^{k}_{i=1}\beta_{i}x_{n-i} + \sum^{k}_{i=1}\gamma_{i}y_{n-i}}{A+\sum^{k}_{j=1}B_{j}x_{n-j} + \sum^{k}_{j=1}C_{j}y_{n-j}},\quad n\in\mathbb{N},$$
$$y_n=\frac{p+\sum^{k}_{i=1}\delta_{i}x_{n-i} + \sum^{k}_{i=1}\epsilon_{i}y_{n-i}}{q+\sum^{k}_{j=1}D_{j}x_{n-j} + \sum^{k}_{j=1}E_{j}y_{n-j}},\quad n\in\mathbb{N},$$
with non-negative parameters and non-negative initial conditions.
Further assume that there exists constants $M_{1},M_{2} > 0$ so that $M_{1}y_{n}\leq x_{n} \leq M_{2}y_{n}$ for all $n\in\mathbb{N}$.
Also suppose that $q=0$, $p =0$, and  $I_{\delta}\cup I_{\epsilon} \subset I_{D}\cup I_{E}$.\newline
Then there exists $M > 0$ and $N\in\mathbb{N}$ so that given any non-negative initial conditions, we have $x_{n},y_{n}\leq M$ for all $n>N$.
\end{thm}

\begin{proof}
We use the fact that $M_{1}y_{n}\leq x_{n} \leq M_{2}y_{n}$ to show
$$y_{n}=\frac{\sum^{k}_{i=1}\delta_{i}x_{n-i} + \sum^{k}_{i=1}\epsilon_{i}y_{n-i}}{\sum^{k}_{j=1}D_{j}x_{n-j} + \sum^{k}_{j=1}E_{j}y_{n-j}} \leq \frac{\sum^{k}_{i=1}\delta_{i}M_{2}y_{n-i} + \sum^{k}_{i=1}\epsilon_{i}y_{n-i}}{\sum^{k}_{j=1}D_{j}M_{1}y_{n-j} + \sum^{k}_{j=1}E_{j}y_{n-j}}\leq $$
$$\frac{\sum_{i\in I_{\delta}}\delta_{i}M_{2}y_{n-i} + \sum_{i\in I_{\epsilon}}\epsilon_{i}y_{n-i}}{\sum_{j\in I_{\delta}\cup I_{\epsilon}}D_{j}M_{1}y_{n-j} + \sum_{j\in I_{\delta}\cup I_{\epsilon}}E_{j}y_{n-j}}= \frac{\sum_{i\in I_{\delta}\cup I_{\epsilon}}(\delta_{i}M_{2}+\epsilon_{i})y_{n-i}}{\sum_{j\in I_{\delta}\cup I_{\epsilon}}(D_{j}M_{1}+E_{j})y_{n-j}}\leq$$
$$\left(\frac{\max_{i\in I_{\delta}\cup I_{\epsilon}}(\delta_{i}M_{2}+\epsilon_{i})}{\min_{j\in I_{\delta}\cup I_{\epsilon}}(D_{j}M_{1}+E_{j})}\right)\frac{\sum_{i\in I_{\delta}\cup I_{\epsilon}}y_{n-i}}{\sum_{j\in I_{\delta}\cup I_{\epsilon}}y_{n-j}}= \frac{\max_{i\in I_{\delta}\cup I_{\epsilon}}(\delta_{i}M_{2}+\epsilon_{i})}{\min_{j\in I_{\delta}\cup I_{\epsilon}}(D_{j}M_{1}+E_{j})}.$$
\end{proof}

\begin{thm}
 Suppose that we have a $k^{th}$ order system of two rational
difference equations
$$x_n=\frac{\alpha+\sum^{k}_{i=1}\beta_{i}x_{n-i} + \sum^{k}_{i=1}\gamma_{i}y_{n-i}}{A+\sum^{k}_{j=1}B_{j}x_{n-j} + \sum^{k}_{j=1}C_{j}y_{n-j}},\quad n\in\mathbb{N},$$
$$y_n=\frac{p+\sum^{k}_{i=1}\delta_{i}x_{n-i} + \sum^{k}_{i=1}\epsilon_{i}y_{n-i}}{q+\sum^{k}_{j=1}D_{j}x_{n-j} + \sum^{k}_{j=1}E_{j}y_{n-j}},\quad n\in\mathbb{N},$$
with non-negative parameters and non-negative initial conditions.
Further assume that there exists constants $M_{1},M_{2} > 0$ so that $M_{1}y_{n}\leq x_{n} \leq M_{2}y_{n}$ for all $n\in\mathbb{N}$.
Also suppose that $A=0$, $\alpha =0$, and $I_{\beta}\cup I_{\gamma} \subset I_{B}\cup I_{C}$.\newline
Then there exists $M > 0$ and $N\in\mathbb{N}$ so that given any non-negative initial conditions, we have $x_{n},y_{n}\leq M$ for all $n>N$.
\end{thm}

\begin{proof}
Since $M_{1}y_{n}\leq x_{n} \leq M_{2}y_{n}$ we have
$$x_{n}=\frac{\sum^{k}_{i=1}\beta_{i}x_{n-i} + \sum^{k}_{i=1}\gamma_{i}y_{n-i}}{\sum^{k}_{j=1}B_{j}x_{n-j} + \sum^{k}_{j=1}C_{j}y_{n-j}} \leq \frac{\sum^{k}_{i=1}\beta_{i}x_{n-i} + (\frac{1}{M_{1}})\sum^{k}_{i=1}\gamma_{i}x_{n-i}}{\sum^{k}_{j=1}B_{j}x_{n-j} + (\frac{1}{M_{2}})\sum^{k}_{j=1}C_{j}x_{n-j}}\leq $$
$$\frac{\sum_{i\in I_{\beta}}\beta_{i}x_{n-i} + (\frac{1}{M_{1}})\sum_{i\in I_{\gamma}}\gamma_{i}x_{n-i}}{\sum_{j\in I_{\beta}\cup I_{\gamma}}B_{j}x_{n-j} + (\frac{1}{M_{2}})\sum_{j\in I_{\beta}\cup I_{\gamma}}C_{j}x_{n-j}}= \frac{\sum_{i\in I_{\beta}\cup I_{\gamma}}(\beta_{i}+\frac{\gamma_{i}}{M_{1}})x_{n-i}}{\sum_{j\in I_{\beta}\cup I_{\gamma}}(B_{j}+\frac{C_{j}}{M_{2}})x_{n-j}}\leq$$
$$\left(\frac{\max_{i\in I_{\beta}\cup I_{\gamma}}(\beta_{i}+\frac{\gamma_{i}}{M_{1}})}{\min_{j\in I_{\beta}\cup I_{\gamma}}(B_{j}+\frac{C_{j}}{M_{2}})}\right)\frac{\sum_{i\in I_{\beta}\cup I_{\gamma}}x_{n-i}}{\sum_{j\in I_{\beta}\cup I_{\gamma}}x_{n-j}}= \frac{\max_{i\in I_{\beta}\cup I_{\gamma}}(\beta_{i}+\frac{\gamma_{i}}{M_{1}})}{\min_{j\in I_{\beta}\cup I_{\gamma}}(B_{j}+\frac{C_{j}}{M_{2}})}.$$
\end{proof}

\section{Some boundedness Results Without Comparison}

In this section we use some algebraic techniques similar to those applied in \cite{c1}. Once we have obtained a difference inequality we apply Theorem 1 of \cite{fpinv}.

\begin{thm}
Suppose that we have a $k^{th}$ order system of two rational
difference equations
$$x_n=\frac{\alpha+\sum^{k}_{i=1}\beta_{i}x_{n-i} + \sum^{k}_{i=1}\gamma_{i}y_{n-i}}{A+\sum^{k}_{j=1}B_{j}x_{n-j} + \sum^{k}_{j=1}C_{j}y_{n-j}},\quad n\in\mathbb{N},$$
$$y_n=\frac{p+\sum^{k}_{i=1}\delta_{i}x_{n-i} + \sum^{k}_{i=1}\epsilon_{i}y_{n-i}}{q+\sum^{k}_{j=1}D_{j}x_{n-j} + \sum^{k}_{j=1}E_{j}y_{n-j}},\quad n\in\mathbb{N},$$
with non-negative parameters and non-negative initial conditions.
Also suppose that $A>0$, $I_{\gamma} \subset I_{C}$, and there exists a positive integer $\eta$, such that for every sequence $\{c_{m}\} ^{\infty}_{m=1} $ with $c_m\in I_{\beta}$ for $m=1,2, \dots $ there exists positive integers,  $
N_{1}
$,$
N_{2} \leq \eta
$ $ $, such that $ \sum^{N_{2}}_{m=N_{1}} c_{m} \in I_{B}$.\newline
Then there exists $M > 0$ and $N\in\mathbb{N}$ so that given any non-negative initial conditions, we have $x_{n}\leq M$ for all $n>N$.
\end{thm}

\begin{proof}
$$x_n=\frac{\alpha+\sum^{k}_{i=1}\beta_{i}x_{n-i} + \sum^{k}_{i=1}\gamma_{i}y_{n-i}}{A+\sum^{k}_{j=1}B_{j}x_{n-j} + \sum^{k}_{j=1}C_{j}y_{n-j}}\leq $$ $$\frac{\sum^{k}_{i=1}\gamma_{i}y_{n-i}}{A+\sum^{k}_{j=1}B_{j}x_{n-j} + \sum^{k}_{j=1}C_{j}y_{n-j}} + \frac{\alpha+\sum^{k}_{i=1}\beta_{i}x_{n-i}}{A+\sum^{k}_{j=1}B_{j}x_{n-j} + \sum^{k}_{j=1}C_{j}y_{n-j}}\leq$$ $$\sum_{i\in I_{\gamma}}\frac{\gamma_{i}}{C_{i}} + \frac{\alpha+\sum^{k}_{i=1}\beta_{i}x_{n-i}}{A+\sum^{k}_{j=1}B_{j}x_{n-j}}.$$
Thus we see that the sequence $\{x_{n}\}^{\infty}_{n=1}$ satisfies the above difference inequality. Moreover using Theorem 1 in \cite{fpinv} we see that there exists $M > 0$ and $N\in\mathbb{N}$ so that given any non-negative initial conditions, we have $x_{n}\leq M$ for all $n>N$.\newline
\end{proof}

\begin{thm}
Suppose that we have a $k^{th}$ order system of two rational
difference equations
$$x_n=\frac{\alpha+\sum^{k}_{i=1}\beta_{i}x_{n-i} + \sum^{k}_{i=1}\gamma_{i}y_{n-i}}{A+\sum^{k}_{j=1}B_{j}x_{n-j} + \sum^{k}_{j=1}C_{j}y_{n-j}},\quad n\in\mathbb{N},$$
$$y_n=\frac{p+\sum^{k}_{i=1}\delta_{i}x_{n-i} + \sum^{k}_{i=1}\epsilon_{i}y_{n-i}}{q+\sum^{k}_{j=1}D_{j}x_{n-j} + \sum^{k}_{j=1}E_{j}y_{n-j}},\quad n\in\mathbb{N},$$
with non-negative parameters and non-negative initial conditions.
Also suppose that $A=0$, $\alpha =0$, $I_{\beta} \subset I_{B}$, and $ I_{\gamma} \subset I_{C}$.\newline
Then there exists $M > 0$ and $N\in\mathbb{N}$ so that given any non-negative initial conditions, we have $x_{n}\leq M$ for all $n>N$.
\end{thm}

\begin{proof}
$$x_n=\frac{\sum^{k}_{i=1}\beta_{i}x_{n-i} + \sum^{k}_{i=1}\gamma_{i}y_{n-i}}{\sum^{k}_{j=1}B_{j}x_{n-j} + \sum^{k}_{j=1}C_{j}y_{n-j}}\leq$$
$$ \left(\frac{\max_{i\in I_{\beta}\cup I_{\gamma}}(\beta_{i}+\gamma_{i})}{\min(\min_{j\in I_{\beta}}(B_{j}),\min_{j\in I_{\gamma}}(C_{j}))}\right)\frac{\sum_{i\in I_{\beta}}x_{n-i} + \sum_{i\in I_{\gamma}}y_{n-i}}{\sum_{j\in I_{\beta}}x_{n-j} + \sum_{j\in I_{\gamma}}y_{n-j}}=$$
$$\frac{\max_{i\in I_{\beta}\cup I_{\gamma}}(\beta_{i}+\gamma_{i})}{\min(\min_{j\in I_{\beta}}(B_{j}),\min_{j\in I_{\gamma}}(C_{j}))}.$$
\end{proof}

\begin{thm}
Suppose that we have a $k^{th}$ order system of two rational
difference equations
$$x_n=\frac{\alpha+\sum^{k}_{i=1}\beta_{i}x_{n-i} + \sum^{k}_{i=1}\gamma_{i}y_{n-i}}{A+\sum^{k}_{j=1}B_{j}x_{n-j} + \sum^{k}_{j=1}C_{j}y_{n-j}},\quad n\in\mathbb{N},$$
$$y_n=\frac{p+\sum^{k}_{i=1}\delta_{i}x_{n-i} + \sum^{k}_{i=1}\epsilon_{i}y_{n-i}}{q+\sum^{k}_{j=1}D_{j}x_{n-j} + \sum^{k}_{j=1}E_{j}y_{n-j}},\quad n\in\mathbb{N},$$
with non-negative parameters and non-negative initial conditions.
Also suppose that $\{y_{n}\}^{\infty}_{n=1}$ is bounded above and below, $I_{C}\neq \emptyset$, and there exists a positive integer $\eta$, such that for every sequence $\{c_{m}\} ^{\infty}_{m=1} $ with $c_m\in I_{\beta}$ for $m=1,2, \dots $ there exists positive integers,  $
N_{1}
$,$
N_{2} \leq \eta
$ $ $, such that $ \sum^{N_{2}}_{m=N_{1}} c_{m} \in I_{B}$.\newline
Then there exists $M > 0$ and $N\in\mathbb{N}$ so that given any non-negative initial conditions, we have $x_{n}\leq M$ for all $n>N$.
\end{thm}

\begin{proof}
We know that $\{y_{n}\}^{\infty}_{n=1}$ is bounded above and below let us call those bounds $m_{1}\leq y_{n}\leq m_{2}$. This means that we have
$$x_n=\frac{\alpha+\sum^{k}_{i=1}\beta_{i}x_{n-i} + \sum^{k}_{i=1}\gamma_{i}y_{n-i}}{A+ \sum^{k}_{j=1}B_{j}x_{n-j} + \sum^{k}_{j=1}C_{j}y_{n-j}}\leq \frac{\alpha+\sum^{k}_{i=1}\beta_{i}x_{n-i} + \left(\sum^{k}_{i=1}\gamma_{i}\right)m_{2}}{\left(\sum^{k}_{j=1}C_{j}\right) m_{1} +\sum^{k}_{j=1}B_{j}x_{n-j}}$$
Thus we see that the sequence $\{x_{n}\}^{\infty}_{n=1}$ satisfies the above difference inequality. Moreover using Theorem 1 in \cite{fpinv} we see that there exists $M > 0$ and $N\in\mathbb{N}$ so that given any non-negative initial conditions, we have $x_{n}\leq M$ for all $n>N$.\newline
\end{proof}

\begin{thm}
Suppose that we have a $k^{th}$ order system of two rational
difference equations
$$x_n=\frac{\alpha+\sum^{k}_{i=1}\beta_{i}x_{n-i} + \sum^{k}_{i=1}\gamma_{i}y_{n-i}}{A+\sum^{k}_{j=1}B_{j}x_{n-j} + \sum^{k}_{j=1}C_{j}y_{n-j}},\quad n\in\mathbb{N},$$
$$y_n=\frac{p+\sum^{k}_{i=1}\delta_{i}x_{n-i} + \sum^{k}_{i=1}\epsilon_{i}y_{n-i}}{q+\sum^{k}_{j=1}D_{j}x_{n-j} + \sum^{k}_{j=1}E_{j}y_{n-j}},\quad n\in\mathbb{N},$$
with non-negative parameters and non-negative initial conditions.
Also suppose that $A=0$ and one of the following holds
\begin{enumerate}[(i)]
\item $I_{B}\subset I_{\beta}$, $I_{C}\subset I_{\gamma}$ and $I_{B}\neq\emptyset$
\item $q=0$, $I_{D}\subset I_{\delta}$, $I_{E}\subset I_{\epsilon}$, and $I_{C}\neq\emptyset$
\item $p,q>0$, $I_{D}\subset I_{\delta}$, $I_{E}\subset I_{\epsilon}$, and $I_{C}\neq\emptyset$
\end{enumerate}
$I_{\gamma} \subset I_{C}$, and there exists a positive integer $\eta$, such that for every sequence $\{c_{m}\} ^{\infty}_{m=1} $ with $c_m\in I_{\beta}$ for $m=1,2, \dots $ there exists positive integers,  $
N_{1}
$,$
N_{2} \leq \eta
$ $ $, such that $ \sum^{N_{2}}_{m=N_{1}} c_{m} \in I_{B}$.\newline
Then there exists $M > 0$ and $N\in\mathbb{N}$ so that given any non-negative initial conditions, we have $x_{n}\leq M$ for all $n>N$.
\end{thm}

\begin{proof}
Suppose $I_{B}\subset I_{\beta}$, $I_{C}\subset I_{\gamma}$, and $I_{B}\neq\emptyset$ then we have
$$x_{n}=\frac{\alpha+\sum^{k}_{i=1}\beta_{i}x_{n-i} + \sum^{k}_{i=1}\gamma_{i}y_{n-i}}{\sum^{k}_{j=1}B_{j}x_{n-j} + \sum^{k}_{j=1}C_{j}y_{n-j}}\geq \frac{\sum_{i\in I_{\beta}}\beta_{i}x_{n-i} + \sum_{i\in I_{\gamma}}\gamma_{i}y_{n-i}}{\sum_{j\in I_{\beta}}B_{j}x_{n-j}+\sum_{j\in I_{\gamma}}C_{j}y_{n-j}}$$
$$\geq \frac{\min(\min_{i\in I_{\beta}}(\beta_{i}),\min_{i\in I_{\gamma}}(\gamma_{i}))}{\max(\max_{j\in I_{\beta}}(B_{j}),\max_{j\in I_{\gamma}}(C_{j}))}.$$
So in this case $\{x_{n}\}$ is bounded below by a constant.
Now suppose $q=0$, $I_{D}\subset I_{\delta}$, $I_{E}\subset I_{\epsilon}$, and $I_{C}\neq\emptyset$ then we have
$$y_{n}=\frac{p+\sum^{k}_{i=1}\delta_{i}x_{n-i} + \sum^{k}_{i=1}\epsilon_{i}y_{n-i}}{\sum^{k}_{j=1}D_{j}x_{n-j} + \sum^{k}_{j=1}E_{j}y_{n-j}}\geq \frac{\sum_{i\in I_{\delta}}\delta_{i}x_{n-i} + \sum_{i\in I_{\epsilon}}\epsilon_{i}y_{n-i}}{\sum_{j\in I_{\delta}}D_{j}x_{n-j}+\sum_{j\in I_{\epsilon}}E_{j}y_{n-j}}$$
$$\geq \frac{\min(\min_{i\in I_{\delta}}(\delta_{i}),\min_{i\in I_{\epsilon}}(\epsilon_{i}))}{\max(\max_{j\in I_{\delta}}(D_{j}),\max_{j\in I_{\epsilon}}(E_{j}))}.$$
So in this case $\{y_{n}\}$ is bounded below by a constant.
Now suppose $p,q>0$, $I_{D}\subset I_{\delta}$, $I_{E}\subset I_{\epsilon}$, and $I_{C}\neq\emptyset$ then we have
$$y_{n}=\frac{p+\sum^{k}_{i=1}\delta_{i}x_{n-i} + \sum^{k}_{i=1}\epsilon_{i}y_{n-i}}{q+\sum^{k}_{j=1}D_{j}x_{n-j} + \sum^{k}_{j=1}E_{j}y_{n-j}}\geq \frac{p+\sum_{i\in I_{\delta}}\delta_{i}x_{n-i} + \sum_{i\in I_{\epsilon}}\epsilon_{i}y_{n-i}}{q+\sum_{j\in I_{\delta}}D_{j}x_{n-j}+\sum_{j\in I_{\epsilon}}E_{j}y_{n-j}}$$
$$\geq \frac{\min(p,\min_{i\in I_{\delta}}(\delta_{i}),\min_{i\in I_{\epsilon}}(\epsilon_{i}))}{\max(q,\max_{j\in I_{\delta}}(D_{j}),\max_{j\in I_{\epsilon}}(E_{j}))}.$$
So in this case $\{y_{n}\}$ is bounded below by a constant.
Now we have shown that there exists $A_{2}>0$, so that $\{x_{n}\}$ is bounded below by $A_{2}$ and $I_{B}\neq\emptyset$ or $\{y_{n}\}$ is bounded below by $A_{2}$ and $I_{C}\neq\emptyset$.
We use this fact to show the following.
$$x_n=\frac{\alpha+\sum^{k}_{i=1}\beta_{i}x_{n-i} + \sum^{k}_{i=1}\gamma_{i}y_{n-i}}{\sum^{k}_{j=1}B_{j}x_{n-j} + \sum^{k}_{j=1}C_{j}y_{n-j}}\leq \frac{\sum^{k}_{i=1}\gamma_{i}y_{n-i}}{\sum^{k}_{j=1}B_{j}x_{n-j} + \sum^{k}_{j=1}C_{j}y_{n-j}}$$ $$ + \frac{\alpha+\sum^{k}_{i=1}\beta_{i}x_{n-i}}{\min_+(\sum^{k}_{j=1}B_{j},\sum^{k}_{j=1}C_{j})\frac{A_{2}}{2}+\frac{1}{2}\sum^{k}_{j=1}B_{j}x_{n-j} + \frac{1}{2}\sum^{k}_{j=1}C_{j}y_{n-j}}$$ $$\leq \sum_{i\in I_{\gamma}}\frac{\gamma_{i}}{C_{i}} + \frac{\alpha+\sum^{k}_{i=1}\beta_{i}x_{n-i}}{\min_+(\sum^{k}_{j=1}B_{j},\sum^{k}_{j=1}C_{j})\frac{A_{2}}{2}+\frac{1}{2}\sum^{k}_{j=1}B_{j}x_{n-j}}.$$
Thus we see that the sequence $\{x_{n}\}^{\infty}_{n=1}$ satisfies the above difference inequality. Moreover using Theorem 1 in \cite{fpinv} we see that there exists $M > 0$ and $N\in\mathbb{N}$ so that given any non-negative initial conditions, we have $x_{n}\leq M$ for all $n>N$.\newline
\end{proof}

\begin{thm}
Suppose that we have a $k^{th}$ order system of two rational
difference equations
$$x_n=\frac{\alpha+\sum^{k}_{i=1}\beta_{i}x_{n-i} + \sum^{k}_{i=1}\gamma_{i}y_{n-i}}{A+\sum^{k}_{j=1}B_{j}x_{n-j} + \sum^{k}_{j=1}C_{j}y_{n-j}},\quad n\in\mathbb{N},$$
$$y_n=\frac{p+\sum^{k}_{i=1}\delta_{i}x_{n-i} + \sum^{k}_{i=1}\epsilon_{i}y_{n-i}}{q+\sum^{k}_{j=1}D_{j}x_{n-j} + \sum^{k}_{j=1}E_{j}y_{n-j}},\quad n\in\mathbb{N},$$
with non-negative parameters and non-negative initial conditions.
Also suppose that $\{y_{n}\}^{\infty}_{n=1}$ is bounded above, $A>0$, and there exists a positive integer $\eta$, such that for every sequence $\{c_{m}\} ^{\infty}_{m=1} $ with $c_m\in I_{\beta}$ for $m=1,2, \dots $ there exists positive integers,  $
N_{1}
$,$
N_{2} \leq \eta
$ $ $, such that $ \sum^{N_{2}}_{m=N_{1}} c_{m} \in I_{B}$.\newline
Then there exists $M > 0$ and $N\in\mathbb{N}$ so that given any non-negative initial conditions, we have $x_{n}\leq M$ for all $n>N$.
\end{thm}

\begin{proof}
We know that $\{y_{n}\}^{\infty}_{n=1}$ is bounded above, let us call that bound $y_{n}\leq m_{3}$.
$$x_n=\frac{\alpha+\sum^{k}_{i=1}\beta_{i}x_{n-i} + \sum^{k}_{i=1}\gamma_{i}y_{n-i}}{A+\sum^{k}_{j=1}B_{j}x_{n-j} + \sum^{k}_{j=1}C_{j}y_{n-j}}\leq $$ $$\frac{\sum^{k}_{i=1}\gamma_{i}m_{3}}{A} + \frac{\alpha+\sum^{k}_{i=1}\beta_{i}x_{n-i}}{A+\sum^{k}_{j=1}B_{j}x_{n-j}}.$$
Thus we see that the sequence $\{x_{n}\}^{\infty}_{n=1}$ satisfies the above difference inequality. Moreover using Theorem 1 in \cite{fpinv} we see that there exists $M > 0$ and $N\in\mathbb{N}$ so that given any non-negative initial conditions, we have $x_{n}\leq M$ for all $n>N$.\newline
\end{proof}

\begin{thm}
Suppose that we have a $k^{th}$ order system of two rational
difference equations
$$x_n=\frac{\alpha+\sum^{k}_{i=1}\beta_{i}x_{n-i} + \sum^{k}_{i=1}\gamma_{i}y_{n-i}}{A+\sum^{k}_{j=1}B_{j}x_{n-j} + \sum^{k}_{j=1}C_{j}y_{n-j}},\quad n\in\mathbb{N},$$
$$y_n=\frac{p+\sum^{k}_{i=1}\delta_{i}x_{n-i} + \sum^{k}_{i=1}\epsilon_{i}y_{n-i}}{q+\sum^{k}_{j=1}D_{j}x_{n-j} + \sum^{k}_{j=1}E_{j}y_{n-j}},\quad n\in\mathbb{N},$$
with non-negative parameters and non-negative initial conditions. Assume that $A,q>0$, $I_{\gamma} \subset I_{C}$, and there exists a positive integer $\eta$, such that for every sequence $\{c_{m}\} ^{\infty}_{m=1} $ with $c_m\in I_{\beta}$ for $m=1,2, \dots $ there exists positive integers,  $
N_{1}
$,$
N_{2} \leq \eta
$ $ $, such that $ \sum^{N_{2}}_{m=N_{1}} c_{m} \in I_{B}$.\newline
Further assume that there exists a positive integer $\eta_{2}$, such that for every sequence $\{d_{m}\} ^{\infty}_{m=1} $ with $d_m\in I_{\epsilon}$ for $m=1,2, \dots $ there exists positive integers,  $
N_{3}
$,$
N_{4} \leq \eta_{2}
$ $ $, such that $ \sum^{N_{4}}_{m=N_{3}} d_{m} \in I_{E}$.
Then there exists $M > 0$ and $N\in\mathbb{N}$ so that given any non-negative initial conditions, we have $x_{n},y_{n}\leq M$ for all $n>N$.
\end{thm}

\begin{proof}
By Theorem 7 we get immediately that there exists $M_{1} > 0$ and $N\in\mathbb{N}$ so that given any non-negative initial conditions, we have $x_{n}\leq M_{1}$ for all $n>N$. We now use this fact as follows
$$y_{n}=\frac{p+\sum^{k}_{i=1}\delta_{i}x_{n-i} + \sum^{k}_{i=1}\epsilon_{i}y_{n-i}}{q+\sum^{k}_{j=1}D_{j}x_{n-j} + \sum^{k}_{j=1}E_{j}y_{n-j}}\leq \frac{p+\sum^{k}_{i=1}\delta_{i}M_{1} + \sum^{k}_{i=1}\epsilon_{i}y_{n-i}}{q+ \sum^{k}_{j=1}E_{j}y_{n-j}}.$$
Thus we see that the sequence $\{y_{n}\}^{\infty}_{n=1}$ satisfies the above difference inequality. Moreover using Theorem 1 in \cite{fpinv} we see that there exists $M_{2} > 0$ and $N\in\mathbb{N}$ so that given any non-negative initial conditions, we have $y_{n}\leq M_{2}$ for all $n>N$.\newline
\end{proof}

\section{Some Boundedness Results Involving a One Sided Comparison}

For the results in this section, we either assume that there exists $M_{1}>0$ so that $y_{n}\leq M_{1}x_{n}$ for all $n\in\mathbb{N}$, or we assume that there exists $M_{1}>0$ and $M_{2}\geq 0$ so that $y_{n}\leq M_{1}x_{n} + M_{2}$ for all $n\in\mathbb{N}$. We use these one sided comparisons and Theorem 1 of \cite{fpinv} to prove that every solution is bounded for certain systems of rational difference equations.

For the first two results we assume that $y_{n}\leq M_{1}x_{n}$ for some $M_{1}>0$. No iteration is used in the next result.

\begin{thm}
Suppose that we have a $k^{th}$ order system of two rational
difference equations
$$x_n=\frac{\alpha+\sum^{k}_{i=1}\beta_{i}x_{n-i} + \sum^{k}_{i=1}\gamma_{i}y_{n-i}}{A+\sum^{k}_{j=1}B_{j}x_{n-j} + \sum^{k}_{j=1}C_{j}y_{n-j}},\quad n\in\mathbb{N},$$
$$y_n=\frac{p+\sum^{k}_{i=1}\delta_{i}x_{n-i} + \sum^{k}_{i=1}\epsilon_{i}y_{n-i}}{q+\sum^{k}_{j=1}D_{j}x_{n-j} + \sum^{k}_{j=1}E_{j}y_{n-j}},\quad n\in\mathbb{N},$$
with non-negative parameters and non-negative initial conditions. Assume that there exists $M_{1}>0$ so that $y_{n}\leq M_{1}x_{n}$ for all $n\in\mathbb{N}$.
Further assume that $A=0$, $\alpha =0$, $I_{\beta}\cup(I_{\gamma}\setminus I_{C})\subset I_{B}$.
Then there exists $M > 0$ and $N\in\mathbb{N}$ so that given any non-negative initial conditions, we have $x_{n},y_{n}\leq M$ for all $n>N$.
\end{thm}

\begin{proof}
Since we have assumed $y_{n}\leq M_{1}x_{n}$ it suffices to find a bound for $\{x_{n}\}^{\infty}_{n=1}$. Notice that
$$x_n=\frac{\sum^{k}_{i=1}\beta_{i}x_{n-i} + \sum^{k}_{i=1}\gamma_{i}y_{n-i}}{\sum^{k}_{j=1}B_{j}x_{n-j} + \sum^{k}_{j=1}C_{j}y_{n-j}}\leq \sum_{i\in I_{C}\cap I_{\gamma}}\frac{\gamma_{i}}{C_{i}}+\frac{\sum^{k}_{i=1}\beta_{i}x_{n-i} + \sum_{i\in I_{\gamma}\setminus I_{C}}\gamma_{i}M_{1}x_{n-i}}{\sum^{k}_{j=1}B_{j}x_{n-j}}$$
$$\leq \sum_{i\in I_{C}\cap I_{\gamma}}\frac{\gamma_{i}}{C_{i}}+\frac{(1+M_{1})\max_{i\in I_{\beta}\cup(I_{\gamma}\setminus I_{C})}(\beta_{i}+\gamma_{i})}{\min_{i\in I_{\beta}\cup(I_{\gamma}\setminus I_{C})}(B_{i})}.$$
So $x_{n}\leq M$ for all $n>N$. Since we have assumed that there exists $M_{1}>0$ so that $y_{n}\leq M_{1}x_{n}$ for all $n\in\mathbb{N}$, we obtain the full result.\newline
\end{proof}

In the following result we use Theorem 1 in \cite{fpinv} to iterate, first with respect to $y_{n}$ and then with respect to $x_{n}$.

\begin{thm}
Suppose that we have a $k^{th}$ order system of two rational
difference equations
$$x_n=\frac{\alpha+\sum^{k}_{i=1}\beta_{i}x_{n-i} + \sum^{k}_{i=1}\gamma_{i}y_{n-i}}{A+\sum^{k}_{j=1}B_{j}x_{n-j} + \sum^{k}_{j=1}C_{j}y_{n-j}},\quad n\in\mathbb{N},$$
$$y_n=\frac{p+\sum^{k}_{i=1}\delta_{i}x_{n-i} + \sum^{k}_{i=1}\epsilon_{i}y_{n-i}}{q+\sum^{k}_{j=1}D_{j}x_{n-j} + \sum^{k}_{j=1}E_{j}y_{n-j}},\quad n\in\mathbb{N},$$
with non-negative parameters and non-negative initial conditions. Assume that there exists $M_{1}>0$ so that $y_{n}\leq M_{1}x_{n}$ for all $n\in\mathbb{N}$.
Further assume that $A=0$,$q>0$, one of the following holds
\begin{enumerate}[(i)]
\item $I_{B}\subset I_{\beta}$, $I_{C}\subset I_{\gamma}$ and $I_{B}\neq\emptyset$
\item $p>0$, $I_{D}\subset I_{\delta}$, $I_{E}\subset I_{\epsilon}$ and $I_{C}\neq\emptyset$
\end{enumerate}
$I_{\delta}\subset I_{D}$, there exists a positive integer $\eta_{1}$, such that for every sequence $\{c_{m}\} ^{\infty}_{m=1} $ with $c_m\in I_{\epsilon}$ for $m=1,2, \dots $ there exists positive integers,  $
N_{1}
$,$
N_{2} \leq \eta_{1}
$ $ $, such that $ \sum^{N_{2}}_{m=N_{1}} c_{m} \in I_{D}\cup I_{E}$, and there exists a positive integer $\eta_{2}$, such that for every sequence $\{d_{m}\} ^{\infty}_{m=1} $ with $d_m\in I_{\beta}$ for $m=1,2, \dots $ there exists positive integers,  $
N_{3}
$,$
N_{4} \leq \eta_{2}
$ $ $, such that $ \sum^{N_{4}}_{m=N_{3}} d_{m} \in I_{B}$.
Then there exists $M > 0$ and $N\in\mathbb{N}$ so that given any non-negative initial conditions, we have $x_{n},y_{n}\leq M$ for all $n>N$.
\end{thm}

\begin{proof}
First we must prove that $\{y_{n}\}^{\infty}_{n=1}$ is bounded above. Since $q>0$ and $I_{\delta}\subset I_{D}$ we have,
$$y_{n}= \frac{p+\sum^{k}_{i=1}\delta_{i}x_{n-i} + \sum^{k}_{i=1}\epsilon_{i}y_{n-i}}{q+\sum^{k}_{j=1}D_{j}x_{n-j} + \sum^{k}_{j=1}E_{j}y_{n-j}}\leq \frac{p}{q}+\sum_{i\in I_{\delta}}\frac{\delta_{i}}{D_{i}}+ \frac{\sum^{k}_{i=1}\epsilon_{i}y_{n-i}}{q+\sum^{k}_{j=1}D_{j}x_{n-j} + \sum^{k}_{j=1}E_{j}y_{n-j}}$$
$$\leq \frac{p}{q}+\sum_{i\in I_{\delta}}\frac{\delta_{i}}{D_{i}}+ \frac{\sum^{k}_{i=1}\epsilon_{i}y_{n-i}}{q+\sum^{k}_{j=1}\frac{D_{j}y_{n-j}}{M_{1}} + \sum^{k}_{j=1}E_{j}y_{n-j}}.$$
Thus the sequence $\{y_{n}\}^{\infty}_{n=1}$ satisfies the above difference inequality. Using Theorem 1 in \cite{fpinv} we see that there exists $M_{2} > 0$ and $N\in\mathbb{N}$ so that given any non-negative initial conditions, we have $y_{n}\leq M_{2}$ for all $n>N$. Let us show that $\{x_{n}\}$ or $\{y_{n}\}$ is bounded below by a constant $A_{2}$.\newline
Suppose $I_{B}\subset I_{\beta}$, $I_{C}\subset I_{\gamma}$, and $I_{B}\neq\emptyset$ then we have
$$x_{n}=\frac{\alpha+\sum^{k}_{i=1}\beta_{i}x_{n-i} + \sum^{k}_{i=1}\gamma_{i}y_{n-i}}{\sum^{k}_{j=1}B_{j}x_{n-j} + \sum^{k}_{j=1}C_{j}y_{n-j}}\geq \frac{\sum_{i\in I_{\beta}}\beta_{i}x_{n-i} + \sum_{i\in I_{\gamma}}\gamma_{i}y_{n-i}}{\sum_{j\in I_{\beta}}B_{j}x_{n-j}+\sum_{j\in I_{\gamma}}C_{j}y_{n-j}}$$
$$\geq \frac{\min(\min_{i\in I_{\beta}}(\beta_{i}),\min_{i\in I_{\gamma}}(\gamma_{i}))}{\max(\max_{j\in I_{\beta}}(B_{j}),\max_{j\in I_{\gamma}}(C_{j}))}.$$
So in this case $\{x_{n}\}$ is bounded below by a constant.
Now suppose $p>0$, $I_{D}\subset I_{\delta}$, $I_{E}\subset I_{\epsilon}$, and $I_{C}\neq\emptyset$ then we have
$$y_{n}=\frac{p+\sum^{k}_{i=1}\delta_{i}x_{n-i} + \sum^{k}_{i=1}\epsilon_{i}y_{n-i}}{q+\sum^{k}_{j=1}D_{j}x_{n-j} + \sum^{k}_{j=1}E_{j}y_{n-j}}\geq \frac{p+\sum_{i\in I_{\delta}}\delta_{i}x_{n-i} + \sum_{i\in I_{\epsilon}}\epsilon_{i}y_{n-i}}{q+\sum_{j\in I_{\delta}}D_{j}x_{n-j}+\sum_{j\in I_{\epsilon}}E_{j}y_{n-j}}$$
$$\geq \frac{\min(p,\min_{i\in I_{\delta}}(\delta_{i}),\min_{i\in I_{\epsilon}}(\epsilon_{i}))}{\max(q,\max_{j\in I_{\delta}}(D_{j}),\max_{j\in I_{\epsilon}}(E_{j}))}.$$
So in this case $\{y_{n}\}$ is bounded below by a constant.
Now we have shown that there exists $A_{2}>0$, so that $\{x_{n}\}$ is bounded below by $A_{2}$ and $I_{B}\neq\emptyset$ or $\{y_{n}\}$ is bounded below by $A_{2}$ and $I_{C}\neq\emptyset$.
We use this fact to show the following.
$$x_n=\frac{\alpha+\sum^{k}_{i=1}\beta_{i}x_{n-i} + \sum^{k}_{i=1}\gamma_{i}y_{n-i}}{\sum^{k}_{j=1}B_{j}x_{n-j} + \sum^{k}_{j=1}C_{j}y_{n-j}} $$ $$\leq  \frac{\alpha+\sum^{k}_{i=1}\beta_{i}x_{n-i}+ \sum^{k}_{i=1}\gamma_{i}M_{2}}{\min_+(\sum^{k}_{j=1}B_{j},\sum^{k}_{j=1}C_{j})\frac{A_{2}}{2}+\frac{1}{2}\sum^{k}_{j=1}B_{j}x_{n-j} + \frac{1}{2}\sum^{k}_{j=1}C_{j}y_{n-j}}$$ $$\leq   \frac{\alpha+\sum^{k}_{i=1}\beta_{i}x_{n-i}+ \sum^{k}_{i=1}\gamma_{i}M_{2}}{\min_+(\sum^{k}_{j=1}B_{j},\sum^{k}_{j=1}C_{j})\frac{A_{2}}{2}+\frac{1}{2}\sum^{k}_{j=1}B_{j}x_{n-j}}.$$
Thus we see that the sequence $\{x_{n}\}^{\infty}_{n=1}$ satisfies the above difference inequality. Using Theorem 1 in \cite{fpinv} we see that there exists $M_{3} > 0$ and $N\in\mathbb{N}$ so that given any non-negative initial conditions, we have $x_{n}\leq M_{3}$ for all $n>N$.\newline
\end{proof}

For the next five results we assume that there exists $M_{1}>0$ and $M_{2}\geq 0$ so that $y_{n}\leq M_{1}x_{n} + M_{2}$ for all $n\in\mathbb{N}$. In the next result we use Theorem 1 of \cite{fpinv} to iterate with respect to $y_{n}$ and then apply Theorem 11.

\begin{thm}
Suppose that we have a $k^{th}$ order system of two rational
difference equations
$$x_n=\frac{\alpha+\sum^{k}_{i=1}\beta_{i}x_{n-i} + \sum^{k}_{i=1}\gamma_{i}y_{n-i}}{A+\sum^{k}_{j=1}B_{j}x_{n-j} + \sum^{k}_{j=1}C_{j}y_{n-j}},\quad n\in\mathbb{N},$$
$$y_n=\frac{p+\sum^{k}_{i=1}\delta_{i}x_{n-i} + \sum^{k}_{i=1}\epsilon_{i}y_{n-i}}{q+\sum^{k}_{j=1}D_{j}x_{n-j} + \sum^{k}_{j=1}E_{j}y_{n-j}},\quad n\in\mathbb{N},$$
with non-negative parameters and non-negative initial conditions. Assume that there exists $M_{1}>0$ and $M_{2}\geq 0$ so that $y_{n}\leq M_{1}x_{n} + M_{2}$ for all $n\in\mathbb{N}$.
Further assume that $A,q>0$ $I_{\delta}\subset I_{D}$, there exists a positive integer $\eta_1$, such that for every sequence $\{c_{m}\} ^{\infty}_{m=1} $ with $c_m\in I_{\epsilon}$ for $m=1,2, \dots $ there exists positive integers,  $
N_{1}
$,$
N_{2} \leq \eta_1
$ $ $, such that $ \sum^{N_{2}}_{m=N_{1}} c_{m} \in I_{D}\cup I_{E}$, and there exists a positive integer $\eta_2$, such that for every sequence $\{d_{m}\} ^{\infty}_{m=1} $ with $d_m\in I_{\beta}$ for $m=1,2, \dots $ there exists positive integers,  $
N_{3}
$,$
N_{4} \leq \eta_2
$ $ $, such that $ \sum^{N_{4}}_{m=N_{3}} d_{m} \in I_{B}$.
Then there exists $M > 0$ and $N\in\mathbb{N}$ so that given any non-negative initial conditions, we have $x_{n},y_{n}\leq M$ for all $n>N$.
\end{thm}

\begin{proof}
It suffices to prove that $\{y_{n}\}^{\infty}_{n=1}$ is bounded above then we simply apply Theorem 11 to obtain the result. Since $q>0$ and $I_{\delta}\subset I_{D}$ we have,
$$y_{n}= \frac{p+\sum^{k}_{i=1}\delta_{i}x_{n-i} + \sum^{k}_{i=1}\epsilon_{i}y_{n-i}}{q+\sum^{k}_{j=1}D_{j}x_{n-j} + \sum^{k}_{j=1}E_{j}y_{n-j}}\leq \frac{p}{q}+\sum_{i\in I_{\delta}}\frac{\delta_{i}}{D_{i}}+ \frac{\sum^{k}_{i=1}\epsilon_{i}y_{n-i}}{q+\sum^{k}_{j=1}D_{j}x_{n-j} + \sum^{k}_{j=1}E_{j}y_{n-j}}$$
$$\leq \frac{p}{q}+\sum_{i\in I_{\delta}}\frac{\delta_{i}}{D_{i}}+ \frac{\sum^{k}_{i=1}\epsilon_{i}y_{n-i}}{\frac{q}{2}+\sum^{k}_{j=1}F_{j}y_{n-j}+ \sum^{k}_{j=1}E_{j}y_{n-j}}.$$
Where $F_{j}=\min(\frac{D_{j}}{M_{1}},\frac{q}{2k(M_{2}+1)})$. Thus the sequence $\{y_{n}\}^{\infty}_{n=1}$ satisfies the above difference inequality. Using Theorem 1 in \cite{fpinv} we see that there exists $M > 0$ and $N\in\mathbb{N}$ so that given any non-negative initial conditions, we have $y_{n}\leq M$ for all $n>N$. Thus applying Theorem 11 we obtain the result.\newline
\end{proof}

For the following result we use Theorem 1 of \cite{fpinv} to iterate with respect to $x_{n}$ and then we use the inequality $y_{n}\leq M_{1}x_{n} + M_{2}$ for $M_{1}>0$ and $M_{2}\geq 0$.
\begin{thm}
Suppose that we have a $k^{th}$ order system of two rational
difference equations
$$x_n=\frac{\alpha+\sum^{k}_{i=1}\beta_{i}x_{n-i} + \sum^{k}_{i=1}\gamma_{i}y_{n-i}}{A+\sum^{k}_{j=1}B_{j}x_{n-j} + \sum^{k}_{j=1}C_{j}y_{n-j}},\quad n\in\mathbb{N},$$
$$y_n=\frac{p+\sum^{k}_{i=1}\delta_{i}x_{n-i} + \sum^{k}_{i=1}\epsilon_{i}y_{n-i}}{q+\sum^{k}_{j=1}D_{j}x_{n-j} + \sum^{k}_{j=1}E_{j}y_{n-j}},\quad n\in\mathbb{N},$$
with non-negative parameters and non-negative initial conditions. Assume that there exists $M_{1}>0$ and $M_{2}\geq 0$ so that $y_{n}\leq M_{1}x_{n} + M_{2}$ for all $n\in\mathbb{N}$.
Further assume that $A>0$ there exists a positive integer $\eta$, such that for every sequence $\{c_{m}\} ^{\infty}_{m=1} $ with $c_m\in I_{\beta}\cup(I_{\gamma}\setminus I_{C})$ for $m=1,2, \dots $ there exists positive integers,  $
N_{1}
$,$
N_{2} \leq \eta
$ $ $, such that $ \sum^{N_{2}}_{m=N_{1}} c_{m} \in I_{B}$.
Then there exists $M > 0$ and $N\in\mathbb{N}$ so that given any non-negative initial conditions, we have $x_{n},y_{n}\leq M$ for all $n>N$.
\end{thm}

\begin{proof}
Since we have assumed $y_{n}\leq M_{1}x_{n}+M_{2}$ it suffices to find a bound for $\{x_{n}\}^{\infty}_{n=1}$. Notice that
$$x_n=\frac{\alpha+\sum^{k}_{i=1}\beta_{i}x_{n-i} + \sum^{k}_{i=1}\gamma_{i}y_{n-i}}{A+\sum^{k}_{j=1}B_{j}x_{n-j} + \sum^{k}_{j=1}C_{j}y_{n-j}}\leq \frac{\alpha}{A}+\sum_{i\in I_{C}\cap I_{\gamma}}\frac{\gamma_{i}}{C_{i}}+\frac{\sum^{k}_{i=1}\beta_{i}x_{n-i} + \sum_{i\in I_{\gamma}\setminus I_{C}}\gamma_{i}y_{n-i}}{A+\sum^{k}_{j=1}B_{j}x_{n-j}}$$
$$\leq \frac{p}{q}+\sum_{i\in I_{C}\cap I_{\gamma}}\frac{\gamma_{i}}{C_{i}}+\frac{\sum_{i\in I_{\gamma}\setminus I_{C}}\gamma_{i}M_{2} + \sum^{k}_{i=1}\beta_{i}x_{n-i} + \sum_{i\in I_{\gamma}\setminus I_{C}}\gamma_{i}M_{1}x_{n-i}}{A+\sum^{k}_{j=1}B_{j}x_{n-j}}.$$
Thus the sequence $\{x_{n}\}^{\infty}_{n=1}$ satisfies the above difference inequality. Using Theorem 1 in \cite{fpinv} we see that there exists $M > 0$ and $N\in\mathbb{N}$ so that given any non-negative initial conditions, we have $x_{n}\leq M$ for all $n>N$. Since there exists $M_{1}>0$ and $M_{2}\geq 0$ so that $y_{n}\leq M_{1}x_{n} + M_{2}$ for all $n\in\mathbb{N}$, we obtain the full result.\newline
\end{proof}

For the next result we show that $x_{n}$ or $y_{n}$ is bounded below and then we use Theorem 1 of \cite{fpinv} to iterate with respect to $x_{n}$. Once we have shown $x_{n}$ is bounded above we use the fact that $y_{n}\leq M_{1}x_{n} + M_{2}$ for $M_{1}>0$ and $M_{2}\geq 0$ to obtain the result.

\begin{thm}
Suppose that we have a $k^{th}$ order system of two rational
difference equations
$$x_n=\frac{\alpha+\sum^{k}_{i=1}\beta_{i}x_{n-i} + \sum^{k}_{i=1}\gamma_{i}y_{n-i}}{A+\sum^{k}_{j=1}B_{j}x_{n-j} + \sum^{k}_{j=1}C_{j}y_{n-j}},\quad n\in\mathbb{N},$$
$$y_n=\frac{p+\sum^{k}_{i=1}\delta_{i}x_{n-i} + \sum^{k}_{i=1}\epsilon_{i}y_{n-i}}{q+\sum^{k}_{j=1}D_{j}x_{n-j} + \sum^{k}_{j=1}E_{j}y_{n-j}},\quad n\in\mathbb{N},$$
with non-negative parameters and non-negative initial conditions. Assume that there exists $M_{1}>0$ and $M_{2}\geq 0$ so that $y_{n}\leq M_{1}x_{n} + M_{2}$ for all $n\in\mathbb{N}$.
Further assume that $A=0$, one of the following holds
\begin{enumerate}[(i)]
\item $I_{B}\subset I_{\beta}$, $I_{C}\subset I_{\gamma}$ and $I_{B}\neq\emptyset$
\item $q=0$, $I_{D}\subset I_{\delta}$, $I_{E}\subset I_{\epsilon}$ and $I_{C}\neq\emptyset$
\item $p,q>0$, $I_{D}\subset I_{\delta}$, $I_{E}\subset I_{\epsilon}$ and $I_{C}\neq\emptyset$
\end{enumerate}
and there exists a positive integer $\eta$, such that for every sequence $\{c_{m}\} ^{\infty}_{m=1} $ with $c_m\in I_{\beta}\cup(I_{\gamma}\setminus I_{C})$ for $m=1,2, \dots $ there exists positive integers,  $
N_{1}
$,$
N_{2} \leq \eta
$ $ $, such that $ \sum^{N_{2}}_{m=N_{1}} c_{m} \in I_{B}$.
Then there exists $M > 0$ and $N\in\mathbb{N}$ so that given any non-negative initial conditions, we have $x_{n},y_{n}\leq M$ for all $n>N$.
\end{thm}

\begin{proof}
Suppose $I_{B}\subset I_{\beta}$, $I_{C}\subset I_{\gamma}$, and $I_{B}\neq\emptyset$ then we have
$$x_{n}=\frac{\alpha+\sum^{k}_{i=1}\beta_{i}x_{n-i} + \sum^{k}_{i=1}\gamma_{i}y_{n-i}}{\sum^{k}_{j=1}B_{j}x_{n-j} + \sum^{k}_{j=1}C_{j}y_{n-j}}\geq \frac{\sum_{i\in I_{\beta}}\beta_{i}x_{n-i} + \sum_{i\in I_{\gamma}}\gamma_{i}y_{n-i}}{\sum_{j\in I_{\beta}}B_{j}x_{n-j}+\sum_{j\in I_{\gamma}}C_{j}y_{n-j}}$$
$$\geq \frac{\min(\min_{i\in I_{\beta}}(\beta_{i}),\min_{i\in I_{\gamma}}(\gamma_{i}))}{\max(\max_{j\in I_{\beta}}(B_{j}),\max_{j\in I_{\gamma}}(C_{j}))}.$$
So in this case $\{x_{n}\}$ is bounded below by a constant.
Now suppose $q=0$, $I_{D}\subset I_{\delta}$, $I_{E}\subset I_{\epsilon}$, and $I_{C}\neq\emptyset$ then we have
$$y_{n}=\frac{p+\sum^{k}_{i=1}\delta_{i}x_{n-i} + \sum^{k}_{i=1}\epsilon_{i}y_{n-i}}{\sum^{k}_{j=1}D_{j}x_{n-j} + \sum^{k}_{j=1}E_{j}y_{n-j}}\geq \frac{\sum_{i\in I_{\delta}}\delta_{i}x_{n-i} + \sum_{i\in I_{\epsilon}}\epsilon_{i}y_{n-i}}{\sum_{j\in I_{\delta}}D_{j}x_{n-j}+\sum_{j\in I_{\epsilon}}E_{j}y_{n-j}}$$
$$\geq \frac{\min(\min_{i\in I_{\delta}}(\delta_{i}),\min_{i\in I_{\epsilon}}(\epsilon_{i}))}{\max(\max_{j\in I_{\delta}}(D_{j}),\max_{j\in I_{\epsilon}}(E_{j}))}.$$
So in this case $\{y_{n}\}$ is bounded below by a constant.
Now suppose $p,q>0$, $I_{D}\subset I_{\delta}$, $I_{E}\subset I_{\epsilon}$, and $I_{C}\neq\emptyset$ then we have
$$y_{n}=\frac{p+\sum^{k}_{i=1}\delta_{i}x_{n-i} + \sum^{k}_{i=1}\epsilon_{i}y_{n-i}}{q+\sum^{k}_{j=1}D_{j}x_{n-j} + \sum^{k}_{j=1}E_{j}y_{n-j}}\geq \frac{p+\sum_{i\in I_{\delta}}\delta_{i}x_{n-i} + \sum_{i\in I_{\epsilon}}\epsilon_{i}y_{n-i}}{q+\sum_{j\in I_{\delta}}D_{j}x_{n-j}+\sum_{j\in I_{\epsilon}}E_{j}y_{n-j}}$$
$$\geq \frac{\min(p,\min_{i\in I_{\delta}}(\delta_{i}),\min_{i\in I_{\epsilon}}(\epsilon_{i}))}{\max(q,\max_{j\in I_{\delta}}(D_{j}),\max_{j\in I_{\epsilon}}(E_{j}))}.$$
So in this case $\{y_{n}\}$ is bounded below by a constant.
Now we have shown that there exists $A_{2}>0$, so that $\{x_{n}\}$ is bounded below by $A_{2}$ and $I_{B}\neq\emptyset$ or $\{y_{n}\}$ is bounded below by $A_{2}$ and $I_{C}\neq\emptyset$.
We use this fact to show the following.
$$x_n=\frac{\alpha+\sum^{k}_{i=1}\beta_{i}x_{n-i} + \sum^{k}_{i=1}\gamma_{i}y_{n-i}}{\sum^{k}_{j=1}B_{j}x_{n-j} + \sum^{k}_{j=1}C_{j}y_{n-j}} \leq \frac{\sum_{i\in I_{\gamma}\cap I_{C}}\gamma_{i}y_{n-i}}{\sum^{k}_{j=1}B_{j}x_{n-j} + \sum^{k}_{j=1}C_{j}y_{n-j}}$$ $$ + \frac{\alpha+\sum^{k}_{i=1}\beta_{i}x_{n-i}+ \sum_{i\in I_{\gamma}\setminus I_{C}}\gamma_{i}y_{n-i}}{\min_+(\sum^{k}_{j=1}B_{j},\sum^{k}_{j=1}C_{j})\frac{A_{2}}{2}+\frac{1}{2}\sum^{k}_{j=1}B_{j}x_{n-j} + \frac{1}{2}\sum^{k}_{j=1}C_{j}y_{n-j}}$$ $$\leq \sum_{i\in I_{\gamma}\cap I_{C}}\frac{\gamma_{i}}{C_{i}} + \frac{\alpha+\sum_{i\in I_{\gamma}\setminus I_{C}}\gamma_{i}M_{2}+\sum^{k}_{i=1}\beta_{i}x_{n-i}+\sum_{i\in I_{\gamma}\setminus I_{C}}\gamma_{i}M_{1}x_{n-i}}{\min_+(\sum^{k}_{j=1}B_{j},\sum^{k}_{j=1}C_{j})\frac{A_{2}}{2}+\frac{1}{2}\sum^{k}_{j=1}B_{j}x_{n-j}}.$$
Thus we see that the sequence $\{x_{n}\}^{\infty}_{n=1}$ satisfies the above difference inequality. Using Theorem 1 in \cite{fpinv} we see that there exists $M > 0$ and $N\in\mathbb{N}$ so that given any non-negative initial conditions, we have $x_{n}\leq M$ for all $n>N$. Since there exists $M_{1}>0$ and $M_{2}\geq 0$ so that $y_{n}\leq M_{1}x_{n} + M_{2}$ for all $n\in\mathbb{N}$, we have the result.\newline
\end{proof}

In the following theorem, we first use Theorem 1 from \cite{fpinv} to iterate with respect to $y_{n}$. We then use the fact that $y_{n}$ is bounded above along with our assumptions to show that $x_{n}$ or $y_{n}$
is bounded below. We then use these facts along with Theorem 1 from \cite{fpinv} to iterate with respect to $x_{n}$.

\begin{thm}
Suppose that we have a $k^{th}$ order system of two rational
difference equations
$$x_n=\frac{\alpha+\sum^{k}_{i=1}\beta_{i}x_{n-i} + \sum^{k}_{i=1}\gamma_{i}y_{n-i}}{A+\sum^{k}_{j=1}B_{j}x_{n-j} + \sum^{k}_{j=1}C_{j}y_{n-j}},\quad n\in\mathbb{N},$$
$$y_n=\frac{p+\sum^{k}_{i=1}\delta_{i}x_{n-i} + \sum^{k}_{i=1}\epsilon_{i}y_{n-i}}{q+\sum^{k}_{j=1}D_{j}x_{n-j} + \sum^{k}_{j=1}E_{j}y_{n-j}},\quad n\in\mathbb{N},$$
with non-negative parameters and non-negative initial conditions. Assume that there exists $M_{1}>0$ and $M_{2}\geq 0$ so that $y_{n}\leq M_{1}x_{n} + M_{2}$ for all $n\in\mathbb{N}$.
Further assume that $A=0$,$q>0$, one of the following holds
\begin{enumerate}[(i)]
\item $I_{B}\subset I_{\beta}$, $I_{C}\subset I_{\gamma}$ and $I_{B}\neq\emptyset$
\item $p>0$, $I_{D}\subset I_{\delta}$, $I_{E}\subset I_{\epsilon}$ and $I_{C}\neq\emptyset$
\end{enumerate}
$I_{\delta}\subset I_{D}$, there exists a positive integer $\eta_1$, such that for every sequence $\{c_{m}\} ^{\infty}_{m=1} $ with $c_m\in I_{\epsilon}$ for $m=1,2, \dots $ there exists positive integers,  $
N_{1}
$,$
N_{2} \leq \eta_1
$ $ $, such that $ \sum^{N_{2}}_{m=N_{1}} c_{m} \in I_{D}\cup I_{E}$, and there exists a positive integer $\eta_2$, such that for every sequence $\{d_{m}\} ^{\infty}_{m=1} $ with $d_m\in I_{\beta}$ for $m=1,2, \dots $ there exists positive integers,  $
N_{3}
$,$
N_{4} \leq \eta_2
$ $ $, such that $ \sum^{N_{4}}_{m=N_{3}} d_{m} \in I_{B}$.
Then there exists $M > 0$ and $N\in\mathbb{N}$ so that given any non-negative initial conditions, we have $x_{n},y_{n}\leq M$ for all $n>N$.
\end{thm}

\begin{proof}
First we must prove that $\{y_{n}\}^{\infty}_{n=1}$ is bounded above. Since $q>0$ and $I_{\delta}\subset I_{D}$ we have,
$$y_{n}= \frac{p+\sum^{k}_{i=1}\delta_{i}x_{n-i} + \sum^{k}_{i=1}\epsilon_{i}y_{n-i}}{q+\sum^{k}_{j=1}D_{j}x_{n-j} + \sum^{k}_{j=1}E_{j}y_{n-j}}\leq \frac{p}{q}+\sum_{i\in I_{\delta}}\frac{\delta_{i}}{D_{i}}+ \frac{\sum^{k}_{i=1}\epsilon_{i}y_{n-i}}{q+\sum^{k}_{j=1}D_{j}x_{n-j} + \sum^{k}_{j=1}E_{j}y_{n-j}}$$
$$\leq \frac{p}{q}+\sum_{i\in I_{\delta}}\frac{\delta_{i}}{D_{i}}+ \frac{\sum^{k}_{i=1}\epsilon_{i}y_{n-i}}{\frac{q}{2}+\sum^{k}_{j=1}F_{j}y_{n-j} + \sum^{k}_{j=1}E_{j}y_{n-j}},$$
where $F_{j}=\min(\frac{D_{j}}{M_{1}},\frac{q}{2k(M_{2}+1)})$. Thus the sequence $\{y_{n}\}^{\infty}_{n=1}$ satisfies the above difference inequality. Using Theorem 1 in \cite{fpinv} we see that there exists $M_{3} > 0$ and $N\in\mathbb{N}$ so that given any non-negative initial conditions, we have $y_{n}\leq M_{3}$ for all $n>N$. Let us show that $\{x_{n}\}$ or $\{y_{n}\}$ is bounded below by a constant $A_{2}$.\newline
Suppose $I_{B}\subset I_{\beta}$, $I_{C}\subset I_{\gamma}$, and $I_{B}\neq\emptyset$ then we have
$$x_{n}=\frac{\alpha+\sum^{k}_{i=1}\beta_{i}x_{n-i} + \sum^{k}_{i=1}\gamma_{i}y_{n-i}}{\sum^{k}_{j=1}B_{j}x_{n-j} + \sum^{k}_{j=1}C_{j}y_{n-j}}\geq \frac{\sum_{i\in I_{\beta}}\beta_{i}x_{n-i} + \sum_{i\in I_{\gamma}}\gamma_{i}y_{n-i}}{\sum_{j\in I_{\beta}}B_{j}x_{n-j}+\sum_{j\in I_{\gamma}}C_{j}y_{n-j}}$$
$$\geq \frac{\min(\min_{i\in I_{\beta}}(\beta_{i}),\min_{i\in I_{\gamma}}(\gamma_{i}))}{\max(\max_{j\in I_{\beta}}(B_{j}),\max_{j\in I_{\gamma}}(C_{j}))}.$$
So in this case $\{x_{n}\}$ is bounded below by a constant.
Now suppose $p>0$, $I_{D}\subset I_{\delta}$, $I_{E}\subset I_{\epsilon}$, and $I_{C}\neq\emptyset$ then we have
$$y_{n}=\frac{p+\sum^{k}_{i=1}\delta_{i}x_{n-i} + \sum^{k}_{i=1}\epsilon_{i}y_{n-i}}{q+\sum^{k}_{j=1}D_{j}x_{n-j} + \sum^{k}_{j=1}E_{j}y_{n-j}}\geq \frac{p+\sum_{i\in I_{\delta}}\delta_{i}x_{n-i} + \sum_{i\in I_{\epsilon}}\epsilon_{i}y_{n-i}}{q+\sum_{j\in I_{\delta}}D_{j}x_{n-j}+\sum_{j\in I_{\epsilon}}E_{j}y_{n-j}}$$
$$\geq \frac{\min(p,\min_{i\in I_{\delta}}(\delta_{i}),\min_{i\in I_{\epsilon}}(\epsilon_{i}))}{\max(q,\max_{j\in I_{\delta}}(D_{j}),\max_{j\in I_{\epsilon}}(E_{j}))}.$$
So in this case $\{y_{n}\}$ is bounded below by a constant.
Now we have shown that there exists $A_{2}>0$, so that $\{x_{n}\}$ is bounded below by $A_{2}$ and $I_{B}\neq\emptyset$ or $\{y_{n}\}$ is bounded below by $A_{2}$ and $I_{C}\neq\emptyset$.
We use this fact to show the following.
$$x_n=\frac{\alpha+\sum^{k}_{i=1}\beta_{i}x_{n-i} + \sum^{k}_{i=1}\gamma_{i}y_{n-i}}{\sum^{k}_{j=1}B_{j}x_{n-j} + \sum^{k}_{j=1}C_{j}y_{n-j}} $$ $$\leq  \frac{\alpha+\sum^{k}_{i=1}\beta_{i}x_{n-i}+ \sum^{k}_{i=1}\gamma_{i}M_{3}}{\min_+(\sum^{k}_{j=1}B_{j},\sum^{k}_{j=1}C_{j})\frac{A_{2}}{2}+\frac{1}{2}\sum^{k}_{j=1}B_{j}x_{n-j} + \frac{1}{2}\sum^{k}_{j=1}C_{j}y_{n-j}}$$ $$\leq   \frac{\alpha+\sum^{k}_{i=1}\beta_{i}x_{n-i}+ \sum^{k}_{i=1}\gamma_{i}M_{3}}{\min_+(\sum^{k}_{j=1}B_{j},\sum^{k}_{j=1}C_{j})\frac{A_{2}}{2}+\frac{1}{2}\sum^{k}_{j=1}B_{j}x_{n-j}}.$$
Thus we see that the sequence $\{x_{n}\}^{\infty}_{n=1}$ satisfies the above difference inequality. Using Theorem 1 in \cite{fpinv} we see that there exists $M_{4} > 0$ and $N\in\mathbb{N}$ so that given any non-negative initial conditions, we have $x_{n}\leq M_{4}$ for all $n>N$.\newline
\end{proof}

In the following theorem, we first show that $x_{n}$ or $y_{n}$
is bounded below. We then use Theorem 1 from \cite{fpinv} to iterate with respect to $y_{n}$. We then use the fact that $y_{n}$ is bounded above and we use the fact that $x_{n}$ or $y_{n}$
is bounded below along with Theorem 1 from \cite{fpinv} to iterate with respect to $x_{n}$.

\begin{thm}
Suppose that we have a $k^{th}$ order system of two rational
difference equations
$$x_n=\frac{\alpha+\sum^{k}_{i=1}\beta_{i}x_{n-i} + \sum^{k}_{i=1}\gamma_{i}y_{n-i}}{A+\sum^{k}_{j=1}B_{j}x_{n-j} + \sum^{k}_{j=1}C_{j}y_{n-j}},\quad n\in\mathbb{N},$$
$$y_n=\frac{p+\sum^{k}_{i=1}\delta_{i}x_{n-i} + \sum^{k}_{i=1}\epsilon_{i}y_{n-i}}{q+\sum^{k}_{j=1}D_{j}x_{n-j} + \sum^{k}_{j=1}E_{j}y_{n-j}},\quad n\in\mathbb{N},$$
with non-negative parameters and non-negative initial conditions. Assume that there exists $M_{1}>0$ and $M_{2}\geq 0$ so that $y_{n}\leq M_{1}x_{n} + M_{2}$ for all $n\in\mathbb{N}$.
Further assume that $A=0$,$q=0$, one of the following holds
\begin{enumerate}[(i)]
\item $I_{B}\subset I_{\beta}$, $I_{C}\subset I_{\gamma}$ and $I_{B},I_{D}\neq\emptyset$
\item $I_{D}\subset I_{\delta}$, $I_{E}\subset I_{\epsilon}$ and $I_{C},I_{E}\neq\emptyset$
\end{enumerate}
$I_{\delta}\subset I_{D}$, there exists a positive integer $\eta_1$, such that for every sequence $\{c_{m}\} ^{\infty}_{m=1} $ with $c_m\in I_{\epsilon}$ for $m=1,2, \dots $ there exists positive integers,  $
N_{1}
$,$
N_{2} \leq \eta_1
$ $ $, such that $ \sum^{N_{2}}_{m=N_{1}} c_{m} \in I_{D}\cup I_{E}$, and there exists a positive integer $\eta_2$, such that for every sequence $\{d_{m}\} ^{\infty}_{m=1} $ with $d_m\in I_{\beta}$ for $m=1,2, \dots $ there exists positive integers,  $
N_{3}
$,$
N_{4} \leq \eta_2
$ $ $, such that $ \sum^{N_{4}}_{m=N_{3}} d_{m} \in I_{B}$.
Then there exists $M > 0$ and $N\in\mathbb{N}$ so that given any non-negative initial conditions, we have $x_{n},y_{n}\leq M$ for all $n>N$.
\end{thm}

\begin{proof}
Suppose $I_{B}\subset I_{\beta}$, $I_{C}\subset I_{\gamma}$ and $I_{B},I_{D}\neq\emptyset$ then we have
$$x_{n}=\frac{\alpha+\sum^{k}_{i=1}\beta_{i}x_{n-i} + \sum^{k}_{i=1}\gamma_{i}y_{n-i}}{\sum^{k}_{j=1}B_{j}x_{n-j} + \sum^{k}_{j=1}C_{j}y_{n-j}}\geq \frac{\sum_{i\in I_{\beta}}\beta_{i}x_{n-i} + \sum_{i\in I_{\gamma}}\gamma_{i}y_{n-i}}{\sum_{j\in I_{\beta}}B_{j}x_{n-j}+\sum_{j\in I_{\gamma}}C_{j}y_{n-j}}$$
$$\geq \frac{\min(\min_{i\in I_{\beta}}(\beta_{i}),\min_{i\in I_{\gamma}}(\gamma_{i}))}{\max(\max_{j\in I_{\beta}}(B_{j}),\max_{j\in I_{\gamma}}(C_{j}))}.$$
So in this case $\{x_{n}\}$ is bounded below by a constant.
Now suppose $I_{D}\subset I_{\delta}$, $I_{E}\subset I_{\epsilon}$ and $I_{C},I_{E}\neq\emptyset$ then we have
$$y_{n}=\frac{p+\sum^{k}_{i=1}\delta_{i}x_{n-i} + \sum^{k}_{i=1}\epsilon_{i}y_{n-i}}{\sum^{k}_{j=1}D_{j}x_{n-j} + \sum^{k}_{j=1}E_{j}y_{n-j}}\geq \frac{\sum_{i\in I_{\delta}}\delta_{i}x_{n-i} + \sum_{i\in I_{\epsilon}}\epsilon_{i}y_{n-i}}{\sum_{j\in I_{\delta}}D_{j}x_{n-j}+\sum_{j\in I_{\epsilon}}E_{j}y_{n-j}}$$
$$\geq \frac{\min(\min_{i\in I_{\delta}}(\delta_{i}),\min_{i\in I_{\epsilon}}(\epsilon_{i}))}{\max(\max_{j\in I_{\delta}}(D_{j}),\max_{j\in I_{\epsilon}}(E_{j}))}.$$
So in this case $\{y_{n}\}$ is bounded below by a constant.
Now we have shown that there exists $A_{2}>0$, so that $\{x_{n}\}$ is bounded below by $A_{2}$ and $I_{D}\neq\emptyset$ or $\{y_{n}\}$ is bounded below by $A_{2}$ and $I_{E}\neq\emptyset$.
We use this fact to show that $\{y_{n}\}$ is bounded above by a constant $M_{3}$.
$$y_{n}= \frac{p+\sum^{k}_{i=1}\delta_{i}x_{n-i} + \sum^{k}_{i=1}\epsilon_{i}y_{n-i}}{\sum^{k}_{j=1}D_{j}x_{n-j} + \sum^{k}_{j=1}E_{j}y_{n-j}}$$ $$\leq \frac{p}{\min_+(\sum^{k}_{j=1}D_{j},\sum^{k}_{j=1}E_{j})A_{2}}+\sum_{i\in I_{\delta}}\frac{\delta_{i}}{D_{i}}$$ $$+ 2\frac{\sum^{k}_{i=1}\epsilon_{i}y_{n-i}}{\min_+(\sum^{k}_{j=1}D_{j},\sum^{k}_{j=1}E_{j})A_{2}+\sum^{k}_{j=1}D_{j}x_{n-j} + \sum^{k}_{j=1}E_{j}y_{n-j}}$$
$$\leq \frac{p}{\min_+(\sum^{k}_{j=1}D_{j},\sum^{k}_{j=1}E_{j})A_{2}}+\sum_{i\in I_{\delta}}\frac{\delta_{i}}{D_{i}}$$ $$+ 2\frac{\sum^{k}_{i=1}\epsilon_{i}y_{n-i}}{\min_+(\sum^{k}_{j=1}D_{j},\sum^{k}_{j=1}E_{j})\frac{A_{2}}{2}+\sum^{k}_{j=1}F_{j}y_{n-j} + \sum^{k}_{j=1}E_{j}y_{n-j}},$$
where $F_{j}=\min(\frac{D_{j}}{M_{1}},\frac{\min_+(\sum^{k}_{j=1}D_{j},\sum^{k}_{j=1}E_{j})A_{2}}{2k(M_{2}+1)})$. Thus the sequence $\{y_{n}\}^{\infty}_{n=1}$ satisfies the above difference inequality. Using Theorem 1 in \cite{fpinv} we see that there exists $M_{3} > 0$ and $N\in\mathbb{N}$ so that given any non-negative initial conditions, we have $y_{n}\leq M_{3}$ for all $n>N$.\newline
Now we have shown that there exists $A_{2}>0$, so that $\{x_{n}\}$ is bounded below by $A_{2}$ and $I_{B}\neq\emptyset$ or $\{y_{n}\}$ is bounded below by $A_{2}$ and $I_{C}\neq\emptyset$.
Finally we use this fact to show that $\{x_{n}\}$ is bounded above by a constant $M_{4}$.
$$x_n=\frac{\alpha+\sum^{k}_{i=1}\beta_{i}x_{n-i} + \sum^{k}_{i=1}\gamma_{i}y_{n-i}}{\sum^{k}_{j=1}B_{j}x_{n-j} + \sum^{k}_{j=1}C_{j}y_{n-j}} $$ $$\leq  \frac{\alpha+\sum^{k}_{i=1}\beta_{i}x_{n-i}+ \sum^{k}_{i=1}\gamma_{i}M_{3}}{\min_+(\sum^{k}_{j=1}B_{j},\sum^{k}_{j=1}C_{j})\frac{A_{2}}{2}+\frac{1}{2}\sum^{k}_{j=1}B_{j}x_{n-j} + \frac{1}{2}\sum^{k}_{j=1}C_{j}y_{n-j}}$$ $$\leq   \frac{\alpha+\sum^{k}_{i=1}\beta_{i}x_{n-i}+ \sum^{k}_{i=1}\gamma_{i}M_{3}}{\min_+(\sum^{k}_{j=1}B_{j},\sum^{k}_{j=1}C_{j})\frac{A_{2}}{2}+\frac{1}{2}\sum^{k}_{j=1}B_{j}x_{n-j}}.$$
Thus we see that the sequence $\{x_{n}\}^{\infty}_{n=1}$ satisfies the above difference inequality. Using Theorem 1 in \cite{fpinv} we see that there exists $M_{4} > 0$ and $N\in\mathbb{N}$ so that given any non-negative initial conditions, we have $x_{n}\leq M_{4}$ for all $n>N$.\newline
\end{proof}

\section{A Comparison on Both Sides Involving Constants}

For the results in this section, we assume that there exists $M_{1},M_{3}>0$ and $M_{4}\geq M_{2}\geq 0$ so that $x_{n}\leq M_{1}y_{n}+M_{2}\leq M_{3}x_{n}+M_{4}$ for all $n\in\mathbb{N}$. We use this comparison and Theorem 1 of \cite{fpinv} to prove that every solution is bounded for certain systems of rational difference equations.
For the following result we use Theorem 1 of \cite{fpinv} to iterate with respect to $x_{n}$ and then we use the fact that there exists $M_{1},M_{3}>0$ and $M_{4}\geq M_{2}\geq 0$ so that $x_{n}\leq M_{1}y_{n}+M_{2}\leq M_{3}x_{n}+M_{4}$ for all $n\in\mathbb{N}$.

\begin{thm}
 Suppose that we have a $k^{th}$ order system of two rational
difference equations
$$x_n=\frac{\alpha+\sum^{k}_{i=1}\beta_{i}x_{n-i} + \sum^{k}_{i=1}\gamma_{i}y_{n-i}}{A+\sum^{k}_{j=1}B_{j}x_{n-j} + \sum^{k}_{j=1}C_{j}y_{n-j}},\quad n\in\mathbb{N},$$
$$y_n=\frac{p+\sum^{k}_{i=1}\delta_{i}x_{n-i} + \sum^{k}_{i=1}\epsilon_{i}y_{n-i}}{q+\sum^{k}_{j=1}D_{j}x_{n-j} + \sum^{k}_{j=1}E_{j}y_{n-j}},\quad n\in\mathbb{N},$$
with non-negative parameters and non-negative initial conditions.
Further assume that there exists $M_{1},M_{3}>0$ and $M_{4}\geq M_{2}\geq 0$ so that $x_{n}\leq M_{1}y_{n}+M_{2}\leq M_{3}x_{n}+M_{4}$ for all $n\in\mathbb{N}$, and suppose that $A>0$ and there exists a positive integer $\eta$, such that for every sequence $\{c_{m}\} ^{\infty}_{m=1} $ with $c_m\in I_{\beta}\cup I_{\gamma}$ for $m=1,2, \dots $ there exists positive integers,  $
N_{1}
$,$
N_{2} \leq \eta
$ $ $, such that $ \sum^{N_{2}}_{m=N_{1}} c_{m} \in I_{B}\cup I_{C}$.
Then there exists $M > 0$ and $N\in\mathbb{N}$ so that given any non-negative initial conditions, we have $x_{n},y_{n}\leq M$ for all $n>N$.
\end{thm}

\begin{proof}
 We have
$$x_{n}=\frac{\alpha+\sum^{k}_{i=1}\beta_{i}x_{n-i} + \sum^{k}_{i=1}\gamma_{i}y_{n-i}}{A+\sum^{k}_{j=1}B_{j}x_{n-j} + \sum^{k}_{j=1}C_{j}y_{n-j}}$$
$$\leq \frac{\alpha+\sum^{k}_{i=1}\beta_{i}x_{n-i} + (\frac{M_{3}}{M_{1}})\sum^{k}_{i=1}\gamma_{i}x_{n-i}+(\frac{M_{4}-M_{2}}{M_{1}})\sum^{k}_{i=1}\gamma_{i}}{\frac{A}{2}+\sum^{k}_{j=1}B_{j}x_{n-j} + \sum^{k}_{j=1}F_{j}x_{n-j}},$$
where $F_{j}=\min(\frac{C_{j}}{M_{1}},\frac{A}{2k(M_{2}+1)})$. Thus we see that the sequence $\{x_{n}\}^{\infty}_{n=1}$ satisfies the above difference inequality. Moreover using Theorem 1 in \cite{fpinv} we see that there exists $M > 0$ and $N\in\mathbb{N}$ so that given any non-negative initial conditions, we have $x_{n}\leq M$ for all $n>N$. Since there exists $M_{1},M_{3}>0$ and $M_{4}\geq M_{2}\geq 0$ so that $x_{n}\leq M_{1}y_{n}+M_{2}\leq M_{3}x_{n}+M_{4}$ for all $n\in\mathbb{N}$, we have the full result.\newline

\end{proof}

For the following theorem we use Theorem 1 of \cite{fpinv} to iterate with respect to $y_{n}$ and then we use the fact that there exists $M_{1},M_{3}>0$ and $M_{4}\geq M_{2}\geq 0$ so that $x_{n}\leq M_{1}y_{n}+M_{2}\leq M_{3}x_{n}+M_{4}$ for all $n\in\mathbb{N}$.

\begin{thm}
  Suppose that we have a $k^{th}$ order system of two rational
difference equations
$$x_n=\frac{\alpha+\sum^{k}_{i=1}\beta_{i}x_{n-i} + \sum^{k}_{i=1}\gamma_{i}y_{n-i}}{A+\sum^{k}_{j=1}B_{j}x_{n-j} + \sum^{k}_{j=1}C_{j}y_{n-j}},\quad n\in\mathbb{N},$$
$$y_n=\frac{p+\sum^{k}_{i=1}\delta_{i}x_{n-i} + \sum^{k}_{i=1}\epsilon_{i}y_{n-i}}{q+\sum^{k}_{j=1}D_{j}x_{n-j} + \sum^{k}_{j=1}E_{j}y_{n-j}},\quad n\in\mathbb{N},$$
with non-negative parameters and non-negative initial conditions.
Further assume that there exists $M_{1},M_{3}>0$ and $M_{4}\geq M_{2}\geq 0$ so that $x_{n}\leq M_{1}y_{n}+M_{2}\leq M_{3}x_{n}+M_{4}$ for all $n\in\mathbb{N}$, and suppose that $q>0$ and there exists a positive integer $\eta$, such that for every sequence $\{c_{m}\} ^{\infty}_{m=1} $ with $c_m\in I_{\delta}\cup I_{\epsilon}$ for $m=1,2, \dots $ there exists positive integers,  $
N_{1}
$,$
N_{2} \leq \eta
$ $ $, such that $ \sum^{N_{2}}_{m=N_{1}} c_{m} \in I_{D}\cup I_{E}$.
Then there exists $M > 0$ and $N\in\mathbb{N}$ so that given any non-negative initial conditions, we have $x_{n},y_{n}\leq M$ for all $n>N$.
\end{thm}

\begin{proof}
  We have
$$y_{n}=\frac{p+\sum^{k}_{i=1}\delta_{i}x_{n-i} + \sum^{k}_{i=1}\epsilon_{i}y_{n-i}}{q+\sum^{k}_{j=1}D_{j}x_{n-j} + \sum^{k}_{j=1}E_{j}y_{n-j}} \leq \frac{p+\sum^{k}_{i=1}\delta_{i}M_{2}+\sum^{k}_{i=1}\delta_{i}M_{1}y_{n-i} + \sum^{k}_{i=1}\epsilon_{i}y_{n-i}}{\frac{q}{2}+\sum^{k}_{j=1}G_{j}y_{n-j} + \sum^{k}_{j=1}E_{j}y_{n-j}},$$
where $G_{j}=\min(\frac{D_{j}M_{1}}{M_{3}},\frac{qM_{1}}{2k(M_{4}-M_{2}+1)})$. Thus we see that the sequence $\{y_{n}\}^{\infty}_{n=1}$ satisfies the above difference inequality. Again using Theorem 1 in \cite{fpinv} we see that there exists $M > 0$ and $N\in\mathbb{N}$ so that given any non-negative initial conditions, we have $y_{n}\leq M$ for all $n>N$. Since there exists $M_{1},M_{3}>0$ and $M_{4}\geq M_{2}\geq 0$ so that $x_{n}\leq M_{1}y_{n}+M_{2}\leq M_{3}x_{n}+M_{4}$ for all $n\in\mathbb{N}$, we have the full result.\newline

\end{proof}

For the following result we first show that $x_{n}$ or $y_{n}$ is bounded below by a constant. Then we use Theorem 1 of \cite{fpinv} to iterate with respect to $x_{n}$. The fact that there exists $M_{1},M_{3}>0$ and $M_{4}\geq M_{2}\geq 0$ so that $x_{n}\leq M_{1}y_{n}+M_{2}\leq M_{3}x_{n}+M_{4}$ for all $n\in\mathbb{N}$ gives us the full result.

\begin{thm}
  Suppose that we have a $k^{th}$ order system of two rational
difference equations
$$x_n=\frac{\alpha+\sum^{k}_{i=1}\beta_{i}x_{n-i} + \sum^{k}_{i=1}\gamma_{i}y_{n-i}}{A+\sum^{k}_{j=1}B_{j}x_{n-j} + \sum^{k}_{j=1}C_{j}y_{n-j}},\quad n\in\mathbb{N},$$
$$y_n=\frac{p+\sum^{k}_{i=1}\delta_{i}x_{n-i} + \sum^{k}_{i=1}\epsilon_{i}y_{n-i}}{q+\sum^{k}_{j=1}D_{j}x_{n-j} + \sum^{k}_{j=1}E_{j}y_{n-j}},\quad n\in\mathbb{N},$$
with non-negative parameters and non-negative initial conditions.
Further assume that there exists $M_{1},M_{3}>0$ and $M_{4}> M_{2}> 0$ so that $x_{n}\leq M_{1}y_{n}+M_{2}\leq M_{3}x_{n}+M_{4}$ for all $n\in\mathbb{N}$, and suppose that $A=0$ and one of the following holds
\begin{enumerate}[(i)]
\item $I_{B}\cup I_{C}\subset I_{\beta}\cup I_{\gamma}$, $\alpha >0$, and $I_{B}\neq\emptyset$
\item $I_{D}\cup I_{E}\subset I_{\delta}\cup I_{\epsilon}$, $p>0$, and $I_{C}\neq\emptyset$
\end{enumerate}
and there exists a positive integer $\eta$, such that for every sequence $\{c_{m}\} ^{\infty}_{m=1} $ with $c_m\in I_{\beta}\cup I_{\gamma}$ for $m=1,2, \dots $ there exists positive integers,  $
N_{1}
$,$
N_{2} \leq \eta
$ $ $, such that $ \sum^{N_{2}}_{m=N_{1}} c_{m} \in I_{B}\cup I_{C}$.
Then there exists $M > 0$ and $N\in\mathbb{N}$ so that given any non-negative initial conditions, we have $x_{n},y_{n}\leq M$ for all $n>N$.
\end{thm}

\begin{proof}
 Suppose that $I_{B}\cup I_{C}\subset I_{\beta}\cup I_{\gamma}$, $\alpha > 0$, and $I_{B}\neq\emptyset$ then we have
$$x_{n}=\frac{\alpha+\sum^{k}_{i=1}\beta_{i}x_{n-i} + \sum^{k}_{i=1}\gamma_{i}y_{n-i}}{\sum^{k}_{j=1}B_{j}x_{n-j} + \sum^{k}_{j=1}C_{j}y_{n-j}}\geq \frac{\frac{\alpha}{2} + \sum_{i\in I_{\beta}}H_{i}y_{n-i} + \sum_{i\in I_{\gamma}}\gamma_{i}y_{n-i}}{\sum_{j\in I_{\beta}\cup I_{\gamma}}(B_{j}M_{2}+B_{j}M_{1}y_{n-j}+C_{j}y_{n-j})}$$
$$\geq \frac{\min(\frac{\alpha}{2},\min_{i\in I_{\beta}}(H_{i}),\min_{i\in I_{\gamma}}(\gamma_{i}))}{\sum_{j\in I_{\beta}\cup I_{\gamma}}(B_{j}M_{2}+B_{j}M_{1}+C_{j})},$$
where $H_{i}=\min(\frac{\beta_{i}M_{1}}{M_{3}},\frac{\alpha M_{1}}{2k(M_{4}-M_{2}+1)})$. So in this case $\{x_{n}\}$ is bounded below by a constant.
Now suppose that $I_{D}\cup I_{E}\subset I_{\delta}\cup I_{\epsilon}$, $p>0$, and $I_{C}\neq\emptyset$ then we have
$$y_{n}=\frac{p+\sum^{k}_{i=1}\delta_{i}x_{n-i} + \sum^{k}_{i=1}\epsilon_{i}y_{n-i}}{q+\sum^{k}_{j=1}D_{j}x_{n-j} + \sum^{k}_{j=1}E_{j}y_{n-j}}\geq \frac{\frac{p}{2}+\sum_{i\in I_{\delta}}S_{i}y_{n-i} + \sum_{i\in I_{\epsilon}}\epsilon_{i}y_{n-i}}{q+\sum_{j\in I_{\delta}\cup I_{\epsilon}}(D_{j}M_{2}+D_{j}M_{1}y_{n-j}+ E_{j}y_{n-j})}$$
$$\geq \frac{\min(\frac{p}{2},\min_{i\in I_{\delta}}(S_{i}),\min_{i\in I_{\epsilon}}(\epsilon_{i}))}{\sum_{j\in I_{\delta}\cup I_{\epsilon}}(q+D_{j}M_{2}+D_{j}M_{1}+E_{j})},$$
where $S_{i}=\min(\frac{\delta_{i}M_{1}}{M_{3}},\frac{p M_{1}}{2k(M_{4}-M_{2}+1)})$. So in this case $\{y_{n}\}$ is bounded below by a constant.
Now we have shown that there exists $A_{2}>0$, so that $\{x_{n}\}$ is bounded below by $A_{2}$ and $I_{B}\neq\emptyset$ or $\{y_{n}\}$ is bounded below by $A_{2}$ and $I_{C}\neq\emptyset$.
We use this fact to show the following.
$$x_{n}=\frac{\alpha+\sum^{k}_{i=1}\beta_{i}x_{n-i} + \sum^{k}_{i=1}\gamma_{i}y_{n-i}}{\sum^{k}_{j=1}B_{j}x_{n-j} + \sum^{k}_{j=1}C_{j}y_{n-j}}$$
$$\leq 2\frac{\alpha+\sum^{k}_{i=1}\beta_{i}x_{n-i} + \sum^{k}_{i=1}\gamma_{i}(\frac{M_{3}x_{n-i}+M_{4}-M_{2}}{M_{1}})}{\min_+(\sum^{k}_{j=1}B_{j},\sum^{k}_{j=1}C_{j})A_{2}+\sum^{k}_{j=1}B_{j}x_{n-j} + \sum^{k}_{j=1}C_{j}y_{n-j}}$$
$$\leq 2\frac{\alpha+\sum^{k}_{i=1}\beta_{i}x_{n-i} + \sum^{k}_{i=1}\gamma_{i}(\frac{M_{3}x_{n-i}+M_{4}-M_{2}}{M_{1}})}{\min_+(\sum^{k}_{j=1}B_{j},\sum^{k}_{j=1}C_{j})\frac{A_{2}}{2}+\sum^{k}_{j=1}B_{j}x_{n-j} + \sum^{k}_{j=1}L_{j}x_{n-j}},$$
where $L_{j}=\min(\frac{C_{j}}{M_{1}},\frac{\min_+(\sum^{k}_{j=1}B_{j},\sum^{k}_{j=1}C_{j})A_{2}}{2k(M_{2}+1)})$. Thus we see that the sequence $\{x_{n}\}^{\infty}_{n=1}$ satisfies the above difference inequality. Moreover using Theorem 1 in \cite{fpinv} we see that there exists $M > 0$ and $N\in\mathbb{N}$ so that given any non-negative initial conditions, we have $x_{n}\leq M$ for all $n>N$. Since there exists $M_{1},M_{3}>0$ and $M_{4}> M_{2}> 0$ so that $x_{n}\leq M_{1}y_{n}+M_{2}\leq M_{3}x_{n}+M_{4}$ for all $n\in\mathbb{N}$, we have the full result.\newline

\end{proof}

For the following theorem we first show that $x_{n}$ or $y_{n}$ is bounded below by a constant. Then we use Theorem 1 of \cite{fpinv} to iterate with respect to $y_{n}$. The fact that there exists $M_{1},M_{3}>0$ and $M_{4}\geq M_{2}\geq 0$ so that $x_{n}\leq M_{1}y_{n}+M_{2}\leq M_{3}x_{n}+M_{4}$ for all $n\in\mathbb{N}$ gives us the full result.

\begin{thm}
  Suppose that we have a $k^{th}$ order system of two rational
difference equations
$$x_n=\frac{\alpha+\sum^{k}_{i=1}\beta_{i}x_{n-i} + \sum^{k}_{i=1}\gamma_{i}y_{n-i}}{A+\sum^{k}_{j=1}B_{j}x_{n-j} + \sum^{k}_{j=1}C_{j}y_{n-j}},\quad n\in\mathbb{N},$$
$$y_n=\frac{p+\sum^{k}_{i=1}\delta_{i}x_{n-i} + \sum^{k}_{i=1}\epsilon_{i}y_{n-i}}{q+\sum^{k}_{j=1}D_{j}x_{n-j} + \sum^{k}_{j=1}E_{j}y_{n-j}},\quad n\in\mathbb{N},$$
with non-negative parameters and non-negative initial conditions.
Further assume that there exists $M_{1},M_{3}>0$ and $M_{4}> M_{2}> 0$ so that $x_{n}\leq M_{1}y_{n}+M_{2}\leq M_{3}x_{n}+M_{4}$ for all $n\in\mathbb{N}$, and suppose that $q=0$ and one of the following holds
\begin{enumerate}[(i)]
\item $I_{B}\cup I_{C}\subset I_{\beta}\cup I_{\gamma}$, $\alpha >0$, and $I_{D}\neq\emptyset$
\item $I_{D}\cup I_{E}\subset I_{\delta}\cup I_{\epsilon}$, $p>0$, and $I_{E}\neq\emptyset$
\end{enumerate}
and there exists a positive integer $\eta$, such that for every sequence $\{c_{m}\} ^{\infty}_{m=1} $ with $c_m\in I_{\delta}\cup I_{\epsilon}$ for $m=1,2, \dots $ there exists positive integers,  $
N_{1}
$,$
N_{2} \leq \eta
$ $ $, such that $ \sum^{N_{2}}_{m=N_{1}} c_{m} \in I_{D}\cup I_{E}$.
Then there exists $M > 0$ and $N\in\mathbb{N}$ so that given any non-negative initial conditions, we have $x_{n},y_{n}\leq M$ for all $n>N$.
\end{thm}

\begin{proof}
 Suppose that $I_{B}\cup I_{C}\subset I_{\beta}\cup I_{\gamma}$, $\alpha > 0$, and $I_{D}\neq\emptyset$ then we have
$$x_{n}=\frac{\alpha+\sum^{k}_{i=1}\beta_{i}x_{n-i} + \sum^{k}_{i=1}\gamma_{i}y_{n-i}}{A+\sum^{k}_{j=1}B_{j}x_{n-j} + \sum^{k}_{j=1}C_{j}y_{n-j}}\geq \frac{\frac{\alpha}{2} + \sum_{i\in I_{\beta}}H_{i}y_{n-i} + \sum_{i\in I_{\gamma}}\gamma_{i}y_{n-i}}{A+\sum_{j\in I_{\beta}\cup I_{\gamma}}(B_{j}M_{2}+B_{j}M_{1}y_{n-j}+C_{j}y_{n-j})}$$
$$\geq \frac{\min(\frac{\alpha}{2},\min_{i\in I_{\beta}}(H_{i}),\min_{i\in I_{\gamma}}(\gamma_{i}))}{\sum_{j\in I_{\beta}\cup I_{\gamma}}(A+B_{j}M_{2}+B_{j}M_{1}+C_{j})},$$
where $H_{i}=\min(\frac{\beta_{i}M_{1}}{M_{3}},\frac{\alpha M_{1}}{2k(M_{4}-M_{2}+1)})$. So in this case $\{x_{n}\}$ is bounded below by a constant.
Now suppose that $I_{D}\cup I_{E}\subset I_{\delta}\cup I_{\epsilon}$, $p>0$, and $I_{E}\neq\emptyset$ then we have
$$y_{n}=\frac{p+\sum^{k}_{i=1}\delta_{i}x_{n-i} + \sum^{k}_{i=1}\epsilon_{i}y_{n-i}}{\sum^{k}_{j=1}D_{j}x_{n-j} + \sum^{k}_{j=1}E_{j}y_{n-j}}\geq \frac{\frac{p}{2}+\sum_{i\in I_{\delta}}S_{i}y_{n-i} + \sum_{i\in I_{\epsilon}}\epsilon_{i}y_{n-i}}{\sum_{j\in I_{\delta}\cup I_{\epsilon}}(D_{j}M_{2}+D_{j}M_{1}y_{n-j}+ E_{j}y_{n-j})}$$
$$\geq \frac{\min(\frac{p}{2},\min_{i\in I_{\delta}}(S_{i}),\min_{i\in I_{\epsilon}}(\epsilon_{i}))}{\sum_{j\in I_{\delta}\cup I_{\epsilon}}(D_{j}M_{2}+D_{j}M_{1}+E_{j})},$$
where $S_{i}=\min(\frac{\delta_{i}M_{1}}{M_{3}},\frac{p M_{1}}{2k(M_{4}-M_{2}+1)})$. So in this case $\{y_{n}\}$ is bounded below by a constant.
Now we have shown that there exists $q_{2}>0$, so that $\{x_{n}\}$ is bounded below by $q_{2}$ and $I_{D}\neq\emptyset$ or $\{y_{n}\}$ is bounded below by $q_{2}$ and $I_{E}\neq\emptyset$.
We use this fact to show the following.
$$y_{n}=\frac{p+\sum^{k}_{i=1}\delta_{i}x_{n-i} + \sum^{k}_{i=1}\epsilon_{i}y_{n-i}}{\sum^{k}_{j=1}D_{j}x_{n-j} + \sum^{k}_{j=1}E_{j}y_{n-j}}$$
$$\leq 2\frac{p+\sum^{k}_{i=1}\delta_{i}M_{1}y_{n-i} + \sum^{k}_{i=1}\epsilon_{i}y_{n-i}+\sum^{k}_{i=1}\delta_{i}M_{2}}{\min_+(\sum^{k}_{j=1}D_{j},\sum^{k}_{j=1}E_{j})q_{2}+\sum^{k}_{j=1}D_{j}x_{n-j} + \sum^{k}_{j=1}E_{j}y_{n-j}}$$
$$\leq 2\frac{p+\sum^{k}_{i=1}\delta_{i}M_{1}y_{n-i} + \sum^{k}_{i=1}\epsilon_{i}y_{n-i}+\sum^{k}_{i=1}\delta_{i}M_{2}}{\min_+(\sum^{k}_{j=1}D_{j},\sum^{k}_{j=1}E_{j})\frac{q_{2}}{2}+\sum^{k}_{j=1}R_{j}y_{n-j} + \sum^{k}_{j=1}E_{j}y_{n-j}},$$
where $R_{j}=\min(\frac{D_{j}M_{1}}{M_{3}},\frac{\min_+(\sum^{k}_{j=1}D_{j},\sum^{k}_{j=1}E_{j})q_{2}M_{1}}{2k(M_{4}-M_{2}+1)})$. Thus we see that the sequence $\{y_{n}\}^{\infty}_{n=1}$ satisfies the above difference inequality. Again using Theorem 1 in \cite{fpinv} we see that there exists $M > 0$ and $N\in\mathbb{N}$ so that given any non-negative initial conditions, we have $y_{n}\leq M$ for all $n>N$.
Since there exists $M_{1},M_{3}>0$ and $M_{4}> M_{2}> 0$ so that $x_{n}\leq M_{1}y_{n}+M_{2}\leq M_{3}x_{n}+M_{4}$ for all $n\in\mathbb{N}$, we have the full result.\newline

\end{proof}

\section{Conditions of Comparability}

Since the comparisons presented above provide us with such useful tools to study boundedness we devote this section to determining whether a system of two equations satisfies any of the necessary inequalities.

\begin{thm}
Suppose that we have a $k^{th}$ order system of two rational
difference equations
$$x_n=\frac{\alpha+\sum^{k}_{i=1}\beta_{i}x_{n-i} + \sum^{k}_{i=1}\gamma_{i}y_{n-i}}{A+\sum^{k}_{j=1}B_{j}x_{n-j} + \sum^{k}_{j=1}C_{j}y_{n-j}},\quad n\in\mathbb{N},$$
$$y_n=\frac{p+\sum^{k}_{i=1}\delta_{i}x_{n-i} + \sum^{k}_{i=1}\epsilon_{i}y_{n-i}}{q+\sum^{k}_{j=1}D_{j}x_{n-j} + \sum^{k}_{j=1}E_{j}y_{n-j}},\quad n\in\mathbb{N},$$
with non-negative parameters and non-negative initial conditions.
Further suppose that $ I_{\delta}\subset I_{\beta}$, $I_{B}\subset I_{D}$, $ I_{\epsilon}\subset I_{\gamma}$, $I_{C}\subset I_{E}$. Also assume that whenever  $A>0$, then $q>0$, and whenever $p > 0$, then $\alpha >0$. Then there exists $M_{1}>0$ so that $y_{n}\leq M_{1}x_{n}$ for all $n\in\mathbb{N}$.
\end{thm}

\begin{proof}
First notice that it suffices to show that eventually for $n\geq N$ there exists $M>0$ so that $y_{n}\leq Mx_{n}$. This is since we may take $M_{1}=\max(M, \max_{n\in L}\frac{y_{n}}{x_{n}})$ , where $L=\{n\in\mathbb{N}|x_{n}\neq 0$ and  $n<N \}$. Notice that since $ I_{\delta}\subset I_{\beta}$, $ I_{\epsilon}\subset I_{\gamma}$, and whenever $p>0$, then $\alpha >0$ we get that whenever $x_{n}=0$, then $y_{n}=0$. Thus
$y_{n}\leq M_{1}x_{n}$ for all $n\in\mathbb{N}$ with $x_{n}=0$. Now let us prove that the inequality eventually holds. In the case where $A>0$ and $p>0$ we have
$$y_n=\frac{p+\sum^{k}_{i=1}\delta_{i}x_{n-i} + \sum^{k}_{i=1}\epsilon_{i}y_{n-i}}{q+\sum^{k}_{j=1}D_{j}x_{n-j} + \sum^{k}_{j=1}E_{j}y_{n-j}} \leq $$
$$\left(\frac{\max(p,\max_{i\in I_{\delta}}(\delta_{i}),\max_{i\in I_{\epsilon}}(\epsilon_{i}))}{\min(q,\min_{j\in I_{D}}(D_{j}),\min_{j\in I_{E}}(E_{j}))} \right)\frac{1+\sum_{i\in I_{\delta}}x_{n-i} +\sum_{i\in I_{\epsilon}}y_{n-i}}{1+\sum_{j\in I_{D}}x_{n-j} + \sum_{j\in I_{E}}y_{n-j}}\leq$$
$$\left(\frac{\max(p,\max_{i\in I_{\delta}}(\delta_{i}),\max_{i\in I_{\epsilon}}(\epsilon_{i}))}{\min(q,\min_{j\in I_{D}}(D_{j}),\min_{j\in I_{E}}(E_{j}))} \right)\frac{1+\sum_{i\in I_{\beta}}x_{n-i} +\sum_{i\in I_{\gamma}}y_{n-i}}{1+\sum_{j\in I_{B}}x_{n-j} + \sum_{j\in I_{C}}y_{n-j}}\leq$$
$$\left(\frac{(\max(A,\max_{j\in I_{B}}(B_{j}),\max_{j\in I_{C}}(C_{j})))(\max(p,\max_{i\in I_{\delta}}(\delta_{i}),\max_{i\in I_{\epsilon}}(\epsilon_{i})))}{(\min(\alpha,\min_{i\in I_{\beta}}(\beta_{i}),\min_{i\in I_{\gamma}}(\gamma_{i})))(\min(q,\min_{j\in I_{D}}(D_{j}),\min_{j\in I_{E}}(E_{j})))} \right)x_{n}.$$
In the case where $A>0$, $p=0$, and $\alpha >0$ we have
$$y_n=\frac{\sum^{k}_{i=1}\delta_{i}x_{n-i} + \sum^{k}_{i=1}\epsilon_{i}y_{n-i}}{q+\sum^{k}_{j=1}D_{j}x_{n-j} + \sum^{k}_{j=1}E_{j}y_{n-j}} \leq $$
$$\left(\frac{\max(\max_{i\in I_{\delta}}(\delta_{i}),\max_{i\in I_{\epsilon}}(\epsilon_{i}))}{\min(q,\min_{j\in I_{D}}(D_{j}),\min_{j\in I_{E}}(E_{j}))} \right)\frac{\sum_{i\in I_{\delta}}x_{n-i} +\sum_{i\in I_{\epsilon}}y_{n-i}}{1+\sum_{j\in I_{D}}x_{n-j} + \sum_{j\in I_{E}}y_{n-j}}\leq$$
$$\left(\frac{\max(\max_{i\in I_{\delta}}(\delta_{i}),\max_{i\in I_{\epsilon}}(\epsilon_{i}))}{\min(q,\min_{j\in I_{D}}(D_{j}),\min_{j\in I_{E}}(E_{j}))} \right)\frac{1+ \sum_{i\in I_{\beta}}x_{n-i} +\sum_{i\in I_{\gamma}}y_{n-i}}{1+\sum_{j\in I_{B}}x_{n-j} + \sum_{j\in I_{C}}y_{n-j}}\leq$$
$$\left(\frac{(\max(A,\max_{j\in I_{B}}(B_{j}),\max_{j\in I_{C}}(C_{j})))(\max(\max_{i\in I_{\delta}}(\delta_{i}),\max_{i\in I_{\epsilon}}(\epsilon_{i})))}{(\min(\alpha, \min_{i\in I_{\beta}}(\beta_{i}),\min_{i\in I_{\gamma}}(\gamma_{i})))(\min(q,\min_{j\in I_{D}}(D_{j}),\min_{j\in I_{E}}(E_{j})))} \right)x_{n}.$$
In the case where $A>0$, $p=0$, and $\alpha =0$ we have
$$y_n=\frac{\sum^{k}_{i=1}\delta_{i}x_{n-i} + \sum^{k}_{i=1}\epsilon_{i}y_{n-i}}{q+\sum^{k}_{j=1}D_{j}x_{n-j} + \sum^{k}_{j=1}E_{j}y_{n-j}} \leq $$
$$\left(\frac{\max(\max_{i\in I_{\delta}}(\delta_{i}),\max_{i\in I_{\epsilon}}(\epsilon_{i}))}{\min(q,\min_{j\in I_{D}}(D_{j}),\min_{j\in I_{E}}(E_{j}))} \right)\frac{\sum_{i\in I_{\delta}}x_{n-i} +\sum_{i\in I_{\epsilon}}y_{n-i}}{1+\sum_{j\in I_{D}}x_{n-j} + \sum_{j\in I_{E}}y_{n-j}}\leq$$
$$\left(\frac{\max(\max_{i\in I_{\delta}}(\delta_{i}),\max_{i\in I_{\epsilon}}(\epsilon_{i}))}{\min(q,\min_{j\in I_{D}}(D_{j}),\min_{j\in I_{E}}(E_{j}))} \right)\frac{\sum_{i\in I_{\beta}}x_{n-i} +\sum_{i\in I_{\gamma}}y_{n-i}}{1+\sum_{j\in I_{B}}x_{n-j} + \sum_{j\in I_{C}}y_{n-j}}\leq$$
$$\left(\frac{(\max(A,\max_{j\in I_{B}}(B_{j}),\max_{j\in I_{C}}(C_{j})))(\max(\max_{i\in I_{\delta}}(\delta_{i}),\max_{i\in I_{\epsilon}}(\epsilon_{i})))}{(\min(\min_{i\in I_{\beta}}(\beta_{i}),\min_{i\in I_{\gamma}}(\gamma_{i})))(\min(q,\min_{j\in I_{D}}(D_{j}),\min_{j\in I_{E}}(E_{j})))} \right)x_{n}.$$
In the case where $A=0$, $q>0$, and $p>0$ we have
$$y_n=\frac{p+\sum^{k}_{i=1}\delta_{i}x_{n-i} + \sum^{k}_{i=1}\epsilon_{i}y_{n-i}}{q+\sum^{k}_{j=1}D_{j}x_{n-j} + \sum^{k}_{j=1}E_{j}y_{n-j}} \leq $$
$$\left(\frac{\max(p,\max_{i\in I_{\delta}}(\delta_{i}),\max_{i\in I_{\epsilon}}(\epsilon_{i}))}{\min(q,\min_{j\in I_{D}}(D_{j}),\min_{j\in I_{E}}(E_{j}))} \right)\frac{1+\sum_{i\in I_{\delta}}x_{n-i} +\sum_{i\in I_{\epsilon}}y_{n-i}}{1+\sum_{j\in I_{D}}x_{n-j} + \sum_{j\in I_{E}}y_{n-j}}\leq$$
$$\left(\frac{\max(p,\max_{i\in I_{\delta}}(\delta_{i}),\max_{i\in I_{\epsilon}}(\epsilon_{i}))}{\min(q,\min_{j\in I_{D}}(D_{j}),\min_{j\in I_{E}}(E_{j}))} \right)\frac{1+\sum_{i\in I_{\beta}}x_{n-i} +\sum_{i\in I_{\gamma}}y_{n-i}}{\sum_{j\in I_{B}}x_{n-j} + \sum_{j\in I_{C}}y_{n-j}}\leq$$
$$\left(\frac{(\max(\max_{j\in I_{B}}(B_{j}),\max_{j\in I_{C}}(C_{j})))(\max(p,\max_{i\in I_{\delta}}(\delta_{i}),\max_{i\in I_{\epsilon}}(\epsilon_{i})))}{(\min(\alpha,\min_{i\in I_{\beta}}(\beta_{i}),\min_{i\in I_{\gamma}}(\gamma_{i})))(\min(q, \min_{j\in I_{D}}(D_{j}),\min_{j\in I_{E}}(E_{j})))} \right)x_{n}.$$
In the case where $A=0$, $q>0$, $p=0$, and $\alpha = 0$ we have
$$y_n=\frac{\sum^{k}_{i=1}\delta_{i}x_{n-i} + \sum^{k}_{i=1}\epsilon_{i}y_{n-i}}{q+\sum^{k}_{j=1}D_{j}x_{n-j} + \sum^{k}_{j=1}E_{j}y_{n-j}} \leq $$
$$\left(\frac{\max(\max_{i\in I_{\delta}}(\delta_{i}),\max_{i\in I_{\epsilon}}(\epsilon_{i}))}{\min(q,\min_{j\in I_{D}}(D_{j}),\min_{j\in I_{E}}(E_{j}))} \right)\frac{\sum_{i\in I_{\delta}}x_{n-i} +\sum_{i\in I_{\epsilon}}y_{n-i}}{1+\sum_{j\in I_{D}}x_{n-j} + \sum_{j\in I_{E}}y_{n-j}}\leq$$
$$\left(\frac{\max(\max_{i\in I_{\delta}}(\delta_{i}),\max_{i\in I_{\epsilon}}(\epsilon_{i}))}{\min(q,\min_{j\in I_{D}}(D_{j}),\min_{j\in I_{E}}(E_{j}))} \right)\frac{\sum_{i\in I_{\beta}}x_{n-i} +\sum_{i\in I_{\gamma}}y_{n-i}}{\sum_{j\in I_{B}}x_{n-j} + \sum_{j\in I_{C}}y_{n-j}}\leq$$
$$\left(\frac{(\max(\max_{j\in I_{B}}(B_{j}),\max_{j\in I_{C}}(C_{j})))(\max(\max_{i\in I_{\delta}}(\delta_{i}),\max_{i\in I_{\epsilon}}(\epsilon_{i})))}{(\min(\min_{i\in I_{\beta}}(\beta_{i}),\min_{i\in I_{\gamma}}(\gamma_{i})))(\min(q, \min_{j\in I_{D}}(D_{j}),\min_{j\in I_{E}}(E_{j})))} \right)x_{n}.$$
In the case where $A=0$, $q=0$, and $p>0$ we have
$$y_n=\frac{p+\sum^{k}_{i=1}\delta_{i}x_{n-i} + \sum^{k}_{i=1}\epsilon_{i}y_{n-i}}{\sum^{k}_{j=1}D_{j}x_{n-j} + \sum^{k}_{j=1}E_{j}y_{n-j}} \leq $$
$$\left(\frac{\max(p,\max_{i\in I_{\delta}}(\delta_{i}),\max_{i\in I_{\epsilon}}(\epsilon_{i}))}{\min(\min_{j\in I_{D}}(D_{j}),\min_{j\in I_{E}}(E_{j}))} \right)\frac{1+\sum_{i\in I_{\delta}}x_{n-i} +\sum_{i\in I_{\epsilon}}y_{n-i}}{\sum_{j\in I_{D}}x_{n-j} + \sum_{j\in I_{E}}y_{n-j}}\leq$$
$$\left(\frac{\max(p,\max_{i\in I_{\delta}}(\delta_{i}),\max_{i\in I_{\epsilon}}(\epsilon_{i}))}{\min(\min_{j\in I_{D}}(D_{j}),\min_{j\in I_{E}}(E_{j}))} \right)\frac{1+\sum_{i\in I_{\beta}}x_{n-i} +\sum_{i\in I_{\gamma}}y_{n-i}}{\sum_{j\in I_{B}}x_{n-j} + \sum_{j\in I_{C}}y_{n-j}}\leq$$
$$\left(\frac{(\max(\max_{j\in I_{B}}(B_{j}),\max_{j\in I_{C}}(C_{j})))(\max(p,\max_{i\in I_{\delta}}(\delta_{i}),\max_{i\in I_{\epsilon}}(\epsilon_{i})))}{(\min(\alpha,\min_{i\in I_{\beta}}(\beta_{i}),\min_{i\in I_{\gamma}}(\gamma_{i})))(\min( \min_{j\in I_{D}}(D_{j}),\min_{j\in I_{E}}(E_{j})))} \right)x_{n}.$$
In the case where $A=0$, $q=0$, $p=0$, and $\alpha =0 $ we have
$$y_n=\frac{\sum^{k}_{i=1}\delta_{i}x_{n-i} + \sum^{k}_{i=1}\epsilon_{i}y_{n-i}}{\sum^{k}_{j=1}D_{j}x_{n-j} + \sum^{k}_{j=1}E_{j}y_{n-j}} \leq $$
$$\left(\frac{\max(\max_{i\in I_{\delta}}(\delta_{i}),\max_{i\in I_{\epsilon}}(\epsilon_{i}))}{\min(\min_{j\in I_{D}}(D_{j}),\min_{j\in I_{E}}(E_{j}))} \right)\frac{\sum_{i\in I_{\delta}}x_{n-i} +\sum_{i\in I_{\epsilon}}y_{n-i}}{\sum_{j\in I_{D}}x_{n-j} + \sum_{j\in I_{E}}y_{n-j}}\leq$$
$$\left(\frac{\max(\max_{i\in I_{\delta}}(\delta_{i}),\max_{i\in I_{\epsilon}}(\epsilon_{i}))}{\min(\min_{j\in I_{D}}(D_{j}),\min_{j\in I_{E}}(E_{j}))} \right)\frac{\sum_{i\in I_{\beta}}x_{n-i} +\sum_{i\in I_{\gamma}}y_{n-i}}{\sum_{j\in I_{B}}x_{n-j} + \sum_{j\in I_{C}}y_{n-j}}\leq$$
$$\left(\frac{(\max(\max_{j\in I_{B}}(B_{j}),\max_{j\in I_{C}}(C_{j})))(\max(\max_{i\in I_{\delta}}(\delta_{i}),\max_{i\in I_{\epsilon}}(\epsilon_{i})))}{(\min(\min_{i\in I_{\beta}}(\beta_{i}),\min_{i\in I_{\gamma}}(\gamma_{i})))(\min( \min_{j\in I_{D}}(D_{j}),\min_{j\in I_{E}}(E_{j})))} \right)x_{n}.$$
In the case where $A=0$, $q>0$, $p=0$, and $\alpha > 0$ we have
$$y_n=\frac{\sum^{k}_{i=1}\delta_{i}x_{n-i} + \sum^{k}_{i=1}\epsilon_{i}y_{n-i}}{q+\sum^{k}_{j=1}D_{j}x_{n-j} + \sum^{k}_{j=1}E_{j}y_{n-j}} \leq $$
$$\left(\frac{\max(\max_{i\in I_{\delta}}(\delta_{i}),\max_{i\in I_{\epsilon}}(\epsilon_{i}))}{\min(q,\min_{j\in I_{D}}(D_{j}),\min_{j\in I_{E}}(E_{j}))} \right)\frac{\sum_{i\in I_{\delta}}x_{n-i} +\sum_{i\in I_{\epsilon}}y_{n-i}}{1+\sum_{j\in I_{D}}x_{n-j} + \sum_{j\in I_{E}}y_{n-j}}\leq$$
$$\left(\frac{\max(\max_{i\in I_{\delta}}(\delta_{i}),\max_{i\in I_{\epsilon}}(\epsilon_{i}))}{\min(q,\min_{j\in I_{D}}(D_{j}),\min_{j\in I_{E}}(E_{j}))} \right)\frac{1+ \sum_{i\in I_{\beta}}x_{n-i} +\sum_{i\in I_{\gamma}}y_{n-i}}{\sum_{j\in I_{B}}x_{n-j} + \sum_{j\in I_{C}}y_{n-j}}\leq$$
$$\left(\frac{(\max(\max_{j\in I_{B}}(B_{j}),\max_{j\in I_{C}}(C_{j})))(\max(\max_{i\in I_{\delta}}(\delta_{i}),\max_{i\in I_{\epsilon}}(\epsilon_{i})))}{(\min(\alpha, \min_{i\in I_{\beta}}(\beta_{i}),\min_{i\in I_{\gamma}}(\gamma_{i})))(\min(q, \min_{j\in I_{D}}(D_{j}),\min_{j\in I_{E}}(E_{j})))} \right)x_{n}.$$
In the case where $A=0$, $q=0$, $p=0$, and $\alpha >0 $ we have
$$y_n=\frac{\sum^{k}_{i=1}\delta_{i}x_{n-i} + \sum^{k}_{i=1}\epsilon_{i}y_{n-i}}{\sum^{k}_{j=1}D_{j}x_{n-j} + \sum^{k}_{j=1}E_{j}y_{n-j}} \leq $$
$$\left(\frac{\max(\max_{i\in I_{\delta}}(\delta_{i}),\max_{i\in I_{\epsilon}}(\epsilon_{i}))}{\min(\min_{j\in I_{D}}(D_{j}),\min_{j\in I_{E}}(E_{j}))} \right)\frac{\sum_{i\in I_{\delta}}x_{n-i} +\sum_{i\in I_{\epsilon}}y_{n-i}}{\sum_{j\in I_{D}}x_{n-j} + \sum_{j\in I_{E}}y_{n-j}}\leq$$
$$\left(\frac{\max(\max_{i\in I_{\delta}}(\delta_{i}),\max_{i\in I_{\epsilon}}(\epsilon_{i}))}{\min(\min_{j\in I_{D}}(D_{j}),\min_{j\in I_{E}}(E_{j}))} \right)\frac{1+ \sum_{i\in I_{\beta}}x_{n-i} +\sum_{i\in I_{\gamma}}y_{n-i}}{\sum_{j\in I_{B}}x_{n-j} + \sum_{j\in I_{C}}y_{n-j}}\leq$$
$$\left(\frac{(\max(\max_{j\in I_{B}}(B_{j}),\max_{j\in I_{C}}(C_{j})))(\max(\max_{i\in I_{\delta}}(\delta_{i}),\max_{i\in I_{\epsilon}}(\epsilon_{i})))}{(\min(\alpha, \min_{i\in I_{\beta}}(\beta_{i}),\min_{i\in I_{\gamma}}(\gamma_{i})))(\min( \min_{j\in I_{D}}(D_{j}),\min_{j\in I_{E}}(E_{j})))} \right)x_{n}.$$

\end{proof}

\begin{thm}
Suppose that we have a $k^{th}$ order system of two rational
difference equations
$$x_n=\frac{\alpha+\sum^{k}_{i=1}\beta_{i}x_{n-i} + \sum^{k}_{i=1}\gamma_{i}y_{n-i}}{A+\sum^{k}_{j=1}B_{j}x_{n-j} + \sum^{k}_{j=1}C_{j}y_{n-j}},\quad n\in\mathbb{N},$$
$$y_n=\frac{p+\sum^{k}_{i=1}\delta_{i}x_{n-i} + \sum^{k}_{i=1}\epsilon_{i}y_{n-i}}{q+\sum^{k}_{j=1}D_{j}x_{n-j} + \sum^{k}_{j=1}E_{j}y_{n-j}},\quad n\in\mathbb{N},$$
with non-negative parameters and non-negative initial conditions.
Further suppose that $I_{\beta}= I_{\delta}$, $I_{B}= I_{D}$, $I_{\gamma}= I_{\epsilon}$, $I_{C}= I_{E}$, $\alpha>0$ if and only if $p>0$, and $A > 0$ if and only if $q>0$. Then there exists constants $M_{1},M_{2} > 0$ so that $M_{1}y_{n}\leq x_{n} \leq M_{2}y_{n}$ for all $n\in\mathbb{N}$.
\end{thm}

\begin{proof}
First notice that Theorem 24 applies to this system. This gives us $M_{3}>0$ so that $y_{n}\leq M_{3}x_{n}$ for all $n\in\mathbb{N}$. Moreover after a very simple change of variables Theorem 24 applies again. The change of variables we refer to here comes from renaming $x_{n}$ as $y_{n}$, $\beta_{i}$ as $\epsilon_{i}$, $B_{i}$ as $E_{i}$, $\gamma_{i}$ as $\delta_{i}$, $C_{i}$ as $D_{i}$, $\alpha$ as $p$, $A$ as $q$, and vice versa.
This gives us $M_{2}>0$ so that $x_{n}\leq M_{2}y_{n}$ for all $n\in\mathbb{N}$. Choose $M_{1}=\frac{1}{M_{3}}$ and we get that there exists constants $M_{1},M_{2} > 0$ so that $M_{1}y_{n}\leq x_{n} \leq M_{2}y_{n}$ for all $n\in\mathbb{N}$.
\end{proof}

\begin{thm}
Suppose that we have a $k^{th}$ order system of two rational
difference equations
$$x_n=\frac{\alpha+\sum^{k}_{i=1}\beta_{i}x_{n-i} + \sum^{k}_{i=1}\gamma_{i}y_{n-i}}{A+\sum^{k}_{j=1}B_{j}x_{n-j} + \sum^{k}_{j=1}C_{j}y_{n-j}},\quad n\in\mathbb{N},$$
$$y_n=\frac{p+\sum^{k}_{i=1}\delta_{i}x_{n-i} + \sum^{k}_{i=1}\epsilon_{i}y_{n-i}}{q+\sum^{k}_{j=1}D_{j}x_{n-j} + \sum^{k}_{j=1}E_{j}y_{n-j}},\quad n\in\mathbb{N},$$
with non-negative parameters and non-negative initial conditions.
Further suppose that $ I_{\delta}\subset I_{\beta}\cup I_{B}$, $I_{B}\subset I_{D}$, $ I_{\epsilon}\subset I_{\gamma}\cup I_{C}$, $I_{C}\subset I_{E}$. Also assume that whenever  $A>0$, then $q>0$, and whenever $p > 0$, then $\alpha >0$ or $A>0$. Then there exists $M_{1}>0$ and $M_{2}\geq 0$ so that $y_{n}\leq M_{1}x_{n} + M_{2}$ for all $n\in\mathbb{N}$.
\end{thm}

\begin{proof}
First notice that it suffices to show that eventually for $n\geq N$ there exists $M_{3}>0$ and $M_{4}\geq 0$ so that $y_{n}\leq M_{3}x_{n} + M_{4}$. This is since we may take $M_{1}=M_{3}$ and $M_{2}=\max(M_{4}, \max_{n<N}y_{n})$ and so
there exists $M_{1}>0$ and $M_{2}\geq 0$ so that $y_{n}\leq M_{1}x_{n} + M_{2}$ for all $n\in\mathbb{N}$.  In the case where $A>0$ and $p>0$ we have
$$y_n=\frac{p+\sum^{k}_{i=1}\delta_{i}x_{n-i} + \sum^{k}_{i=1}\epsilon_{i}y_{n-i}}{q+\sum^{k}_{j=1}D_{j}x_{n-j} + \sum^{k}_{j=1}E_{j}y_{n-j}} \leq $$
$$\left(\frac{\max(p,\max_{i\in I_{\delta}}(\delta_{i}),\max_{i\in I_{\epsilon}}(\epsilon_{i}))}{\min(q,\min_{j\in I_{D}}(D_{j}),\min_{j\in I_{E}}(E_{j}))} \right)\frac{1+\sum_{i\in I_{\delta}}x_{n-i} +\sum_{i\in I_{\epsilon}}y_{n-i}}{1+\sum_{j\in I_{D}}x_{n-j} + \sum_{j\in I_{E}}y_{n-j}}\leq$$
$$\left(\frac{\max(p,\max_{i\in I_{\delta}}(\delta_{i}),\max_{i\in I_{\epsilon}}(\epsilon_{i}))}{\min(q,\min_{j\in I_{D}}(D_{j}),\min_{j\in I_{E}}(E_{j}))} \right)\frac{1+\sum_{i\in I_{\beta}\cup I_{B}}x_{n-i} +\sum_{i\in I_{\gamma}\cup I_{C}}y_{n-i}}{1+\sum_{j\in I_{B}}x_{n-j} + \sum_{j\in I_{C}}y_{n-j}}\leq$$
$$M_{1} \left(\frac{\alpha+A+\sum^{k}_{i=1}(\beta_{i}+B_{i})x_{n-i}+\sum^{k}_{i=1}(\gamma_{i}+C_{i})y_{n-i}}{A+\sum^{k}_{j=1}B_{j}x_{n-j} + \sum^{k}_{j=1}C_{j}y_{n-j}} \right) = M_1(x_{n}+1),$$
where $$M_1=\frac{(\max(A,\max_{j\in I_{B}}(B_{j}),\max_{j\in I_{C}}(C_{j})))(\max(p,\max_{i\in I_{\delta}}(\delta_{i}),\max_{i\in I_{\epsilon}}(\epsilon_{i})))}{(\min(\alpha +A,\min_{i\in I_{\beta}\cup I_{B}}(\beta_{i}+ B_{i}),\min_{i\in I_{\gamma}\cup I_{C}}(\gamma_{i}+C_{i})))(\min(q,\min_{j\in I_{D}}(D_{j}),\min_{j\in I_{E}}(E_{j})))}.$$ In the case where $A>0$ and $p=0$ we have
$$y_n=\frac{\sum^{k}_{i=1}\delta_{i}x_{n-i} + \sum^{k}_{i=1}\epsilon_{i}y_{n-i}}{q+\sum^{k}_{j=1}D_{j}x_{n-j} + \sum^{k}_{j=1}E_{j}y_{n-j}} \leq $$
$$\left(\frac{\max(\max_{i\in I_{\delta}}(\delta_{i}),\max_{i\in I_{\epsilon}}(\epsilon_{i}))}{\min(q,\min_{j\in I_{D}}(D_{j}),\min_{j\in I_{E}}(E_{j}))} \right)\frac{\sum_{i\in I_{\delta}}x_{n-i} +\sum_{i\in I_{\epsilon}}y_{n-i}}{1+\sum_{j\in I_{D}}x_{n-j} + \sum_{j\in I_{E}}y_{n-j}}\leq$$
$$\left(\frac{\max(\max_{i\in I_{\delta}}(\delta_{i}),\max_{i\in I_{\epsilon}}(\epsilon_{i}))}{\min(q,\min_{j\in I_{D}}(D_{j}),\min_{j\in I_{E}}(E_{j}))} \right) \left(\frac{1+\sum_{i\in I_{\beta}\cup I_{B}}x_{n-i} +\sum_{i\in I_{\gamma}\cup I_{C}}y_{n-i}}{1+\sum_{j\in I_{B}}x_{n-j} + \sum_{j\in I_{C}}y_{n-j}} \right) \leq M_1(x_{n}+1),$$
where $$M_1=\frac{(\max(A,\max_{j\in I_{B}}(B_{j}),\max_{j\in I_{C}}(C_{j})))(\max(\max_{i\in I_{\delta}}(\delta_{i}),\max_{i\in I_{\epsilon}}(\epsilon_{i})))}{(\min(\alpha +A,\min_{i\in I_{\beta}\cup I_{B}}(\beta_{i}+ B_{i}),\min_{i\in I_{\gamma}\cup I_{C}}(\gamma_{i}+C_{i})))(\min(q,\min_{j\in I_{D}}(D_{j}),\min_{j\in I_{E}}(E_{j})))}.$$ In the case where $A=0$, $q>0$, and $p>0$ we have
$$y_n=\frac{p+\sum^{k}_{i=1}\delta_{i}x_{n-i} + \sum^{k}_{i=1}\epsilon_{i}y_{n-i}}{q+\sum^{k}_{j=1}D_{j}x_{n-j} + \sum^{k}_{j=1}E_{j}y_{n-j}} \leq $$
$$\left(\frac{\max(p,\max_{i\in I_{\delta}}(\delta_{i}),\max_{i\in I_{\epsilon}}(\epsilon_{i}))}{\min(q,\min_{j\in I_{D}}(D_{j}),\min_{j\in I_{E}}(E_{j}))} \right)\frac{1+\sum_{i\in I_{\delta}}x_{n-i} +\sum_{i\in I_{\epsilon}}y_{n-i}}{1+\sum_{j\in I_{D}}x_{n-j} + \sum_{j\in I_{E}}y_{n-j}}\leq$$
$$\left(\frac{\max(p,\max_{i\in I_{\delta}}(\delta_{i}),\max_{i\in I_{\epsilon}}(\epsilon_{i}))}{\min(q,\min_{j\in I_{D}}(D_{j}),\min_{j\in I_{E}}(E_{j}))} \right) \left(\frac{1+\sum_{i\in I_{\beta}\cup I_{B}}x_{n-i} +\sum_{i\in I_{\gamma}\cup I_{C}}y_{n-i}}{\sum_{j\in I_{B}}x_{n-j} + \sum_{j\in I_{C}}y_{n-j}} \right) \leq M_1(x_{n}+1),$$
where $$M_1=\frac{(\max(\max_{j\in I_{B}}(B_{j}),\max_{j\in I_{C}}(C_{j})))(\max(p,\max_{i\in I_{\delta}}(\delta_{i}),\max_{i\in I_{\epsilon}}(\epsilon_{i})))}{(\min(\alpha +A,\min_{i\in I_{\beta}\cup I_{B}}(\beta_{i}+ B_{i}),\min_{i\in I_{\gamma}\cup I_{C}}(\gamma_{i}+C_{i})))(\min(q, \min_{j\in I_{D}}(D_{j}),\min_{j\in I_{E}}(E_{j})))}.$$ In the case where $A=0$, $q>0$, $p=0$, and $\alpha = 0$ we have
$$y_n=\frac{\sum^{k}_{i=1}\delta_{i}x_{n-i} + \sum^{k}_{i=1}\epsilon_{i}y_{n-i}}{q+\sum^{k}_{j=1}D_{j}x_{n-j} + \sum^{k}_{j=1}E_{j}y_{n-j}} \leq $$
$$\left(\frac{\max(\max_{i\in I_{\delta}}(\delta_{i}),\max_{i\in I_{\epsilon}}(\epsilon_{i}))}{\min(q,\min_{j\in I_{D}}(D_{j}),\min_{j\in I_{E}}(E_{j}))} \right)\frac{\sum_{i\in I_{\delta}}x_{n-i} +\sum_{i\in I_{\epsilon}}y_{n-i}}{1+\sum_{j\in I_{D}}x_{n-j} + \sum_{j\in I_{E}}y_{n-j}}\leq$$
$$\left(\frac{\max(\max_{i\in I_{\delta}}(\delta_{i}),\max_{i\in I_{\epsilon}}(\epsilon_{i}))}{\min(q,\min_{j\in I_{D}}(D_{j}),\min_{j\in I_{E}}(E_{j}))} \right) \left(\frac{\sum_{i\in I_{\beta}\cup I_{B}}x_{n-i} +\sum_{i\in I_{\gamma}\cup I_{C}}y_{n-i}}{\sum_{j\in I_{B}}x_{n-j} + \sum_{j\in I_{C}}y_{n-j}} \right) \leq M_1(x_{n}+1),$$
where $$M_1=\frac{(\max(\max_{j\in I_{B}}(B_{j}),\max_{j\in I_{C}}(C_{j})))(\max(\max_{i\in I_{\delta}}(\delta_{i}),\max_{i\in I_{\epsilon}}(\epsilon_{i})))}{(\min(\min_{i\in I_{\beta}\cup I_{B}}(\beta_{i}+ B_{i}),\min_{i\in I_{\gamma}\cup I_{C}}(\gamma_{i}+C_{i})))(\min(q, \min_{j\in I_{D}}(D_{j}),\min_{j\in I_{E}}(E_{j})))}.$$ In the case where $A=0$, $q=0$, and $p>0$ we have
$$y_n=\frac{p+\sum^{k}_{i=1}\delta_{i}x_{n-i} + \sum^{k}_{i=1}\epsilon_{i}y_{n-i}}{\sum^{k}_{j=1}D_{j}x_{n-j} + \sum^{k}_{j=1}E_{j}y_{n-j}} \leq $$
$$\left(\frac{\max(p,\max_{i\in I_{\delta}}(\delta_{i}),\max_{i\in I_{\epsilon}}(\epsilon_{i}))}{\min(\min_{j\in I_{D}}(D_{j}),\min_{j\in I_{E}}(E_{j}))} \right)\frac{1+\sum_{i\in I_{\delta}}x_{n-i} +\sum_{i\in I_{\epsilon}}y_{n-i}}{\sum_{j\in I_{D}}x_{n-j} + \sum_{j\in I_{E}}y_{n-j}}\leq$$
$$\left(\frac{\max(p,\max_{i\in I_{\delta}}(\delta_{i}),\max_{i\in I_{\epsilon}}(\epsilon_{i}))}{\min(\min_{j\in I_{D}}(D_{j}),\min_{j\in I_{E}}(E_{j}))} \right) \left(\frac{1+\sum_{i\in I_{\beta}\cup I_{B}}x_{n-i} +\sum_{i\in I_{\gamma}\cup I_{C}}y_{n-i}}{\sum_{j\in I_{B}}x_{n-j} + \sum_{j\in I_{C}}y_{n-j}} \right)\leq M_1(x_{n}+1),$$
where $$M_1=\frac{(\max(\max_{j\in I_{B}}(B_{j}),\max_{j\in I_{C}}(C_{j})))(\max(p,\max_{i\in I_{\delta}}(\delta_{i}),\max_{i\in I_{\epsilon}}(\epsilon_{i})))}{(\min(\alpha +A,\min_{i\in I_{\beta}\cup I_{B}}(\beta_{i}+ B_{i}),\min_{i\in I_{\gamma}\cup I_{C}}(\gamma_{i}+C_{i})))(\min( \min_{j\in I_{D}}(D_{j}),\min_{j\in I_{E}}(E_{j})))}.$$ In the case where $A=0$, $q=0$, $p=0$, and $\alpha =0 $ we have
$$y_n=\frac{\sum^{k}_{i=1}\delta_{i}x_{n-i} + \sum^{k}_{i=1}\epsilon_{i}y_{n-i}}{\sum^{k}_{j=1}D_{j}x_{n-j} + \sum^{k}_{j=1}E_{j}y_{n-j}} \leq $$
$$\left(\frac{\max(\max_{i\in I_{\delta}}(\delta_{i}),\max_{i\in I_{\epsilon}}(\epsilon_{i}))}{\min(\min_{j\in I_{D}}(D_{j}),\min_{j\in I_{E}}(E_{j}))} \right)\frac{\sum_{i\in I_{\delta}}x_{n-i} +\sum_{i\in I_{\epsilon}}y_{n-i}}{\sum_{j\in I_{D}}x_{n-j} + \sum_{j\in I_{E}}y_{n-j}}\leq$$
$$\left(\frac{\max(\max_{i\in I_{\delta}}(\delta_{i}),\max_{i\in I_{\epsilon}}(\epsilon_{i}))}{\min(\min_{j\in I_{D}}(D_{j}),\min_{j\in I_{E}}(E_{j}))} \right) \left(\frac{\sum_{i\in I_{\beta}\cup I_{B}}x_{n-i} +\sum_{i\in I_{\gamma}\cup I_{C}}y_{n-i}}{\sum_{j\in I_{B}}x_{n-j} + \sum_{j\in I_{C}}y_{n-j}} \right) \leq M_1(x_{n}+1),$$
where $$M_1=\frac{(\max(\max_{j\in I_{B}}(B_{j}),\max_{j\in I_{C}}(C_{j})))(\max(\max_{i\in I_{\delta}}(\delta_{i}),\max_{i\in I_{\epsilon}}(\epsilon_{i})))}{(\min(\min_{i\in I_{\beta}\cup I_{B}}(\beta_{i}+ B_{i}),\min_{i\in I_{\gamma}\cup I_{C}}(\gamma_{i}+C_{i})))(\min( \min_{j\in I_{D}}(D_{j}),\min_{j\in I_{E}}(E_{j})))}.$$ In the case where $A=0$, $q>0$, $p=0$, and $\alpha > 0$ we have
$$y_n=\frac{\sum^{k}_{i=1}\delta_{i}x_{n-i} + \sum^{k}_{i=1}\epsilon_{i}y_{n-i}}{q+\sum^{k}_{j=1}D_{j}x_{n-j} + \sum^{k}_{j=1}E_{j}y_{n-j}} \leq $$
$$\left(\frac{\max(\max_{i\in I_{\delta}}(\delta_{i}),\max_{i\in I_{\epsilon}}(\epsilon_{i}))}{\min(q,\min_{j\in I_{D}}(D_{j}),\min_{j\in I_{E}}(E_{j}))} \right)\frac{\sum_{i\in I_{\delta}}x_{n-i} +\sum_{i\in I_{\epsilon}}y_{n-i}}{1+\sum_{j\in I_{D}}x_{n-j} + \sum_{j\in I_{E}}y_{n-j}}\leq$$
$$\left(\frac{\max(\max_{i\in I_{\delta}}(\delta_{i}),\max_{i\in I_{\epsilon}}(\epsilon_{i}))}{\min(q,\min_{j\in I_{D}}(D_{j}),\min_{j\in I_{E}}(E_{j}))} \right)\left(\frac{1+\sum_{i\in I_{\beta}\cup I_{B}}x_{n-i} +\sum_{i\in I_{\gamma}\cup I_{C}}y_{n-i}}{\sum_{j\in I_{B}}x_{n-j} + \sum_{j\in I_{C}}y_{n-j}} \right) \leq M_1(x_{n}+1),$$

where $$M_1=\frac{(\max(\max_{j\in I_{B}}(B_{j}),\max_{j\in I_{C}}(C_{j})))(\max(\max_{i\in I_{\delta}}(\delta_{i}),\max_{i\in I_{\epsilon}}(\epsilon_{i})))}{(\min(\alpha +A,\min_{i\in I_{\beta}\cup I_{B}}(\beta_{i}+ B_{i}),\min_{i\in I_{\gamma}\cup I_{C}}(\gamma_{i}+C_{i})))(\min(q, \min_{j\in I_{D}}(D_{j}),\min_{j\in I_{E}}(E_{j})))}.$$ In the case where $A=0$, $q=0$, $p=0$, and $\alpha >0 $ we have
$$y_n=\frac{\sum^{k}_{i=1}\delta_{i}x_{n-i} + \sum^{k}_{i=1}\epsilon_{i}y_{n-i}}{\sum^{k}_{j=1}D_{j}x_{n-j} + \sum^{k}_{j=1}E_{j}y_{n-j}} \leq $$
$$\left(\frac{\max(\max_{i\in I_{\delta}}(\delta_{i}),\max_{i\in I_{\epsilon}}(\epsilon_{i}))}{\min(\min_{j\in I_{D}}(D_{j}),\min_{j\in I_{E}}(E_{j}))} \right)\frac{\sum_{i\in I_{\delta}}x_{n-i} +\sum_{i\in I_{\epsilon}}y_{n-i}}{\sum_{j\in I_{D}}x_{n-j} + \sum_{j\in I_{E}}y_{n-j}}\leq$$
$$\left(\frac{\max(\max_{i\in I_{\delta}}(\delta_{i}),\max_{i\in I_{\epsilon}}(\epsilon_{i}))}{\min(\min_{j\in I_{D}}(D_{j}),\min_{j\in I_{E}}(E_{j}))} \right) \left(\frac{1+\sum_{i\in I_{\beta}\cup I_{B}}x_{n-i} +\sum_{i\in I_{\gamma}\cup I_{C}}y_{n-i}}{\sum_{j\in I_{B}}x_{n-j} + \sum_{j\in I_{C}}y_{n-j}} \right)\leq M_1(x_{n}+1),$$
where $$M_1=\frac{(\max(\max_{j\in I_{B}}(B_{j}),\max_{j\in I_{C}}(C_{j})))(\max(\max_{i\in I_{\delta}}(\delta_{i}),\max_{i\in I_{\epsilon}}(\epsilon_{i})))}{(\min(\alpha +A,\min_{i\in I_{\beta}\cup I_{B}}(\beta_{i}+ B_{i}),\min_{i\in I_{\gamma}\cup I_{C}}(\gamma_{i}+C_{i})))(\min( \min_{j\in I_{D}}(D_{j}),\min_{j\in I_{E}}(E_{j})))}.$$

\end{proof}

\begin{thm}
Suppose that we have a $k^{th}$ order system of two rational
difference equations
$$x_n=\frac{\alpha+\sum^{k}_{i=1}\beta_{i}x_{n-i} + \sum^{k}_{i=1}\gamma_{i}y_{n-i}}{A+\sum^{k}_{j=1}B_{j}x_{n-j} + \sum^{k}_{j=1}C_{j}y_{n-j}},\quad n\in\mathbb{N},$$
$$y_n=\frac{p+\sum^{k}_{i=1}\delta_{i}x_{n-i} + \sum^{k}_{i=1}\epsilon_{i}y_{n-i}}{q+\sum^{k}_{j=1}D_{j}x_{n-j} + \sum^{k}_{j=1}E_{j}y_{n-j}},\quad n\in\mathbb{N},$$
with non-negative parameters and non-negative initial conditions.
Further suppose that $ I_{\delta}\subset I_{\beta}\cup I_{B}$, $I_{\beta}\subset I_{\delta}\cup I_{D}$, $I_{B} = I_{D}$, $ I_{\epsilon}\subset I_{\gamma}\cup I_{C}$, $I_{\gamma}\subset I_{\epsilon}\cup I_{E}$ $I_{C} = I_{E}$. Also assume that $A>0$ if and only if $q>0$, and whenever $p > 0$, then $\alpha >0$ or $A>0$, also whenever $\alpha > 0$, then $p>0$ or $q>0$. Then
there exists $M_{1},M_{3}>0$ and $M_{4}\geq M_{2}\geq 0$ so that $x_{n}\leq M_{1}y_{n}+M_{2}\leq M_{3}x_{n}+M_{4}$ for all $n\in\mathbb{N}$.
\end{thm}

\begin{proof}
 First notice that Theorem 26 applies to this system. This gives us $M_{5}>0$ and $M_{6}\geq 0$ so that $y_{n}\leq M_{5}x_{n} + M_{6}$ for all $n\in\mathbb{N}$. Moreover after a very simple change of variables Theorem 26 applies again. The change of variables we refer to here comes from renaming $x_{n}$ as $y_{n}$, $\beta_{i}$ as $\epsilon_{i}$, $B_{i}$ as $E_{i}$, $\gamma_{i}$ as $\delta_{i}$, $C_{i}$ as $D_{i}$, $\alpha$ as $p$, $A$ as $q$, and vice versa.
This gives us $M_{1}>0$ and $M_{2}\geq 0$ so that $x_{n}\leq M_{1}y_{n} + M_{2}$ for all $n\in\mathbb{N}$. Choose $M_{3}=M_{1}M_{5}$ and $M_{4}=M_{1}M_{6}+M_{2}$ and we get that there exists $M_{1},M_{3}>0$ and $M_{4}\geq M_{2}\geq 0$ so that $x_{n}\leq M_{1}y_{n}+M_{2}\leq M_{3}x_{n}+M_{4}$ for all $n\in\mathbb{N}$.
\end{proof}

\section{Some Examples}

Here we present some examples which demonstrate how the results presented earlier in this article are applied to particular special cases.

% Example 1

\begin{expl}
Consider the following system of two rational difference equations
$$x_{n}=\frac{\alpha + \beta_{1}x_{n-1}}{A + C_{2}y_{n-2}}, \quad n\in \mathbb{N}, $$
$$y_{n}=\frac{p + \delta_{1}x_{n-1}}{q + E_{2}y_{n-2}}, \quad n\in \mathbb{N} .$$
We assume positive parameters and non-negative initial conditions. This implies that the solutions $x_{n}$ and $y_{n}$ are bounded above by a positive constant.
\end{expl}
\begin{proof}
We apply Theorem 1. First notice that, by theorem 25, that there exists constants $M_{1},M_{2} > 0$ so that $M_{1}y_{n}\leq x_{n} \leq M_{2}y_{n}$ for all $n\in\mathbb{N}$. This is since
$$\begin{array}{c}
\{1\}=I_{\beta} = I_{\delta}\\
\emptyset=I_{B} = I_{D}\\
\emptyset=I_{\gamma} = I_{\epsilon}\\
\{2\}=I_{C} = I_{E}\\
\alpha>0\quad and \quad p>0\\
A>0\quad and \quad q>0.\\
\end{array}$$
For the final condition, let $\eta = 2$ so that for any sequence $\{c_{m}\} ^{\infty}_{m=1} $ with $c_m\in I_{\beta}\cup I_{\gamma}=\{1\}$ for $m=1,2, \dots $ we choose  $
N_{1}=1
$,$
N_{2}=2 \leq \eta
$ $ $ so that $ \sum^{N_{2}}_{m=N_{1}} 1 \in I_{B}\cup I_{C} = \{2\}$.
\end{proof}

% Example 2

\begin{expl}
Consider the following system of two rational difference equations
$$x_{n}=\frac{\beta_{2}x_{n-2} + \gamma_{1}y_{n-1} + \gamma_{2}y_{n-2}}{B_{2}x_{n-2} + C_{1}y_{n-1}}, \quad n\in \mathbb{N}, $$
$$y_{n}=\frac{p + \delta_{2}x_{n-2} + \epsilon_{1}y_{n-1} + \epsilon_{2}y_{n-2}}{q + D_{2}x_{n-2} + E_{1}y_{n-1}}, \quad n\in \mathbb{N} .$$
We assume positive parameters and positive initial conditions. This implies that the solutions $x_{n}$ and $y_{n}$ are bounded above by a positive constant.
\end{expl}
\begin{proof}
We apply Theorem 3 case $(iii)$.
We will now prove that there exists constants $M_{1},M_{2} > 0$ so that $M_{1}y_{n}\leq x_{n} \leq M_{2}y_{n}$ for all $n\in\mathbb{N}$. To show this, we first we show that there exists $L$ so that $y_{n} \geq L$ for all $n\in\mathbb{N}$. We can choose $L=\frac{\min(p,\delta_{2},\epsilon_{1})}{\max(q,D_{2},E_{1})}$, since
$$y_{n}=\frac{p + \delta_{2}x_{n-2} + \epsilon_{1}y_{n-1} + \epsilon_{2}y_{n-2}}{q + D_{2}x_{n-2} + E_{1}y_{n-1}}\geq \frac{\min(p,\delta_{2},\epsilon_{1})}{\max(q,D_{2},E_{1})}.$$
We may choose $M_{2}=\max\left(\left(\frac{\max(1,\beta_{2},\gamma_{1},\gamma_{2})}{\min(\frac{C_{1}L}{2},B_{2},\frac{C_{1}}{2})}\right)\left(\frac{\max(q,D_{2},E_{1})}{\min(p,\delta_{2},\epsilon_{1},\epsilon_{2})}\right),\frac{x_{1}}{y_{1}}\right)$ since
$$x_{n}=\frac{\beta_{2}x_{n-2}+\gamma_{1}y_{n-1}+\gamma_{2}y_{n-2}}{B_{2}x_{n-2}+C_{1}y_{n-1}}\leq
\frac{1+\beta_{2}x_{n-2}+\gamma_{1}y_{n-1}+\gamma_{2}y_{n-2}}{\frac{C_{1}L}{2}+B_{2}x_{n-2}+\frac{C_{1}}{2}y_{n-1}}$$
$$\leq\left(\frac{\max(1,\beta_{2},\gamma_{1},\gamma_{2})}{\min(\frac{C_{1}L}{2},B_{2},\frac{C_{1}}{2})}\right)\left(\frac{1+x_{n-2}+y_{n-1}+y_{n-2}}{1+x_{n-2}+y_{n-1}}\right)$$
$$\leq\left(\frac{\max(1,\beta_{2},\gamma_{1},\gamma_{2})}{\min(\frac{C_{1}L}{2},B_{2},\frac{C_{1}}{2})}\right)\left(\frac{\max(q,D_{2},E_{1})}{\min(p,\delta_{2},\epsilon_{1},\epsilon_{2})}\right)y_{n}.$$
We may choose $M_{1}=\min\left(\left(\frac{\min(\frac{\gamma_{1}L}{2},\beta_{2},\frac{\gamma_{1}}{2},\gamma_{2})}{\max(1,B_{2},C_{1})}\right)\left(\frac{\min(q,D_{2},E_{1})}{\max(p,\delta_{2},\epsilon_{1},\epsilon_{2})}\right),\frac{x_{1}}{y_{1}}\right)$
since
$$x_{n}=\frac{\beta_{2}x_{n-2}+\gamma_{1}y_{n-1}+\gamma_{2}y_{n-2}}{B_{2}x_{n-2}+C_{1}y_{n-1}}\geq
\frac{\frac{\gamma_{1}L}{2}+\beta_{2}x_{n-2}+\frac{\gamma_{1}}{2}y_{n-1}+\gamma_{2}y_{n-2}}{1+B_{2}x_{n-2}+C_{1}y_{n-1}}$$
$$\geq\left(\frac{\min(\frac{\gamma_{1}L}{2},\beta_{2},\frac{\gamma_{1}}{2},\gamma_{2})}{\max(1,B_{2},C_{1})}\right)\left(\frac{1+x_{n-2}+y_{n-1}+y_{n-2}}{1+x_{n-2}+y_{n-1}}\right)$$
$$\geq\left(\frac{\min(\frac{\gamma_{1}L}{2},\beta_{2},\frac{\gamma_{1}}{2},\gamma_{2})}{\max(1,B_{2},C_{1})}\right)\left(\frac{\min(q,D_{2},E_{1})}{\max(p,\delta_{2},\epsilon_{1},\epsilon_{2})}\right)y_{n}.$$
The conditions specific to case $(iii)$ are satisfied since $p,q>0$, $\{1,2\}=I_{D}\cup I_{E}\subset I_{\delta}\cup I_{\epsilon}=\{1,2\}$, and $I_{C}\neq\emptyset$.
For the final condition, let $\eta=1$ so that for any sequence $\{c_{m}\} ^{\infty}_{m=1} $ with $c_m\in I_{\beta}\cup I_{\gamma}=\{1,2\}$ for $m=1,2, \dots $ we choose
$N_{1}=1$, $N_{2}=1 \leq \eta$, since $c_{1}\in \{1,2\}$, so that
$\sum^{N_{2}}_{m=N_{1}} c_{m}=c_{1} \in I_{B}\cup I_{C} = \{1,2\}$.
\end{proof}

% Example 3

\begin{expl}
Consider the following system of two rational difference equations
$$x_{n}=\frac{\alpha + \beta_{1}x_{n-1} + \gamma_{1}y_{n-1}}{B_{2}x_{n-2} + C_{1}y_{n-1}}, \quad n\in \mathbb{N}, $$
$$y_{n}=\frac{p + \delta_{2}x_{n-2} + \epsilon_{1}y_{n-1}}{q}, \quad n\in \mathbb{N} .$$
We assume positive parameters and non-negative initial conditions. This implies that the solution $x_{n}$ is bounded above by a positive constant.
\end{expl}
\begin{proof}
We apply by Theorem 10 case $(iii)$. The conditions  specific for case $(iii)$ are satisfied since
$$\begin{array}{c}
\emptyset = I_{D}\subset I_{\delta}\\
\emptyset = I_{E}\subset I_{\epsilon}\\
\{1\} = I_{\gamma}=I_{C}\\
A=0\quad with \quad p,q>0.\\
\end{array}$$
For the final condition, let $\eta=2$ so that for any sequence $\{c_{m}\} ^{\infty}_{m=1} $ with $c_m\in I_{\beta}=\{1\}$, for $m=1,2, \dots $, we choose  $
N_{1}=1
$,$
N_{2}=2 \leq \eta
$ $ $, such that $ \sum^{N_{2}}_{m=N_{1}} 1 \in I_{B} = \{2\}$.
\end{proof}

% Example 4

\begin{expl}
Consider the following system of two rational difference equations
$$x_{n}=\frac{\beta_{1}x_{n-1} + \beta_{2}x_{n-2} + \gamma_{2}y_{n-2}}{B_{2}x_{n-2} + C_{1}y_{n-1}}, \quad n\in \mathbb{N}, $$
$$y_{n}=\frac{p + \delta_{1}x_{n-1} + \delta_{2}x_{n-2} +\epsilon_{2}y_{n-2}}{q + D_{2}x_{n-2} + E_{1}y_{n-1}}, \quad n\in \mathbb{N} .$$
We assume positive parameters and positive initial conditions. This implies that the solutions $x_{n}$ and $y_{n}$ are bounded above by a positive constant.
\end{expl}
\begin{proof}
We apply Theorem 6.
We will now prove that there exists constants $M_{1},M_{2} > 0$ so that $M_{1}y_{n}\leq x_{n} \leq M_{2}y_{n}$ for all $n\in\mathbb{N}$. To show this, we first show that there exists $L$ so that $y_{n} \geq L$ or all $n\in\mathbb{N}$.  We now show that we may choose
$L=\min\left(\frac{\min(p,\delta_{1},\delta_{2})}{\max(q+E_{1}\left(\frac{p}{q}\right),D_{2},bE_{1})}\right), y_{1})$. To show this, we first deduce the following inequality
$$y_{n}=\frac{p + \delta_{1}x_{n-1} + \delta_{2}x_{n-2} + \epsilon_{2}y_{n-2}}{q + D_{2}x_{n-2} + E_{1}y_{n-1}}\leq\frac{p}{q} + \frac{\delta_{1}x_{n-1} + \delta_{2}x_{n-2} + \epsilon_{2}y_{n-2}}{D_{2}x_{n-2} + E_{1}y_{n-1}}$$ $$\leq
\frac{p}{q} + \left(\frac{\max(\delta_{1},\delta_{2},\epsilon_{2})}{\min(D_{2},E_{1})}\right)\left(\frac{x_{n-1} + x_{n-2} + y_{n-2}}{x_{n-2} + y_{n-1}}\right)$$ $$\leq\frac{p}{q} + \left(\frac{\max(\delta_{1},\delta_{2},\epsilon_{2})}{\min(D_{2},E_{1})}\right)\left(\frac{\max(B_{2},C_{1})}{\min(\beta_{1},\beta_{2},\gamma_{2})}\right)x_{n}.$$
Let $b=\left(\frac{\max(\delta_{1},\delta_{2},\epsilon_{2})}{\min(D_{2},E_{1})}\right)\left(\frac{\max(B_{2},C_{1})}{\min(\beta_{1},\beta_{2},\gamma_{2})}\right)$.
Using the above, we get
$$y_{n}=\frac{p + \delta_{1}x_{n-1} + \delta_{2}x_{n-2} +\epsilon_{2}y_{n-2}}{q + D_{2}x_{n-2} + E_{1}y_{n-1}}\geq \frac{p + \delta_{1}x_{n-1} + \delta_{2}x_{n-2}}{q + D_{2}x_{n-2} + E_{1}\left(\frac{p}{q} + bx_{n-1}\right)}$$ $$\geq \frac{\min(p,\delta_{1},\delta_{2})}{\max(q+E_{1}\left(\frac{p}{q}\right),D_{2},bE_{1})}.$$
We now show that we may choose\newline
$M_{2}=\max\left(\frac{x_{1}}{y_{1}},\frac{x_{2}}{y_{2}},\frac{x_{3}}{y_{3}},\left(\frac{\max(1,\beta_{1},\beta_{2},\gamma_{2})}{\min(\frac{C_{1}L}{2},B_{2},\frac{C_{1}}{2})}\right)\left(\frac{\max(q,D_{2},E_{1})}{\min(p,\delta_{1},\delta_{2},\epsilon_{2})}\right)\right)$
since
$$x_{n}=\frac{\beta_{1}x_{n-1} + \beta_{2}x_{n-2} + \gamma_{2}y_{n-2}}{B_{2}x_{n-2} + C_{1}y_{n-1}}\leq
\frac{1+\beta_{1}x_{n-1} + \beta_{2}x_{n-2} + \gamma_{2}y_{n-2}}{\frac{C_{1}L}{2}+B_{2}x_{n-2}+\frac{C_{1}}{2}y_{n-1}}$$
$$\leq\left(\frac{\max(1,\beta_{1},\beta_{2},\gamma_{2})}{\min(\frac{C_{1}L}{2},B_{2},\frac{C_{1}}{2})}\right)\left(\frac{1+x_{n-1}+x_{n-2}+y_{n-2}}{1+x_{n-2}+y_{n-1}}\right)\leq$$
$$\left(\frac{\max(1,\beta_{1},\beta_{2},\gamma_{2})}{\min(\frac{C_{1}L}{2},B_{2},\frac{C_{1}}{2})}\right)\left(\frac{\max(q,D_{2},E_{1})}{\min(p,\delta_{1},\delta_{2},\epsilon_{2})}\right)y_{n}.$$
We may also choose $M_{1}=\min\left(\frac{x_{1}}{y_{1}},\frac{x_{2}}{y_{2}},\frac{x_{3}}{y_{3}},\left(\frac{\min(\frac{\gamma_{2}L}{2},\beta_{1},\beta_{2},\frac{\gamma_{2}}{2})}{\max(1,B_{2},C_{1})}\right)\left(\frac{\min(q,D_{2},E_{1})}{\max(p,\delta_{1},\delta_{2},\epsilon_{2})}\right)\right)$
since
$$x_{n}=\frac{\beta_{1}x_{n-1} + \beta_{2}x_{n-2} + \gamma_{2}y_{n-2}}{B_{2}x_{n-2} + C_{1}y_{n-1}}\geq
\frac{\frac{\gamma_{2}L}{2}+\beta_{1}x_{n-1} + \beta_{2}x_{n-2} + \frac{\gamma_{2}}{2}y_{n-2}}{1+B_{2}x_{n-2}+C_{1}y_{n-1}}$$
$$\geq\left(\frac{\min(\frac{\gamma_{2}L}{2},\beta_{1},\beta_{2},\frac{\gamma_{2}}{2})}{\max(1,B_{2},C_{1})}\right)\left(\frac{1+x_{n-1}+x_{n-2}+y_{n-2}}{1+x_{n-2}+y_{n-1}}\right)$$
$$\geq\left(\frac{\min(\frac{\gamma_{2}L}{2},\beta_{1},\beta_{2},\frac{\gamma_{2}}{2})}{\max(1,B_{2},C_{1})}\right)\left(\frac{\min(q,D_{2},E_{1})}{\max(p,\delta_{1},\delta_{2},\epsilon_{2})}\right)y_{n}.$$
For the final condition needed to be satisfied, it is observed that $\{1,2\}=I_{\beta}\cup I_{\gamma}\subset I_{B}\cup I_{C}=\{1,2\}$.

\end{proof}

% Example 5
\begin{expl}
Consider the following system of two rational difference equations
$$x_{n}=\frac{\gamma_{1}y_{n-1} + \gamma_{2}y_{n-2}}{A + B_{2}x_{n-2} + C_{2}y_{n-2}}, \quad n\in \mathbb{N}, $$
$$y_{n}=\frac{p + \delta_{1}x_{n-1} + \epsilon_{1}y_{n-1} + \epsilon_{2}y_{n-2}}{q + D_{1}x_{n-1} + D_{2}x_{n-2} + E_{2}y_{n-2}}, \quad n\in \mathbb{N} .$$
We assume positive parameters and non-negative initial conditions. This implies that the solutions $x_{n}$ and $y_{n}$ are bounded above by a positive constant.
\end{expl}
\begin{proof}
We apply Theorem 16. We first notice that there exists $M_{1}>0$ and $M_{2}\geq 0$ so that $y_{n}\leq M_{1}x_{n} + M_{2}$ for all $n\in\mathbb{N}$. This is since
$$y_{n}=\frac{p + \delta_{1}x_{n-1} + \epsilon_{1}y_{n-1} + \epsilon_{2}y_{n-2}}{q + D_{1}x_{n-1} + D_{2}x_{n-2} + E_{2}y_{n-2}}\leq\left(\frac{\max(p,\delta_{1},\epsilon_{1},\epsilon_{2})}{\min(q,D_{1},D_{2},E_{2})}\right)\left(\frac{1 + x_{n-1} + y_{n-1} + y_{n-2}}{1 + x_{n-1} + x_{n-2} + y_{n-2}}\right)$$
$$\leq\left(\frac{\max(p,\delta_{1},\epsilon_{1},\epsilon_{2})}{\min(q,D_{1},D_{2},E_{2})}\right)\left(\frac{1 + x_{n-1}}{1 + x_{n-1} + x_{n-2} + y_{n-2}}\right) $$ $$+\left(\frac{\max(p,\delta_{1},\epsilon_{1},\epsilon_{2})}{\min(q,D_{1},D_{2},E_{2})}\right)\left(\frac{y_{n-1} + y_{n-2}}{1 + x_{n-1} + x_{n-2} + y_{n-2}}\right)$$
$$\leq \left(\frac{\max(p,\delta_{1},\epsilon_{1},\epsilon_{2})}{\min(q,D_{1},D_{2},E_{2})}\right) +\left(\frac{\max(p,\delta_{1},\epsilon_{1},\epsilon_{2})}{\min(q,D_{1},D_{2},E_{2})}\right)\left(\frac{y_{n-1} + y_{n-2}}{1 + x_{n-2} + y_{n-2}}\right)$$
$$\leq M_{1}x_{n}+M_{2}$$
where $M_{1}=\left(\frac{\max(p,\delta_{1},\epsilon_{1},\epsilon_{2})}{\min(q,D_{1},D_{2},E_{2})}\right)\left(\frac{\max(A,B_{2},C_{2})}{\min(\gamma_{1},\gamma_{2})}\right)$ and $M_{2}=\left(\frac{\max(p,\delta_{1},\epsilon_{1},\epsilon_{2})}{\min(q,D_{1},D_{2},E_{2})}\right)$.
For the final condition, let $\eta = 2$ so that for any sequence $\{c_{m}\} ^{\infty}_{m=1} $ with $c_m\in I_{\beta}\cup (I_{\gamma}\backslash I_{C}) = \{1\}$ for $m=1,2, \dots $ we choose  $
N_{1}=1
$,$
N_{2}=2 \leq \eta
$ $ $, such that $ \sum^{N_{2}}_{m=N_{1}} 1 \in I_{B} = \{2\}$.
\end{proof}

% Example 6

\begin{expl}
Consider the following first order system of two rational difference equations
$$x_{n}=\frac{\alpha + \gamma_{1}y_{n-1}}{A + B_{1}x_{n-1}},  \quad n\in \mathbb{N},$$
$$y_{n}=\frac{\delta_{1}x_{n-1} + \epsilon_{1}y_{n-1}}{q + D_{1}x_{n-1}}, \quad n\in \mathbb{N} .$$
We assume positive parameters and non-negative initial conditions. This implies that the solutions $x_{n}$ and $y_{n}$ are bounded above by a positive constant.
\end{expl}

\begin{proof}
We apply Theorem 20. First notice that, by Theorem 27, there exists $M_{1},M_{3}>0$ and $M_{4}\geq M_{2}\geq 0$ so that $x_{n}\leq M_{1}y_{n}+M_{2}\leq M_{3}x_{n}+M_{4}$ for all $n\in\mathbb{N}$. This is since
$$\begin{array}{c}
\{1\}=I_{\delta}\subset I_{\beta}\cup I_{B}=\{1\}\\
\emptyset=I_{\beta}\subset I_{\delta}\cup I_{D}\\
\{1\}=I_{B} = I_{D}\\
\{1\}=I_{\epsilon}\subset I_{\gamma}\cup I_{C}=\{1\}\\
\{1\}=I_{\gamma}\subset I_{\epsilon}\cup I_{E}=\{1\}\\
\emptyset=I_{C} = I_{E}\\
\alpha>0\quad and \quad A>0\quad with \quad q>0.\\
\end{array}$$
For the final condition, let $\eta=1$ so that for any sequence $\{c_{m}\} ^{\infty}_{m=1} $ with $c_m\in I_{\beta}\cup I_{\gamma}=\{1\}$ for $m=1,2, \dots $ we choose
$N_{1}=1$, $N_{2}=1 \leq \eta$ so that
$\sum^{N_{2}}_{m=N_{1}} 1 \in I_{B}\cup I_{C} = \{1\}.$
\end{proof}

% Example 7

\begin{expl}
Consider the following first order system of two rational difference equations
$$x_{n}=\frac{\beta_{1}x_{n-1}+ \gamma_{1}y_{n-1}}{A + B_{1}x_{n-1}},  \quad n\in \mathbb{N},$$
$$y_{n}=\frac{p+ \delta_{1}x_{n-1} + \epsilon_{1}y_{n-1}}{q+ D_{1}x_{n-1}}, \quad n\in \mathbb{N} .$$
We assume positive parameters and non-negative initial conditions. This implies that the solutions $x_{n}$ and $y_{n}$ are bounded above by a positive constant.
\end{expl}

\begin{proof}
We apply Theorem 21. First notice that, by Theorem 27, there exists $M_{1},M_{3}>0$ and $M_{4}\geq M_{2}\geq 0$ so that $x_{n}\leq M_{1}y_{n}+M_{2}\leq M_{3}x_{n}+M_{4}$ for all $n\in\mathbb{N}$. This is since
$$\begin{array}{c}
\{1\}=I_{\delta}\subset I_{\beta}\cup I_{B}=\{1\}\\
\{1\}=I_{\beta}\subset I_{\delta}\cup I_{D}=\{1\}\\
\{1\}=I_{B} = I_{D}\\
\{1\}=I_{\epsilon}\subset I_{\gamma}\cup I_{C}=\{1\}\\
\{1\}=I_{\gamma}\subset I_{\epsilon}\cup I_{E}=\{1\}\\
\emptyset=I_{C} = I_{E}\\
p>0\quad and \quad q>0\quad with \quad A>0.\\
\end{array}$$
For the final condition, let $\eta=1$ so that for any sequence $\{c_{m}\} ^{\infty}_{m=1} $ with $c_m\in I_{\delta}\cup I_{\epsilon}=\{1\}$ for $m=1,2, \dots $ we choose
$N_{1}=1$, $N_{2}=1 \leq \eta$ so that
$\sum^{N_{2}}_{m=N_{1}} 1 \in I_{D}\cup I_{E} = \{1\}.$
\end{proof}

% Example 8

\begin{expl}
Consider the following first order system of two rational difference equations
$$x_{n}=\frac{\alpha + \beta_{1}x_{n-1}}{B_{1}x_{n-1} + C_{1}y_{n-1}},  \quad n\in \mathbb{N},$$
$$y_{n}=\frac{p + \epsilon_{1}y_{n-1}}{D_{1}x_{n-1} + E_{1}y_{n-1}}, \quad n\in \mathbb{N} .$$
We assume positive parameters and non-negative initial conditions. This implies that the solutions $x_{n}$ and $y_{n}$ are bounded above by a positive constant.
\end{expl}

\begin{proof}
We apply Theorem 22 case $(i)$. First notice that, by Theorem 27, there exists $M_{1},M_{3}>0$ and $M_{4}\geq M_{2}\geq 0$ so that $x_{n}\leq M_{1}y_{n}+M_{2}\leq M_{3}x_{n}+M_{4}$ for all $n\in\mathbb{N}$. This is since
$$\begin{array}{c}
\emptyset=I_{\delta}\subset I_{\beta}\cup I_{B}\\
\{1\}=I_{\beta}\subset I_{\delta}\cup I_{D}=\{1\}\\
\{1\}=I_{B} = I_{D}\\
\{1\}=I_{\epsilon}\subset I_{\gamma}\cup I_{C}=\{1\}\\
\emptyset=I_{\gamma}\subset I_{\epsilon}\cup I_{E}\\
\{1\}=I_{C} = I_{E}\\
\alpha>0\quad and \quad p>0.\\
\end{array}$$
Thus $x_{n}\leq M_{1}y_{n}+M_{2} +1\leq M_{3}x_{n}+M_{4} + 2$ for all $n\in\mathbb{N}$.

Case $(i)$ is satisfied, since $\{1\}=I_{B}\cup I_{C}\subset I_{\beta}\cup I_{\gamma}=\{1\}$, $\alpha>0$, and $I_{B}\neq\emptyset$.
For the final condition, let $\eta=1$ so that for any sequence $\{c_{m}\} ^{\infty}_{m=1} $ with $c_m\in I_{\beta}\cup I_{\gamma}=\{1\}$ for $m=1,2, \dots $ we choose
$N_{1}=1$, $N_{2}=1 \leq \eta$ so that
$\sum^{N_{2}}_{m=N_{1}} 1 \in I_{B}\cup I_{C} = \{1\}.$
\end{proof}

% Example 9
\begin{expl}
Consider the following system of two rational difference equations
$$x_{n}=\frac{\alpha + \beta_{2}x_{n-2} + \gamma_{1}y_{n-1} + \gamma_{2}y_{n-2}}{A + B_{2}x_{n-2} + C_{1}y_{n-1} + C_{2}y_{n-2}}, \quad n\in \mathbb{N}, $$
$$y_{n}=\frac{p+\delta_{2}x_{n-2} + \epsilon_{1}y_{n-1} + \epsilon_{2}y_{n-2}}{D_{2}x_{n-2} + E_{2}y_{n-2}}, \quad n\in \mathbb{N} .$$
We assume positive parameters and non-negative initial conditions. This implies that the solutions $x_{n}$ and $y_{n}$ are bounded above by a positive constant.
\end{expl}
\begin{proof}
We apply Theorem 14 case $(ii)$. To apply Theorem 14 we must use a change of variables. The change of variables we refer to here comes from renaming $x_{n}$ as $y_{n}$, $\beta_{i}$ as $\epsilon_{i}$, $B_{i}$ as $E_{i}$, $\gamma_{i}$ as $\delta_{i}$, $C_{i}$ as $D_{i}$, $\alpha$ as $p$, $A$ as $q$, and vice versa.\newline
Although we use this change of variables, we keep notation consistent with our notation before the change of variables for the remainder of this example. We do this to avoid confusion.
First notice that, by applying Theorem 24 after the change of variables, we see that there exists $M_{1}>0$ so that $x_{n}\leq M_{1}y_{n}$ for all $n\in\mathbb{N}$. This is since
$$\begin{array}{c}
\{1,2\}=I_{\gamma}=I_{\epsilon}\\
\{2\}=I_{E}\subset I_{C}=\{1,2\}\\
\{2\}=I_{\beta}=I_{\delta}\\
\{2\}=I_{D}=I_{B}=\{2\}\\
\alpha>0\quad and \quad p>0.\\
\end{array}$$
The conditions for case $(ii)$ are satisfied since
$$\begin{array}{c}
\alpha >0\\
\{1,2\}=I_{C}=I_{\gamma}\\
\{2\}=I_{B}=I_{\beta}\\
I_{D}\neq\emptyset.\\
\end{array}$$
For the final condition let $\eta_{1}=1$ so that for any sequence $\{c_{m}\} ^{\infty}_{m=1} $ with $c_{m}\in I_{\beta}=\{2\}$ for $m=1,2, \dots $
we choose
$N_{1}=1,N_{2}=1 \leq\eta_{1}$, since $c_{1}=2$,
so that $ \sum^{N_{2}}_{m=N_{1}} 2 \in I_{C}\cup I_{B} = \{1,2\}$.
Now, we let $\eta_{2}=2$ so that for any sequence $\{d_{m}\} ^{\infty}_{m=1} $ with $d_{m}\in I_{\epsilon}=\{1,2\}$, for $m=1,2, \dots $,
we choose
$N_{3}=1,N_{4}=1 \leq\eta_{2}$, if $d_{1}=2$,
so that $ \sum^{N_{4}}_{m=N_{3}} d_{m} = 2\in I_{E} = \{2\}$,
we choose
$N_{3}=1 \leq \eta_{2}$ and $N_{4}=2 \leq \eta_{2}$, if $d_{1}=1$ and $d_{2}=1$,
so that $ \sum^{N_{4}}_{m=N_{3}} d_{m}=2\in I_{E} = \{2\}$
and
we choose
$N_{3}=2,N_{4}=2 \leq \eta_{2}=2$, if $d_{2}=2$ with $d_{1}=1$,
so that $ \sum^{N_{4}}_{m=N_{3}} d_{m} \in I_{E} = \{2\}$.
\end{proof}

% Example 10

\begin{expl}
Consider the following system of two rational difference equations
$$x_{n}=\frac{\beta_{1}x_{n-1} + \beta_{2}x_{n-2}}{B_{2}x_{n-2} + C_{1}y_{n-1}}, n=0,1,2,\dots, $$
$$y_{n}=\frac{p + \delta_{1}x_{n-1} + \delta_{2}x_{n-2}}{q + D_{2}x_{n-2} + E_{1}y_{n-1}}, n=0,1,2,\dots .$$
We assume positive parameters and positive initial conditions. This implies that the solutions $x_{n}$ and $y_{n}$ are bounded above by a positive constant.
\end{expl}
\begin{proof}
We apply Theorem 22 case $(ii)$.
We will now prove that there exists $M_{1},M_{3}>0$ and $M_{4}> M_{2}> 0$ so that $x_{n}\leq M_{1}y_{n}+M_{2}\leq M_{3}x_{n}+M_{4}$ for all $n\in\mathbb{N}$. We now show that we may choose
$M_{3}=\max\left(M_{1}\left(\frac{\max(\delta_{1},\delta_{2})}{\min(D_{2},E_{1})}\right)\left(\frac{\max(B_{2},C_{1})}{\min(\beta_{1},\beta_{2})}\right),\frac{M_{1}y_{1}}{x_{1}}\right)$ and $M_{4}=M_{1}(\frac{p}{q}+1)+M_{2}$. This is since
$$y_{n}=\frac{p + \delta_{1}x_{n-1} + \delta_{2}x_{n-2}}{q + D_{2}x_{n-2} + E_{1}y_{n-1}}\leq\frac{p}{q} + \frac{\delta_{1}x_{n-1} + \delta_{2}x_{n-2}}{D_{2}x_{n-2} + E_{1}y_{n-1}}\leq$$
$$\frac{p}{q} + \left(\frac{\max(\delta_{1},\delta_{2})}{\min(D_{2},E_{1})}\right)\left(\frac{x_{n-1} + x_{n-2}}{x_{n-2} + y_{n-1}}\right)\leq\frac{p}{q} + \left(\frac{\max(\delta_{1},\delta_{2})}{\min(D_{2},E_{1})}\right)\left(\frac{\max(B_{2},C_{1})}{\min(\beta_{1},\beta_{2})}\right)x_{n}.$$
Let $b=\left(\frac{\max(\delta_{1},\delta_{2})}{\min(D_{2},E_{1})}\right)\left(\frac{\max(B_{2},C_{1})}{\min(\beta_{1},\beta_{2})}\right)$.
To deduce $M_{1}$ and $M_{2}$ we show that there exists $L$ so that $y_{n} \geq L$ for all $n\in\mathbb{N}$. So,
$$y_{n}=\frac{p + \delta_{1}x_{n-1} + \delta_{2}x_{n-2}}{q + D_{2}x_{n-2} + E_{1}y_{n-1}}\geq \frac{p + \delta_{1}x_{n-1} + \delta_{2}x_{n-2}}{q + D_{2}x_{n-2} + E_{1}\left(\frac{p}{q} + bx_{n-1}\right)}\geq \frac{\min(p,\delta_{1},\delta_{2})}{\max(q+E_{1}\frac{p}{q},D_{2},bE_{1})}.$$
Hence, $L=\min\left(\frac{\min(p,\delta_{1},\delta_{2})}{\max(q+E_{1}\frac{p}{q},D_{2},bE_{1})},y_{1}\right)$.
Now we show that we may choose \newline $M_{1}=\max\left(\left(\frac{\max(1,\beta_{1},\beta_{2})}{\min(\frac{C_{1}L}{2},B_{2},\frac{C_{1}}{2})}\right)\left(\frac{\max(q,D_{2},E_{1})}{\min(p,\delta_{1},\delta_{2})}\right),\frac{x_{n}}{y_{n}}\right)$ and $M_{2}=1$. So,
$$x_{n}=\frac{\beta_{1}x_{n-1} + \beta_{2}x_{n-2}}{B_{2}x_{n-2} + C_{1}y_{n-1}}\leq
\frac{1+\beta_{1}x_{n-1} + \beta_{2}x_{n-2}}{\frac{C_{1}L}{2}+B_{2}x_{n-2}+\frac{C_{1}}{2}y_{n-1}}$$
$$\leq\left(\frac{\max(1,\beta_{1},\beta_{2})}{\min(\frac{C_{1}L}{2},B_{2},\frac{C_{1}}{2})}\right)\left(\frac{1+x_{n-1}+x_{n-2}}{1+x_{n-2}+y_{n-1}}\right)\leq$$
$$\left(\frac{\max(1,\beta_{1},\beta_{2})}{\min(\frac{C_{1}L}{2},B_{2},\frac{C_{1}}{2})}\right)\left(\frac{\max(q,D_{2},E_{1})}{\min(p,\delta_{1},\delta_{2})}\right)y_{n}+1.$$
The conditions of case $(ii)$ are satisfied since $\{1,2\}=I_{D}\cup I_{E} = I_{\delta}\cup I_{\epsilon}=\{1,2\}$, $p>0$ and $I_{C}\neq\emptyset$.
For the final condition, let $\eta=1$ so that for any sequence $\{c_{m}\} ^{\infty}_{m=1} $ with $c_m\in I_{\beta}\cup I_{\gamma}=\{1,2\}$ for $m=1,2, \dots $ we choose
$N_{1}=1$, $N_{2}=1 \leq \eta$ so that
$\sum^{N_{2}}_{m=N_{1}} c_{m}=c_{1} \in I_{B}\cup I_{C} = \{1,2\}$, since $c_{1}\in \{1,2\}$.

\end{proof}

\section{Conclusion}
We have presented numerous techniques which apply the method of iteration to systems of rational difference equations. These techniques provide a starting point for the immense task of understanding the boundedness character of systems of rational difference equations of order greater than one.
There are three directions for further work which we feel have promise. One important goal is to provide some type of comprehensive criterion which, when satisfied, guarantees the success of the boundedness by iteration technique.
Theorem $6$ in \cite{c1} provided this type of criterion for rational difference equations of order greater than one. We feel that a similar approach is required for systems of rational difference equations of order greater than one. Another important direction for further work is to apply similar techniques to those presented above for systems of three or more rational difference equations. Analogues to the ideas in \cite{knopfhuang} for systems of rational difference equations would be another direction of interest. See \cite{syspln}, \cite{ladassystems1}, and \cite{lugopal} for further work on systems of rational difference equations.

\vfill
\pagebreak
\par

\vspace{0.5 cm}
\begin{thebibliography}{99}
\bibitem{syspln} E. Camouzis, M.R.S. Kulenovi\'c, G. Ladas, and O. Merino, Rational systems in the plane,\;\emph{J. Difference Equa. Appl.}\;\textbf{15}(2009), 303-323.
\vspace{0.5 cm}
\bibitem{ladassystems1} E. Camouzis and G. Ladas, Global results on rational systems in the plane, Part 1,\;\emph{J. Difference Equa. Appl.}\;\textbf{15}(2009), in press.
\vspace{0.5 cm}
\bibitem{book} E. Camouzis and G. Ladas, \emph{Dynamics of Third-Order Rational
Difference Equations with Open Problems and Conjectures}, Chapman \& Hall/CRC
Press, Boca Raton, 2007.
\vspace{0.5 cm}
\bibitem{c1} E. Camouzis, G. Ladas, F. Palladino, and E.P. Quinn, On the Boundedness Character of Rational Equations, Part 1,\;\emph{J. Difference Equa. Appl.}\;\textbf{12}(2006), 503-523.
\vspace{0.5 cm}
\bibitem{knopfhuang} P.M. Knopf and Y.S. Huang, On the boundedness character of some rational difference equations,\;\emph{J. Difference Equa. Appl.}\;\textbf{14}(2008), 769-777.
\vspace{0.5 cm}
\bibitem{lugopal} G. Lugo and F.J. Palladino, Unboundedness results for systems,\;\emph{Cent. Eur. J Math.}\; in press.
\vspace{0.5 cm}
\bibitem{fpinv}  F.J. Palladino, Difference inequalities, comparison tests, and
some consequences,\;\emph{Involve}\;\textbf{1}(2008), 91-100.
\vspace{0.5 cm}
\end{thebibliography}
\end{document}